\documentclass[12pt]{amsart}
\usepackage{amsthm}
\usepackage{mathtools,amssymb,latexsym,graphics,enumerate}
\usepackage[mathscr]{eucal}
\usepackage{amsmath,amsfonts,amsthm,amssymb}
\usepackage{color}
\usepackage{xcolor}
\usepackage{hyperref}
\usepackage{caption}
\usepackage{subcaption}
\usepackage{cleveref}
\usepackage[left=1in,right=1in,top=1in,bottom=1in]{geometry}

\usepackage{graphicx}
\numberwithin{equation}{section}

\newcommand{\rr}{\mathbb{R}}
\newcommand{\E}{\mathbb{E}}
\newcommand{\lan}{\langle}
\newcommand{\ran}{\rangle}
\newcommand{\be}{\begin{eqnarray*}}
\newcommand{\bel}{\begin{eqnarray}}
\newcommand{\ee}{\end{eqnarray*}}
\newcommand{\eel}{\end{eqnarray}}
\newcommand{\ba}{\begin{aligned}}
\newcommand{\ea}{\end{aligned}}
\newcommand{\de}{\Delta}

\newcommand{\na}{\nabla}

\newcommand{\pa}{\partial}
\newcommand{\wh}{\widehat}

\newcommand{\LL}{{\mathbb{L}}}
\newcommand{\cL}{{\mathcal{L}}}
\newcommand{\nb}{\nonumber}
\newcommand{\TT}{{\mathcal{T}}}

\newtheorem{theorem}{Theorem}

\newtheorem{lem}{Lemma}
\newtheorem{pro}{Proposition}
\newtheorem{remark}{Remark}

\newcommand{\norm}[1]{\left\lVert#1\right\rVert}

\newcommand\Torus{{\mathbb T}}
\newcommand\Real{{\mathbb R}}
\newcommand{\Xt}{\mathbf{X}_t}
\newcommand{\Bt}{\mathbf{B}_t}

\title[Random search in fluid flow aided by chemotaxis]{Random search in fluid flow aided by chemotaxis}
\author{Yishu Gong} \thanks{yishu.gong@duke.edu, Department of Mathematics, Duke University}
\author{Siming He} \thanks{simhe@math.duke.edu, Department of Mathematics, Duke University}
\author{Alexander Kiselev}\thanks{kiselev@math.duke.edu, Department of Mathematics, Duke University}
\date{\today}

\begin{document}
\begin{abstract}

In this paper, we consider the dynamics of a 2D target-searching agent performing Brownian motion under the influence of fluid shear flow and chemical attraction.
The analysis is motivated by numerous situations in biology where these effects are present, such as broadcast spawning of marine animals and other
reproduction processes or workings of the immune systems. We rigorously characterize the limit of the expected hit time in the large flow amplitude limit
as corresponding to the effective one-dimensional problem. We also perform numerical computations to characterize the finer properties of the expected duration of the search.
The numerical experiments show many interesting features of the process, and in particular existence of the optimal value of the shear flow that
minimizes the expected target hit time and outperforms the large flow limit.
\end{abstract}

%\pagecolor{brown}
%\color{white}
\maketitle
%\tableofcontents

\section{Introduction}

In this paper, we analyze the agent randomly searching for a target in the ambient shear flow, aided by chemotaxis on the chemical released
by the target. This process occurs in multiple settings in biology. One example is reproduction for many species, where eggs
secrete chemicals that attracts sperm and help improve fertilization rates. This is especially
well studied for marine life such as corals, sea urchins, mollusks, etc (see \cite{HRZZ,Riffelletall04,RiffellZimmer07,ZimmerRiffell11} for
further references), but the role of chemotaxis in fertilization extends to a great number of
species, including humans \cite{Raltetal}. Another process where chemotaxis plays an important
role is mammal immune systems fighting bacterial infections. Inflamed tissues release special
proteins, called chemokines, that serve to chemically attract monocytes, blood killer cells, to
the source of infection \cite{Desh,Taub}. Chemotaxis can also be involved when things go awry, for
instance, playing a role in tumor growth \cite{VanCoil}. One can also envision future applications to medical mini-robots
tasked with finding some sort of targets.
These processes take place in fluids, and on the
length scales where the ambient fluid motion can be effectively regarded as shear flow.

As a mathematical model, we consider the following stochastic differential equation (SDE) subject to initial condition $\mathbf{X}_{0}=(x_0,y_0)$ on the torus $\mathbb{T}^2=[0,L)^2:$
\begin{subequations}\label{SDE_full}
\begin{align}
\left(\begin{array}{rr}dX_t\\dY_t\end{array}\right)=&\left(\begin{array}{cc} {A} u(Y_t)\\0\end{array}\right)dt+\left(\begin{array}{c}V^{(1)}\\V^{(2)}\end{array}\right)dt+\sqrt{ 2\nu }\left(\begin{array}{rr} dB_t^{(1)}\\ dB_t^{(2)}\end{array}\right);\label{EQ:SDE}\\
V=&(V^{(1)},V^{(2)})=\varphi( {\chi} |\nabla c|)\frac{\na c}{|\nabla c|},\quad-\de c+ {A}u(y)\pa_x c=n -c.\label{EQ:chemical}
\end{align}%\end{align}\begin{align}
\end{subequations}
The terms on the right hand side of \eqref{EQ:SDE} model advection, chemotaxis, and random motion respectively.
The function $c$ is the concentration of the chemical, and $n$ is the density of the target, that will be located in a small area at the center of the torus.
%The form of the equation determining $c$ from $n$ in \eqref{EQ:chemical} assumes that the chemical diffusive time scale is much faster than
%other relevant processes, leading to the steady state distribution satisfying elliptic equation.
For chemotaxis, we choose a variant of flux-limited models that have been studied in, for example, \cite{CKWW,HP,HPS,PVW}.
and which is more realistic than the classical Keller-Segel form in that it places a speed limit on the agent.
%The function $\varphi$ is bounded and will be assumed
%to be close to piecewise linear, allowing its parametrization with essentially two parameters: the maximum speed and chemical sensitivity.

The SDE \eqref{EQ:SDE} models a single agent's searching process subject to ambient fluid advection, random Brownian motion, and chemical attraction. The searching is successful if the agent reaches the region occupied by the target population $n$.
The agent, positioned at point $\Xt=(X_t, Y_t)$, is transported by the ambient  shear flow $({A}u(Y_t),0)$ with magnitude ${A}\in \rr_+$.
Meanwhile, the agent also moves randomly, which is captured by the Brownian motion $\sqrt{\nu}d\Bt=\sqrt{\nu}(dB_t^{(1)}, dB_t^{(2)})$ with diffusivity $\nu$.
Finally, the agent aggregates towards the higher concentration of chemoattractant density $c$, secreted by the target population $n$.
The parameter ${\chi}$ denotes the chemical sensitivity, and the smooth cut-off function $\varphi\in C^2(\rr_+)$ enforces finite speed of aggregation.
In fact, we will take the function $\varphi$ to be close to piecewise linear, allowing its parametrization with essentially two parameters: the maximum speed and chemical sensitivity.
We consider the regime where the chemoattractant density $c$ reaches equilibrium at a much faster time scale than other relevant time scales of the problem, thus
 the density $c$ satisfies the elliptic equation \eqref{EQ:chemical}.    %Throughout the paper, the diffusivity $\nu$ and the chemical sensitivity $\chi$ are $\mathcal{O}(1)$ constants. %Poisson

\ifx\footnote{\textcolor{red}{To get $\nu\de$ in 2-dimension, the Brownian motion should be $ \sqrt{{2}\nu}(dB_t^1,dB_t^2)^T$. Check the paper %https://arxiv.org/pdf/2008.11710.pdf
}}
\fi

Our main goal in this paper is to gain insight into the interaction of the three transport mechanisms present in the model and their cumulative effect on the expected length of the search.
There are many works dedicated to interaction of advection and diffusion - see e.g. \cite{FreidlinWentzell} and \cite{IyerNovikovRyzhikZlatos10}; the latter source specifically looks at diffusion exit times and
contains further references. However we are not aware of any detailed mathematical - rigorous or numerical - analysis of the problem when chemotaxis is added in the mix.
Our work has been largely inspired by biological experiments on broadcast spawning of abalone conducted by Riffell and Zimmer.
Marine animals such as abalones, corals, and shrimp release their egg and sperm cells into the ambient ocean. The gametes are positively buoyant and rise to the surface, where fertilization happens.
The eggs are not mobile but release attractive pheromone. The sperms aggregate towards the eggs by a combination of random motion and chemotaxis-guided transport. Since the processes occurs in a fluid flow that is effectively shear on length scales involved, it is of biological interest to study the relation of fertilization success rate, chemotaxis, and shear flow speed. In the papers \cite{Riffelletall04}, \cite{RiffellZimmer07}, \cite{ZimmerRiffell11}, Zimmer, Riffell and their research group put well-mixed abalone sperms and eggs in a Taylor-Couette tank and study the quantitative relationships. The positive effect of chemotaxis has been clearly established; as far as the shear flow, the researchers observed that its effect is two-fold. If the shear rate is moderately slow, the shear enhanced fertilization.
On the other hand, if the shear rate is faster than a certain threshold, the fertilization rate starts declining.
We notice that the sperms are evenly distributed in the seawater in the biological experiment, and the experimental time is limited ($\approx 15$ seconds).
Hence only the group of sperms surrounding the eggs have access to the egg zone. As a result, the microflow environment play a dominant role during the fertilization process under this setup.
Our model addresses a related but different situation where there is a single searcher and, instead of the fertilization success rate (percentage of fertilized eggs), it monitors expected search time.
We consider this set up since we would like to represent the problem in the most fundamental form, and to understand the role of different forces affecting the search on this fundamental level.
Nevertheless, in our numerical computations, we mostly focus on the parameter regimes relevant for the experiments of Riffell and Zimmer.

Our main results are as follows.
On the rigorous level, we are able to establish that the very large shear rates are a dimension reduction mechanism: the expected search time converges to the one of the corresponding
one-dimensional problem with effective 1D chemotaxis. Such result is not unexpected, and is similar to the findings of Freidlin-Wentzel theory \cite{FreidlinWentzell}.
The presence of chemotaxis, however, necessitates some novel elements in the analysis. Numerically, we observe several phenomena that we find interesting.
First, we see quite fast decrease of the expected hitting time already for quite small values of shear and chemotactic coupling.
Second, we discover that in the context of our model,
there is an optimal shear rate range where the searching time is minimal, and is {\it less} than the limiting 1D large shear searching time.
This is entirely due to presence of chemotaxis; when it is absent, the expected hitting time is monotone decaying in flow amplitude.
This finding agrees with the results of biological
experiments. The difference is that in the experiments the fertilization rate declines much more steeply (and probably towards zero) for large shears. This effect is likely not because sperm have trouble finding eggs,
but rather due to their inability to stay around for time necessary for fertilization. In our set up, there is no such mechanism: we just register the hit. Nevertheless,
even on the fundamental level of random search/chemotaxis/fluid flow interaction, we observe the non-trivial phenomena of optimal shear range. The third interesting effect we find
in the framework of our model and within the range of parameters that we tested is that increasing chemical sensitivity appears meaningfully more important for improving search performance
than increasing the maximal speed. We believe that these observations can be useful for better understanding of biological processes that may involve additional elements and
factors, but include the interaction of the three basic forces that we study here.
We note that numerical experiments on random search in a shear flow (without chemotaxis) were also carried out in \cite{CGHKLMP}.

The paper is organized as follows: in the next section, we introduce the general set up and key parameters of the model in more detail and present our numerical scheme.
We then proceed to describe the results of the numerical experiments. After this we state the rigorous results that we are able to prove, and proceed with the proofs.
%This way both the readers that are interested
%primarily in more nuanced numerical results and in rigorous details should be able to navigate the paper easily.

% To conclude, the experiment explores the relationship between fertilization success rate and shear flow speed on a  small scale.

%Even though the sperms flow past the eggs, they cannot attach to them—the fail in attachment results in a dropping fertilization rate.

\section{General Set Up and Numerical Scheme}

\subsection{The set up and first hitting time}
Recall that in this paper we focus on the searching success of an \emph{individual agent}, whose initial position may be \emph{relatively distant} from the target zone.
We focus on the following geometric configuration in our numerical and analytic exploration. The domain $(L\Torus)^2$ has dimension $[0,L)^2$.
The chemical cutoff function $\varphi(\cdot)\in C^\infty([0,\infty))$ is chosen so that $\varphi(0)=0$ and $\varphi'(0)=1$, $\varphi$ is monotone increasing, and $\lim_{r\rightarrow \infty}\varphi(r)= ||\varphi||_\infty<\infty $ (see Figure \ref{fig:phi}).
The norm $||\varphi||_\infty$ has the meaning of the maximal chemical-induced speed of the agent.
%Even though there is a component of the agent's speed that goes into diffusion, throughout the paper
%we will call $||\varphi||_\infty$ \it the maximal agent speed. \rm
The target density $n$ is stationary and concentrated in the target zone $E_\delta$, which is a disk $B((\frac{L}{2},\frac{L}{2});\delta)$. The size of the target $\delta$ is much smaller than $L$.
The total mass of the target density is $||n||_{L^1(\Torus^2)}$ is normalized to be equal to one.
The searching starts at a point $(x_0,y_0)$. In our numerics, the agent starts at point $(0,0)$, which is a distance $\frac{L}{\sqrt{2}}$ away from the target zone. 
\textcolor{black}{ We tried other starting positions with very similar qualitative results.
 The shear profile is adapted to the size of the  torus and is given by $u(y)=\sin\left(\frac{2\pi (y-L/2)}{L}\right),$ with coupling constant $A$. 
We simulated a few other shear flow profiles, but did not observe a significant difference as long as the shear rate near the egg zone is the same. Therefore, we restrict ourselves to the $A\sin\left(\frac{2\pi (y-L/2)}{L}\right)$-flow for the sake of simplicity. 
 In Figure \ref{Fig:Problem_setup}, we provide a diagram with the general setup. }

\begin{figure}[h]
\centering
\includegraphics[scale=0.75]{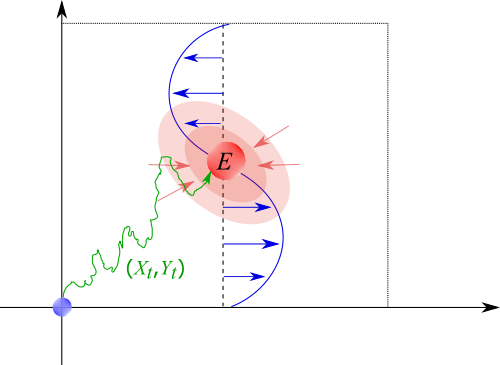}
\caption{Problem setup.}\label{Fig:Problem_setup}
\end{figure}

\begin{figure}[htb!]
    \centering
    \includegraphics[width=0.6\linewidth]{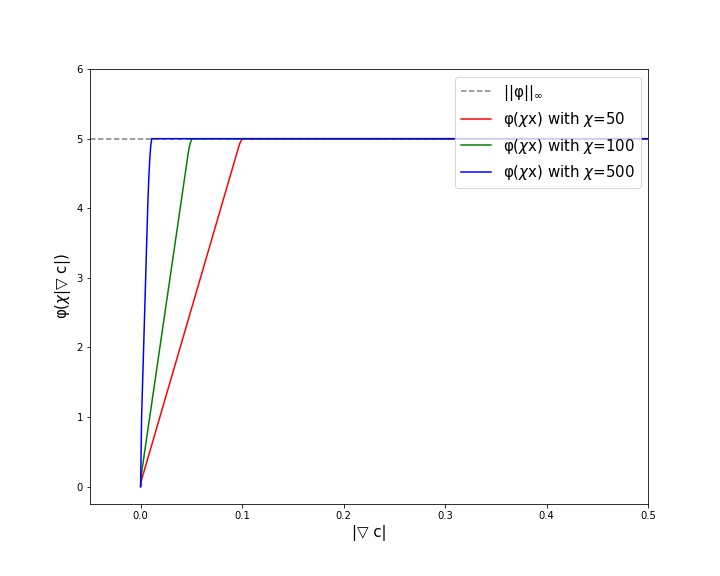}
    \caption{Cutoff function $\varphi (x)$ for chemotaxis.}
    \label{fig:phi}
\end{figure}

To quantify the success of the search, we consider the expectation of the first hitting time of the target zone, i.e.
\begin{align}\label{First_hitting_time}
T^{(x_0,y_0)  }(\omega)=\min\{\tau|\mathbf{X}_\tau(\omega)\in E_\delta=B((L/2,L/2);\delta)\}.
\end{align}
Here $\textcolor{black}{\bf X_\tau(\omega)}$ is the realization of the solution to the stochastic differential equation \eqref{SDE_full}. The expected hitting time depends on the various parameters involved in the system, specifically the size of the torus $L$, shear amplitude $A$ (or shear rate $A/L$),
diffusivity $\nu,$ target size $\delta,$ chemical sensitivity $\chi,$ and maximal chemotactic speed $\|\varphi\|_\infty.$ Observe that to simplify the presentation, we can normalize any two parameters by rescaling space and time. We choose to normalize $\delta=1,$ which means that we measure everything
in units of target size. This size is about $0.1$mm for biological experiments of Riffell and Zimmer, and we annotate our numerical plots accordingly. For different interpretations, one just needs to remember that the target size (or, to be precise in our setting, its radius) is the
unit of length we use.
In the numerical simulations, we will also rescale time to set the value of diffusivity at $0.25.$
 With these changes, we have four parameters remaining in our model that in the terms of the original ones can be expressed as $\tilde L= L/\delta,$ $\tilde A=A \delta /4\nu,$ $\|\tilde \varphi\|_\infty = \|\varphi\|_\infty \delta/4\nu,$ and $\tilde \chi = \chi/\delta.$
 In the rest of the paper, we will abuse notation and omit tilde for the remaining four renormalized parameters. Finally we remark that it is not unreasonable to also introduce two more parameters in front of $c$ and $n$ on the right hand side of the equation \eqref{EQ:chemical} for the attractive
 chemical. However in this paper we do not pursue the analysis of expected hitting time dependence on these parameters.
%We study rigorously the asymptotic regime $A \rightarrow \infty$ and implement numerical experiments to characterize the finer relative importance of each factor.

We note that the first hitting/exit time of the Brownian motion is a classical topic in stochastic analysis. It has a close relation to analysis of elliptic equations - see e.g. well known treatises \cite{Durrett96} and \cite{Oksendal}.
We are going to recall this connection below in the rigorous analysis sections.

{\color{black} Another interesting question that we do not address in this paper is concerned with the fastest or a small group of fastest searchers out of many rather than with the average search time. Indeed, in some settings such as fertilization this may be
the more relevant question, even though in other situations like immune response, average time might be more important as many agents are needed to perform the biological function. We point out that extreme first passage statistics have been analyzed in \cite{Lawley201}
in the limit of the very large number of agents. In \cite{MeersonRedner15}, a similar problem was considered for one-dimensional diffusion and mortal searchers. A review \cite{SchussBasnayakeHolcman19} discusses a variety of settings in biology where the extreme
statistics are relevant. These papers contain further references on the role of redundancy through multiple agents in biological functions. While we performed some simulations to analyze the shortest search time from a collection of agents, we found the
outcomes to be very unstable, at least at the number of simulations that we were able to run. It is thus difficult to compute error bounds on such results as there is no convergence to a fixed value with the increase in the number of agents - rather, one has to explore the entire
probability distribution function. We leave this very interesting question to future work.  }

%\textcolor{black}{It is worth mentioning that there is a significant amount of biological/mathematical research devoted to the multi-agent random searching dynamics. The `redundancy' of searching agents is believed to enhance the net speed of the fertilization and various other biological processes, see, %e.g., \cite{MeersonRedner15,SchussBasnayakeHolcman19}. The relevant concept in the multi-agent searching scenario is the extreme first hitting time. A rigorous discussion of the asymptotic behavior of the average and moments of the extreme first hitting time can be found in
%\cite{Lawley201,Lawley202} and the reference therein. It will be a fascinating question to consider various regimes involving large shear and large crowd limits of the extreme first hitting time on the torus. We will leave it to a later paper.  }

Throughout the paper, the $C$'s denote various constants that do not depend on the key parameters and their value may change from line to line.

\subsection{The numerical scheme}
In the numerical experiment, we run multiple simulations to calculate the average hitting time. The well known Euler-Maruyama method  (see, e.g., \cite{KloedenPearson77}) is applied to simulate the motion of the agent.

\begin{figure}[htb!]
    \centering
    \includegraphics[width=0.5\linewidth]{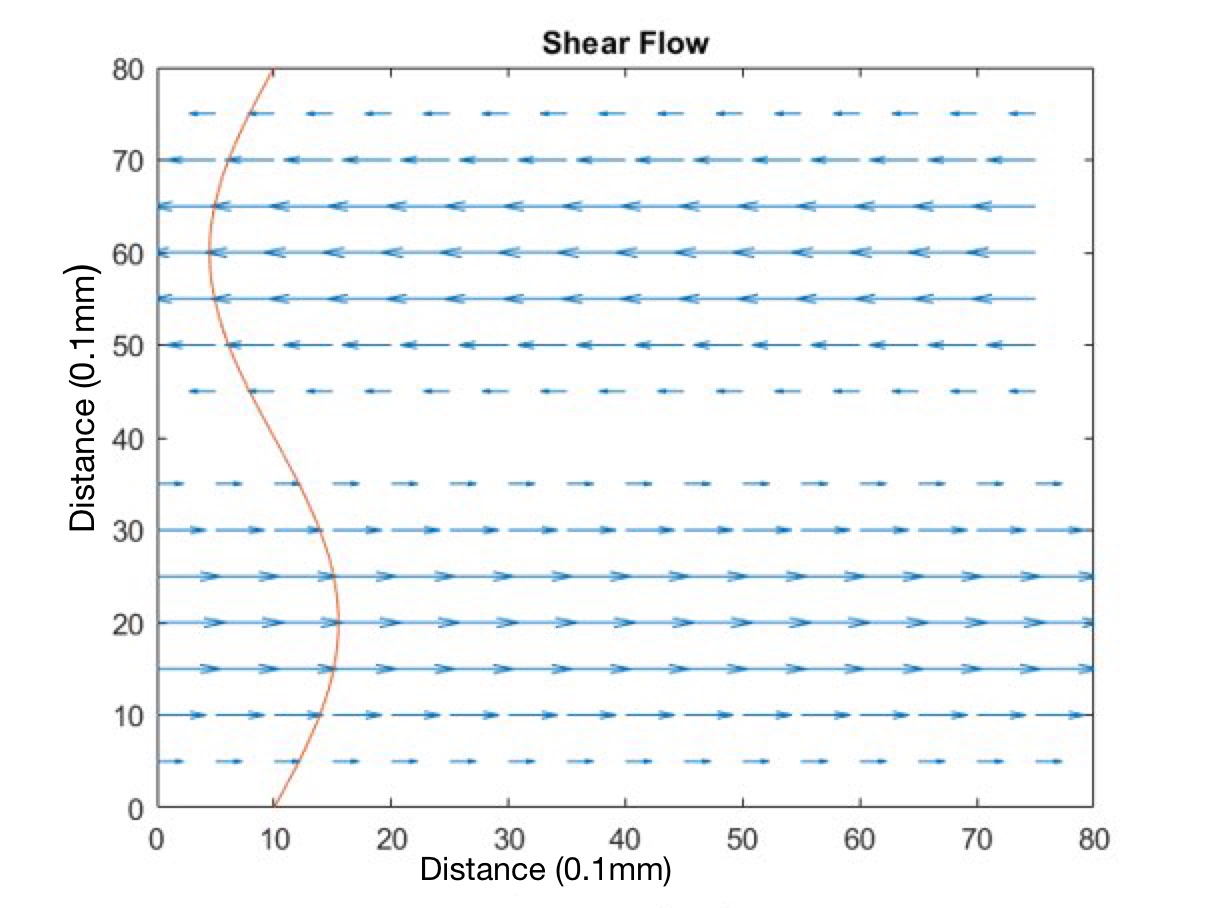}
    \caption{Shear strength multiplied by $dt = 0.01s$.}
    \label{fig:shear}
\end{figure}

There are two less standard aspects in the simulation of the system \eqref{SDE_full}. The first aspect is approximating the shear flow advection, and the second is calculating the aggregation towards the chemoattractant.

For the first issue, we take $u(y) = \sin (2\pi (y-L/2)/L);$ recall that the target is located at $y=L/2,$ so the fluid flow can be thought of as taken relative to the target (Figure \ref{fig:shear}). For very large amplitudes of $A$ (specifically we took $800$ as threshold in the simulation), we saturate the value of the shear replacing $Au(Y_t)$ with $800 \frac{Au(Y_t)}{|Au(Y_t)|}$ if $|Au(Y_t)|>800.$ We do this in order not to take the time step excessively small. We ran several simulations to check that this cutoff procedure does not affect the expected hit time.
From our experiments, it appears that only the structure of $u$ near the target (basically just the shear rate) meaningfully affects the result. %, and the cutoff only applies some distance away from the target.

Next, we calculate the chemoattractant distribution $c$, which in turn determines the aggregation. The standard finite difference method is applied. We use the five-point stencil method to approximate the Laplacian and the central difference method to discretize the advection $Au(y)\partial_x c$. {\color{black}
By inverting the linear system corresponding to the discretized elliptic PDE satisfied by the chemical density, we obtain the numerical value of $c$ on the grid points. The numerical chemical gradient can be obtained through standard finite difference. However, interpolation is needed to determine the chemical gradient on the point away from the grid point. The explicit scheme is as follows. For a fixed position $(x_0,y_0)$, we first identify $36(= 6 \times 6)$ grid points in its neighborhood. Then we apply the cubic spline interpolation to determine the gradient value at the point $(x_0,y_0)$.}
Finally,  the Euler-Maruyama method is applied to simulate the SDE \eqref{EQ:SDE}. 
%It is rare for the particle to land on the grid points on which we have the explicit chemoattractant density. We use the cubic spline interpolation with nearby points to obtain the chemoattractant density profile near the particle's current location. Then we numerically calculate chemical gradient $\nabla c$ at the particle's location using centered differences.

In Figure \ref{Fig:chemo_concentration}, we illustrate that as the shear rate increases, the distribution of the chemical, as expected, gets stretched in the horizontal direction. In Figure \ref{Fig:Chemical_gradient}, we plot the chemical gradient vector.

\begin{figure}[hbt!]\label{Fig:chemotaxis}
    \centering
    \begin{subfigure}[b]{0.5\textwidth}
        \centering
        \includegraphics[width=0.9\linewidth]{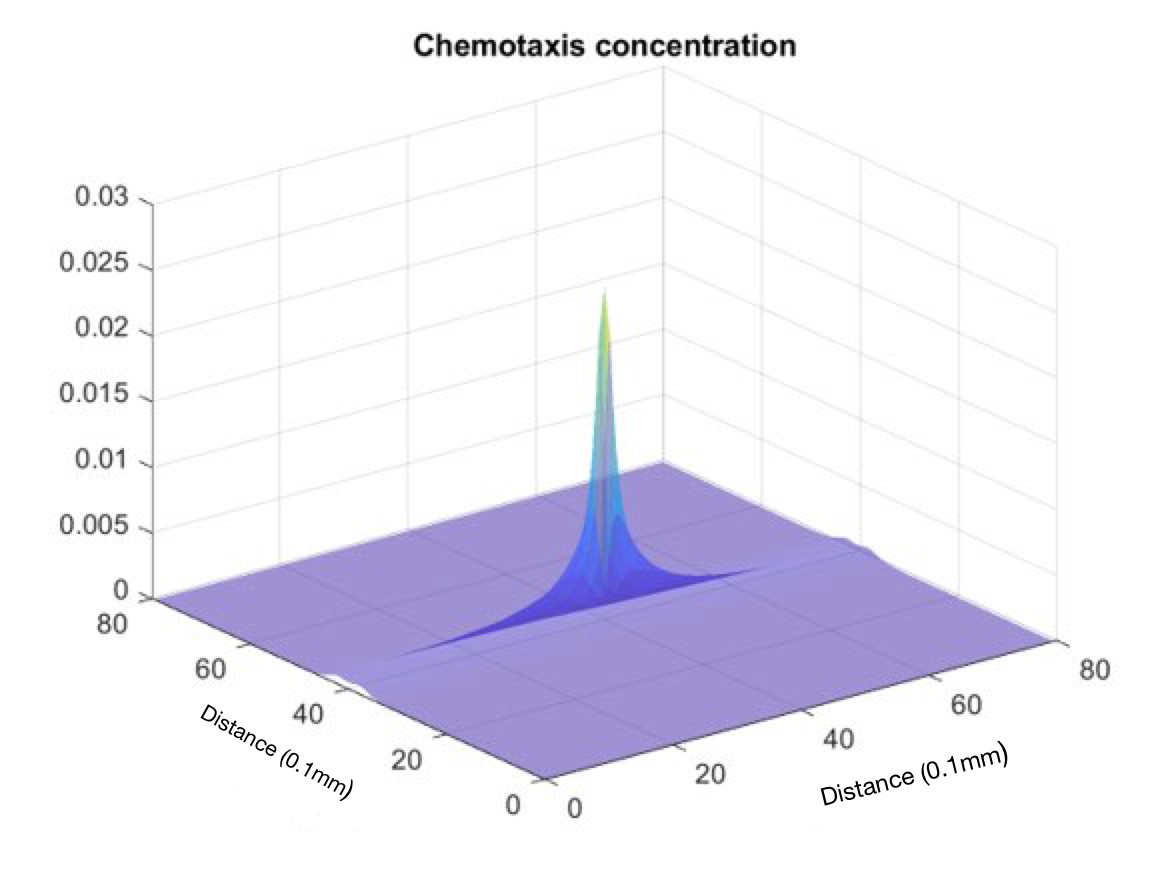}
        \caption{Chemical concentration $c$}
        \label{Fig:chemo_concentration}
    \end{subfigure}%
    ~
    \begin{subfigure}[b]{0.5\textwidth}
        \centering
        \includegraphics[width=0.9\linewidth]{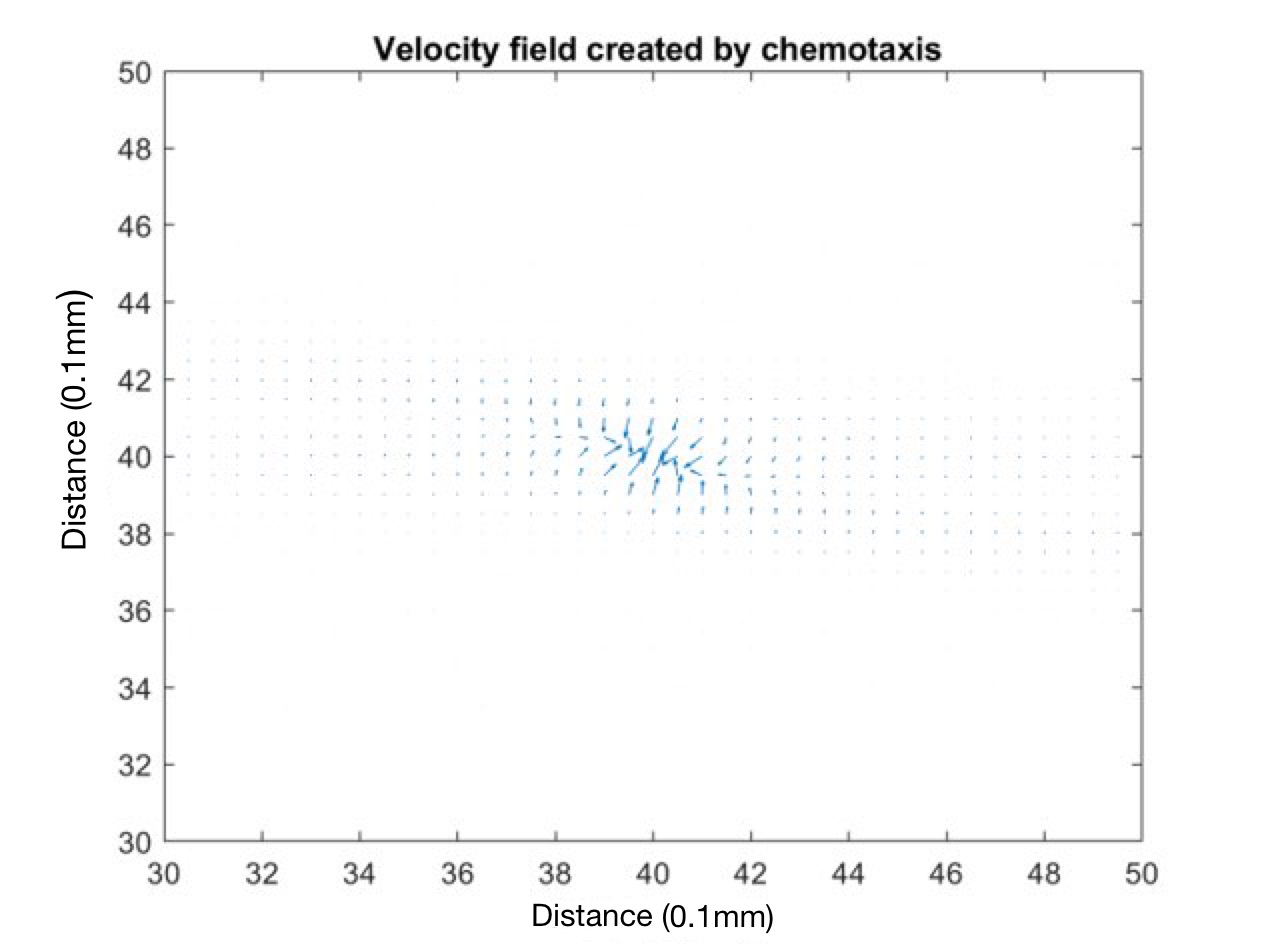}
        \caption{Chemical gradient $\nabla c$}
        \label{Fig:Chemical_gradient}
    \end{subfigure}
    \caption{Chemotaxis concentration and gradient under the effect of shear flow. Boxsize is $80 \ (0.1mm)$ and shear amplitude cutoff is set to $200 \ (0.1 mm) / s$. Based on $2000$ simulations.}
\end{figure}

\section{Numerical Results}
\subsection{Ranges of parameters}
We try to roughly align the ranges of our parameters with the corresponding ranges in the biological experiments of Zimmer and Riffell \cite{ZimmerRiffell11} - at least where these biological parameters are known or can be
roughly estimated. The shear rate in our simulation is defined as the ratio $2\pi A/L$. %The ranges of the parameters are largely inspired by the biological experiments
The typical abalone egg diameter is $0.2mm$. The Taylor-Couette tank used in the experiment has distance of about $8$mm between concentric cylinders,
and the shear rates tested range between $0$ and $10{\rm s}^{-1}.$ Our simulation covers these ranges of parameters and more.
The maximal chemotactic speed $\|\varphi\|_\infty$ is limited by the sperm speed ability, which is $v_0 \sim 0.05$mm/s.
One parameter that is not immediate to estimate is the effective diffusion coefficient.
{\color{black} There is a large number of results in the literature that rigorously deduce effective diffusion-type equations from the underlying velocity-jump processes
describing individual agents - see for example \cite{HO1,OH1} for such derivations in the biological context. However, a limiting asymptotic assumption is necessarily involved in such a transition.
We are not aware of the studies on sperm motion that would support a particular model for the change in direction.
For this reason we adapt a very simplistic heuristic estimate of the diffusion coefficient outlined below.   }
In the absence of chemical stimuli, sperm appear to move in some direction for a while,
then change the direction randomly.  Let $t$ be the time that sperm maintains direction.
Assuming that sperm maintain speed comparable to maximal, over the larger time $T=nt,$ the displacement $D_T$ is given by $D_T = \sum_{i=1}^n X_i,$ where $X_i$ are independent 2D random variables with amplitude
$v_0t$ and random direction uniformly distributed over $[0,2\pi).$ We can estimate the expected displacement by
\[ \E(D_T^2) = \E \left(\sum_{i=1}^n X_i \cdot \sum_{j=1}^n X_j \right) = n \E(X_i^2) = n v_0^2 t^2 = T v_0^2 t. \]
Now for a 2D Brownian motion, we have $\E(B_T^2) = 4 \nu T,$ where $\nu$ is diffusion coefficient. Thus in two dimensions it is reasonable to adopt an estimate
$\nu = \frac14 v_0^2 t.$ The only parameter that we do not have readily available is $t$, but looking at the trajectories of sperm motion provided in \cite{ZimmerRiffell11},
taking $t \sim 1$s appears to be reasonable. This leads to the estimate $\nu \approx 0.06(0.1 {\rm mm})^2/{\rm s}.$
In the numerical simulation, we choose not to have a very small diffusion coefficient and pick $\nu_{\mathrm{num}} =0.25,$ which is equivalent to changing numerical unit time from one second to
roughly $4$ second units. Two parameters in the Table \ref{Table_1} that get affected are the shear rate and maximal chemotactic speed, that in the new units range $0-600(4s)^{-1}$ and $0-5(0.1mm)/(4s)$
(respectively $0-150s^{-1}$ and $0-1.25(0.1mm)/s$ in natural units) in our simulation.
One parameter that we cannot estimate from the biological experiment is the chemical sensitivity $\chi.$
Although the parameter ranges are coordinated with the experiment \cite{ZimmerRiffell11}, we find it likely that they are relevant in a wider range of biological applications.
While we use $4s$ time units in parameter table, on our plots we make an adjustment to more natural time units of just seconds. We do keep the $0.1mm$ length unit on the plots since this unit is the
intrinsic target size parameter rather than a numerical artifact.

%In Table \ref{Table_1}, one can find the ranges of parameters used in the simulation compared to biological experiments of Riffell and Zimmer,%in the respective natural units. % (the numerical values differ for two parameters - that is discussed below)\textcolor{black}{The explanation is not clear, please specify. The sentence in the parenthesis is ambiguous??}.
\begin{table}[ht]
\begin{tabular}{lccl}
Parameters   & Value/Range in {Biology} &   Value/Range in Simulation \\ % & Unit\\
Diffusion Coefficient   & $\sim 0.06$ $(0.1 mm)^2/s$    & $0.25$ $(0.1 mm)^2/(4s)$ \\ %& $(0.1 mm)^2/s$ \\
Egg Radius   & $1$ ($0.1 mm$)       & $1$ ($0.1 mm)$             \\
Box Size     & $\sim 80$ ($0.1 mm$)    &  $50-200$ ($0.1 mm$) \\%       &                 \\
Shear Rate   & $0-12$ ($s^{-1}$)    & $0-600$ $(4s)^{-1}$ \\ %  & $s^{-1}$
Chemical Sensitivity & not clear & $ 50-50000(0.1mm)$ \\
%Shear Amplitude Cutoff & -            & $\sim 800 \ (0.1 mm) / (4s)$\\
Maximal chemotactic speed &  $\sim 0.5$ (0.1mm/s) & $  0 - 5 \ (0.1 mm) / (4s)$
\end{tabular}
\caption{Parameters}\label{Table_1}
\end{table}
\subsection{Brownian motion subject to shear flow}
If there is no chemical attraction, i.e., $\chi=0$, the average first hitting time is monotone decreasing in terms of the shear rate  (Figure \ref{Fig:First_hitting_time_shear}).
Moreover, significant decay happens already at the relatively small values of shear rate (note the logarithmic scale of the graph). For large shear rates, the expected hitting time
approaches the hitting time of 1D Brownian motion where the $x$ coordinate is eliminated (drawn as a line on the graph). The explicit formula for this 1D hitting time is well known
and equal to $\frac{1}{\nu}\left(\frac{L}{2}-\delta\right)^2$ (see the argument before \eqref{form} for a sketch).
\begin{figure}[htb!]
\includegraphics[width=0.75\linewidth]{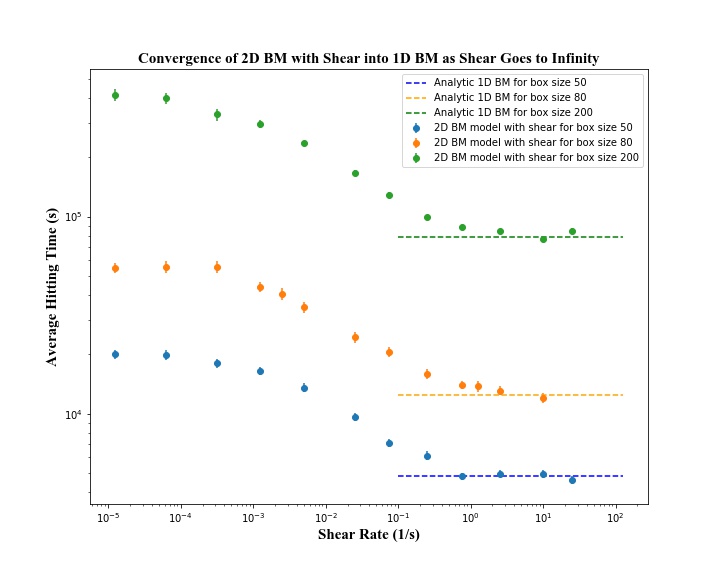}
\centering
\caption{First hitting time with increasing shear rate. Shear amplitude cutoff is set to $200 \ (0.1 mm) / s$. Chemotaxis is not present.}\label{Fig:First_hitting_time_shear}
\end{figure}
We refer to \cite{CGHKLMP} for more detailed information on a similar simulation.

\subsection{Brownian motion subject to shear flow and chemical attraction: optimal shear}

We carry out simulations in tori with three different sizes: $50$, $80$, $200 \ (0.1mm)$. In the Figures \ref{Fig:Fig_50}, \ref{Fig:Fig_80} and \ref{Fig:Fig_200}
we compare the behavior of the expected hitting time dependence on shear amplitude without chemotaxis and with the maximal chemotaxis speed in the range we analyze.
The expected time is computed by averaging over $1000-2000$ simulations, and the vertical bars at each point are two standard deviations of the sample in each direction.
In all these simulations, presence of chemotaxis results in more than double reduction of the expected hitting time even when shear is zero.
Even very small values of shear rate lead to further meaningful reduction of expected hitting time. The optimal value of the shear rate is in all cases
around $0.15-0.3s^{-1}.$ %(note that on plots we use the natural simulation time scale $(4s)^{-1}.$
This indicates that the optimal shear value is not affected by the ambient box size.
At optimal shear rate, the expected hitting time is reduced more than by
another factor of two, and by about a third (less for the largest box) outperforms the one-dimensional large shear rate limit.
The value of this limit is indicated on the plots as a solid line. Note that this is a large shear limit \it without \rm chemotaxis, which is explicitly
computable. Our results in the analytic section rigorously establish that the expected search time converges in the large $A$ limit to 1D problem \it with
\rm effective chemotaxis. However, our numerical simulations suggest that the 1D hitting time of this 1D effective chemotaxis problem is quite close to that of the
1D diffusion without chemotaxis - at least within our ranges of parameters.
\textcolor{black}{The figures \ref{Fig:Fig_50}, \ref{Fig:Fig_50} and \ref{Fig:Fig_50} give an impression that the optimal shear effect becomes less pronounced with increasing box size.
Indeed, the minimal and the large $A$ limit expected time ratios are approximately $3370/4610 \sim 73\%,$ $7880/12170 \sim 65\%,$ and $70400/78400 \sim 90\%$ for the box sizes $50,$ $80,$ and $200.$
One can conjecture that the effect of chemotaxis is relatively short range in the vertical direction, and so for very large box size any possible gain compared to the 1D effective problem
is going to be limited. }

%this is mostly because of the change of vertical axis scale due to larger expected times for small shear rates in bigger box sizes.
%The absolute difference between the optimal average first hitting time and the limiting average first hitting time is 2173.38s for the box of size 50,\, 4760.82s for the box of size 80, and 50599.5s for the box of size 200, respectively. We can see that as the shear rate increases, the absolute %difference is apparent.} (\textcolor{black}{Siming: I think the impression that the ratio is decreasing (or the picture is flattened) is mainly due to the maximum as A goes to zero is enormous in larger boxes. It has nothing to do with the contrast between limiting and optimal behaviors.})

\begin{figure}[hbt!]
     \centering
     \hspace{-3em}
     \begin{subfigure}[b]{0.55\textwidth}
         \centering
         \includegraphics[width=\textwidth]{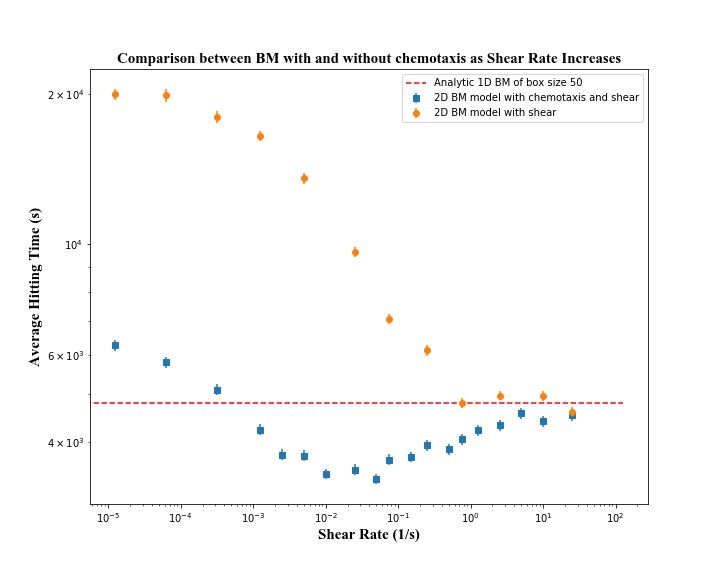}
         \caption{First hitting time with shear flow and \\ chemical attraction: Box size 50 (0.1 mm).}
         \label{Fig:Fig_50}
     \end{subfigure}
     \hspace{-2em}%
     \begin{subfigure}[b]{0.55\textwidth}
         \centering
         \includegraphics[width=\textwidth]{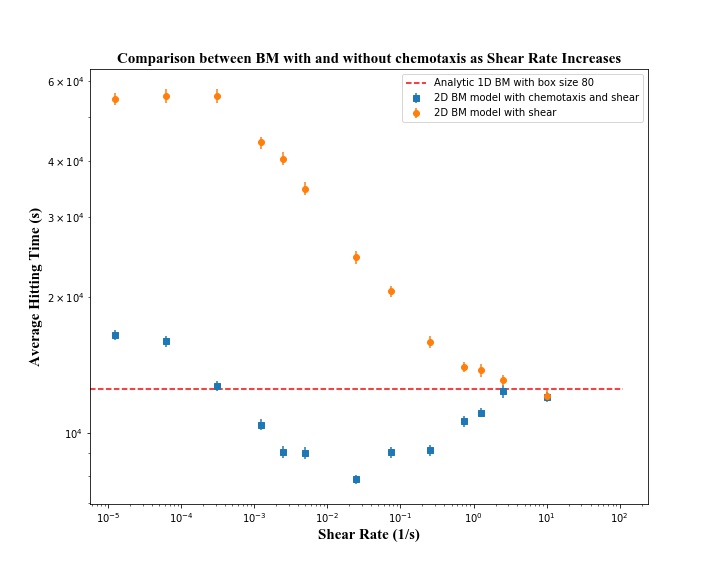}
         \caption{First hitting time with shear flow and \\ chemical attraction: Box size 80 (0.1 mm).}
         \label{Fig:Fig_80}
     \end{subfigure}
      \hspace{-3em}
     \begin{subfigure}[b]{0.55\textwidth}
         \centering
         \includegraphics[width=\textwidth]{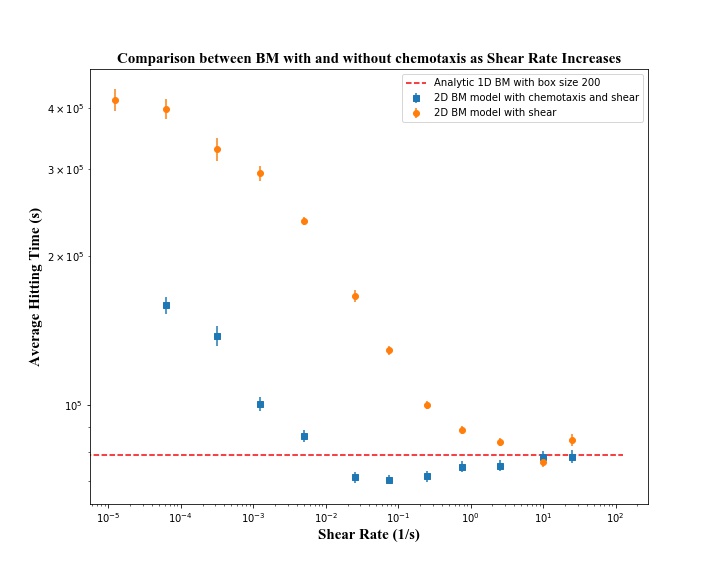}
         \caption{First hitting time with shear flow and \\ chemical attraction: Box size 200 (0.1 mm).}
         \label{Fig:Fig_200}
     \end{subfigure}

     \caption{First hitting time with shear flow and chemical attraction in different box sizes. Shear amplitude cutoff is set to $200 \  (0.1 mm) / s$, chemotaxis sensitivity $\chi = 500 \ (0.1 mm)$, chemotaxis cutoff \textcolor{black}{$||\varphi||_{\infty}=1.25 \ (0.1 mm) / s$}.
     Based on $2000$ simulations.}
     \label{fig:curve}
\end{figure}
\ifx
\begin{figure}[hbt!]
     \centering

     \begin{subfigure}[b]{0.545\textwidth}
         \centering
         \includegraphics[width=\textwidth]{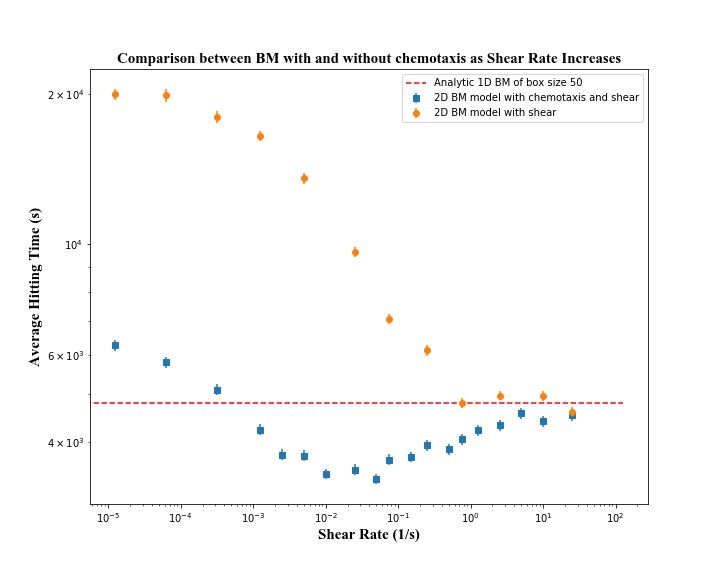}
         \caption{First hitting time with shear flow and chemical attraction: Box size 50 (0.1 mm).}
         \label{Fig:Fig_50}
     \end{subfigure}
     \hfill
     \begin{subfigure}[b]{0.545\textwidth}
         \centering
         \includegraphics[width=\textwidth]{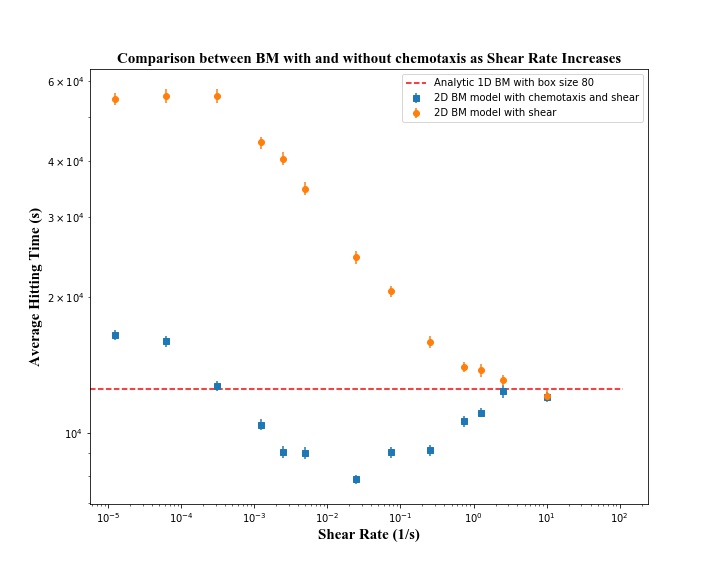}
         \caption{First hitting time with shear flow and chemical attraction: Box size 80 (0.1 mm).}
         \label{Fig:Fig_80}
     \end{subfigure}
          \hfill
     \begin{subfigure}[b]{0.545\textwidth}
         \centering
         \includegraphics[width=\textwidth]{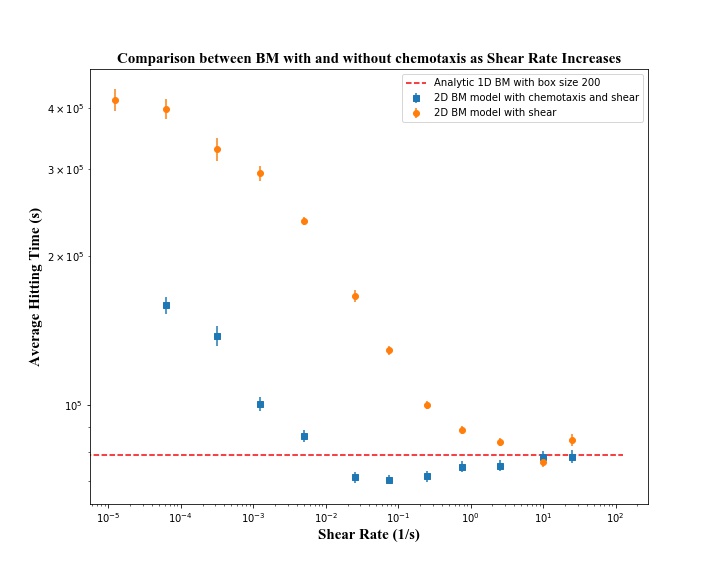}
         \caption{First hitting time with shear flow and chemical attraction: Box size 200 (0.1 mm).}
         \label{Fig:Fig_200}
     \end{subfigure}
     \caption{First hitting time with shear flow and chemical attraction in different box sizes. Shear amplitude cutoff is set to $200 \  (0.1 mm) / s$, chemotaxis sensitivity $\chi = 500 \ (0.1 mm)$, chemotaxis cutoff $||\varphi||_{\infty}=1.25 \ (0.1 mm) / s$.}
     \label{fig:curve}
\end{figure}
\fi

In Figure \ref{fig:different_cutoff}, we explore the dependence of the expected hitting time on shear rate for different values of maximal chemotactic speed.
Here the box size is taken to be $50$ $(0.1)mm$, where the effects we are going to describe are most pronounced (but they are similar for larger boxes). The first interesting
phenomena we observe is that beyond certain point, increasing $\|\varphi\|_\infty$ does not have much effect on expected hitting time:  \textcolor{black}{values $0.125,$ $0.25$ (not pictured) and $1.25$ $(0.1mm / s)$} lead to very close
outcomes. Due to piecewise linear structure of $\varphi,$ the maximal chemotactic speed applies only where the gradient of the chemical $c$ is maximal, meaning near the target.
Although in our simulations this region is not void, apparently it is not sufficiently expansive even for fairly large sensitivity to meaningfully affect the expected hitting time.
Other interesting effects we observe are quite wide range of shear rates where the agent performance exceeds the limiting one-dimensional large shear rate expected search time (the plateau effect),
and the drift in the value of optimal shear rate depending on the maximal chemotactic speed.

\begin{figure}[htb!]
    \centering
    \includegraphics[width=0.75\linewidth]{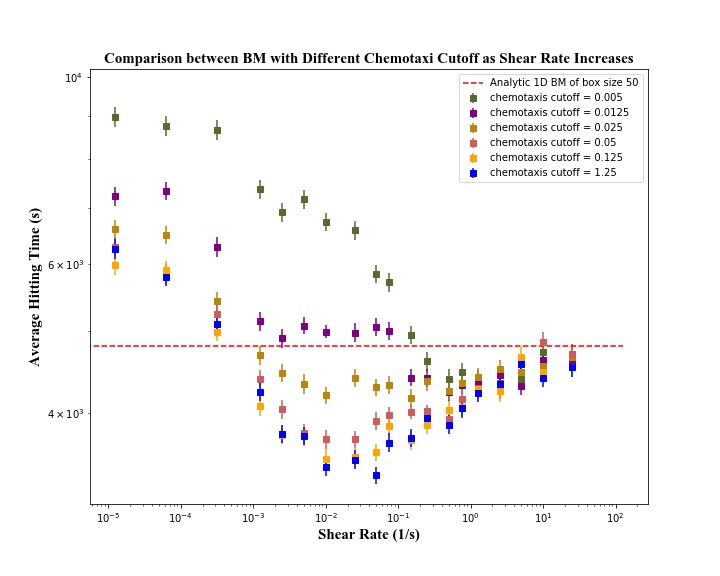}
    \caption{Average first hitting time with varying chemical cut-off $||\varphi||_{\infty}$ and shear rate $A/L$. Box size is $50 \ (0.1mm)$, shear amplitude cutoff is set to $200 \  (0.1 mm) /s$, and chemotaxis sensitivity $\chi = 500\ (0.1 mm)$. }
    \label{fig:different_cutoff}
\end{figure}

The Figure \ref{Fig:plateau effect} illustrates the plateau effect, and shows that for shear rates between $0.015$ and $60$ $s^{-1},$ the agent meaningfully outperforms the limiting one-dimensional large shear rate expected search time
for all chemotactic maximal speeds from \textcolor{black}{$0.025$ to $1.25$ $(0.1mm /  s) $}.  For small chemotactic maximal speeds such as \textcolor{black}{$0.025$ $(0.1mm/s)$}, the expected hitting time values form an almost constant plateau for this
entire range, meaning that even very small values of shear rate combined with very small chemotactic speed are preferable to the dimensional reduction of very high shear rates. We note that sharp improvement
in the agent's search ability even for small values of shear and maximal chemotactic speed are in complete agreement with the results of biological experiments by Riffell and Zimmar. As we mentioned before,
the fertilization rate success in their experiments starts falling  for large values of shear much more dramatically than we observe in our computations. But this is natural: as discussed in the papers \cite{Riffelletall04}, \cite{RiffellZimmer07}, \cite{ZimmerRiffell11}, in the fast shear environment, strong shear flows triggers spinning of the searching sperms. As a result, the sperms are less likely to attach to the eggs and succeed in fertilization.
Hence the fertilization process often fails in the fast shear regime. However, the spinning effect is not taken into account in our numerical simulation. Nevertheless, we observe the great enhancement
of the searching functions for small parameters and the optimal shear rate effect even in the context of our fundamental model, which suggests that these effects are prevalent across many different settings in biology.

\begin{figure}[hbt!]\label{Fig:chemotaxis_cutoff}
    \centering
     \hspace{-4em}
    \begin{subfigure}[b]{0.55\textwidth}
        \centering
        \includegraphics[width=1\linewidth]{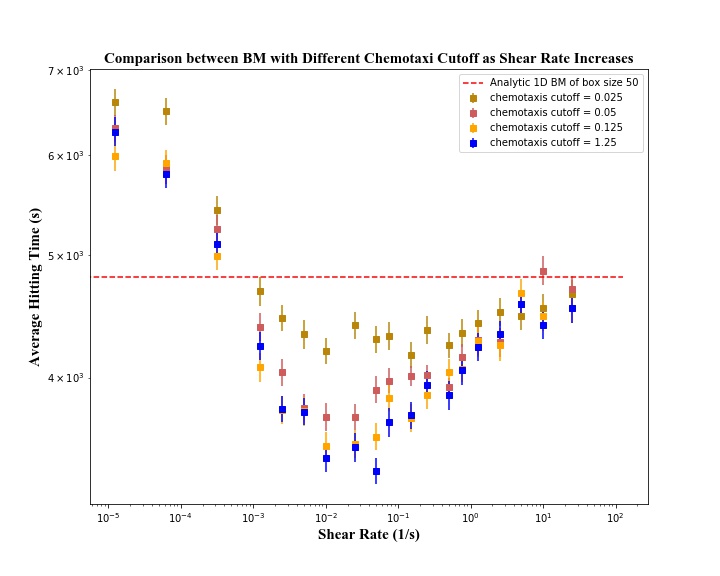}
        \caption{Plateau effect}
        \label{Fig:plateau effect}
    \end{subfigure}%
    \hspace{-2em}%
    \begin{subfigure}[b]{0.55\textwidth}
        \centering
        \includegraphics[width=1\linewidth]{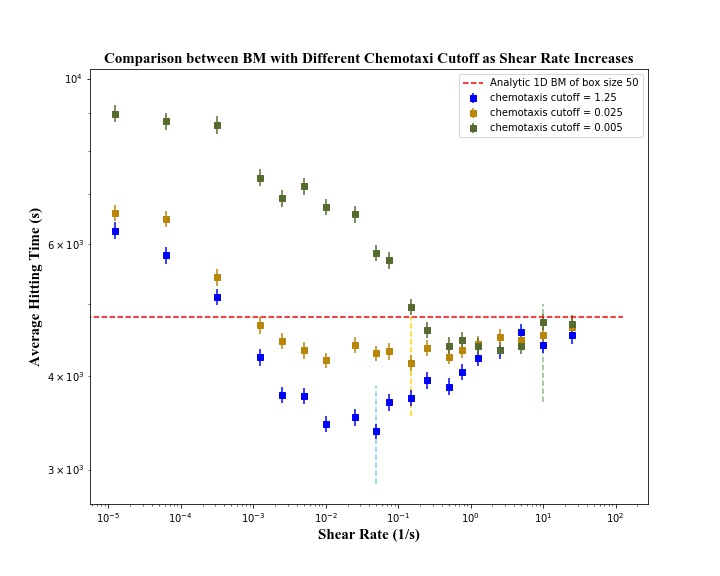}
        \caption{Changes of optimal shear rate}
        \label{Fig:optimal_shear}
    \end{subfigure}
     \hspace{-4em}
    \caption{Plateau effect and changes of optimal shear rate under different chemotaxis cutoff Boxsize is $80 \ (0.1mm)$ and shear amplitude cutoff is set to $200 \ (0.1 mm) /s$.}
\end{figure}

\ifx
\begin{figure}[hbt!]
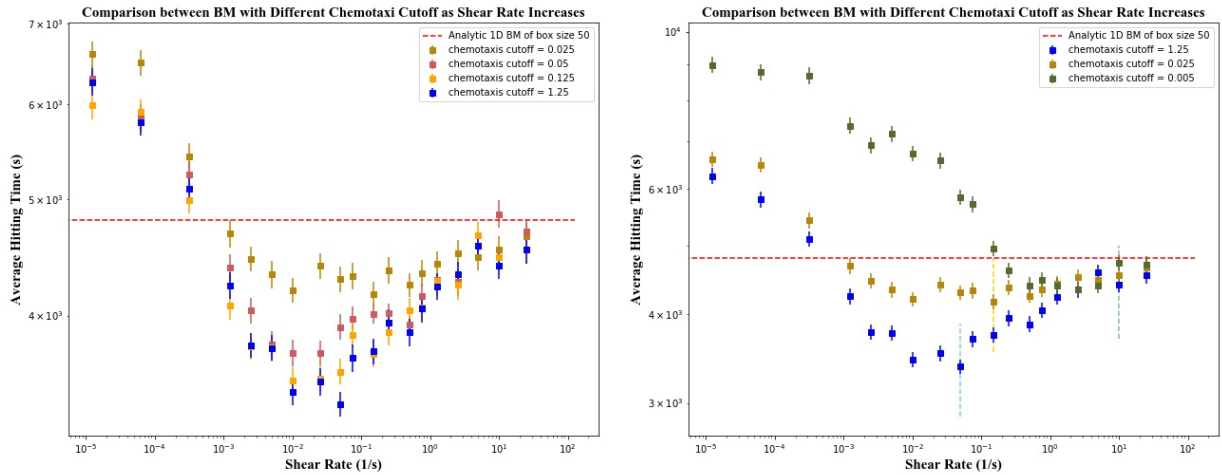
\label{Fig:chemotaxis_cutoff}
    \centering
    \begin{subfigure}[b]{0.5\textwidth}
        \centering
        \includegraphics[width=0.9\linewidth]{chemo_and_shear3.jpg}
        \caption{Plateau effect}
        \label{Fig:plateau effect}
    \end{subfigure}%
    ~
    \begin{subfigure}[b]{0.5\textwidth}
        \centering
        \includegraphics[width=0.9\linewidth]{chemo_and_shear4.jpg}
        \caption{Changes of optimal shear rate}
        \label{Fig:optimal_shear}
    \end{subfigure}
    \caption{Plateau effect and changes of optimal shear rate under different chemotaxis cutoff Boxsize is $80 \ (0.1mm)$ and shear amplitude cutoff is set to $200 \ (0.1 mm) /s$.}
\end{figure}
\fi

The Figure \ref{Fig:optimal_shear} illustrates the dependence of the optimal shear value on chemotactic maximal speed. A natural conjecture is whether the combined effect of shear and chemotaxis is strongest
at the threshold where the agent is just able to outswim the shear flow in the neighborhood of the target. Indeed, if the shear becomes too strong it may nullify the ability of the agent to benefit from the chemical
signal even if it is perceptible. However if this simple mechanism was indeed accurate, we should observe the decline in the optimal shear value when chemotactic maximal speed declines. We could not isolate the parameter regime where
such phenomenon would be clearly observable. Apparently, the interaction between shear and chemotaxis is more nuanced and subtle.
It appears that for the strong and moderate values of chemotactic maximal speeds, the optimal shear values were comparable, in $\sim 0.06-0.3 s^{-1}$ range. For the small values
of maximal chemotactic speed, the optimal shear value tended to go up, not down. For example, for the maximal speed $ 0.005mm/s,$ the optimal shear value is around $1.5s^{-1}.$
The benefit of the shear flow appears to outweigh inability of the agent to go against it for small values of maximal chemotactic speed - up to a point. Very strong shears lead to expected times close to the effective
1D problem for all values of maximal chemotactic speed (at least in the range considered in this paper).

\begin{figure}[htb!]
    \centering
    \includegraphics[width=0.8\linewidth]{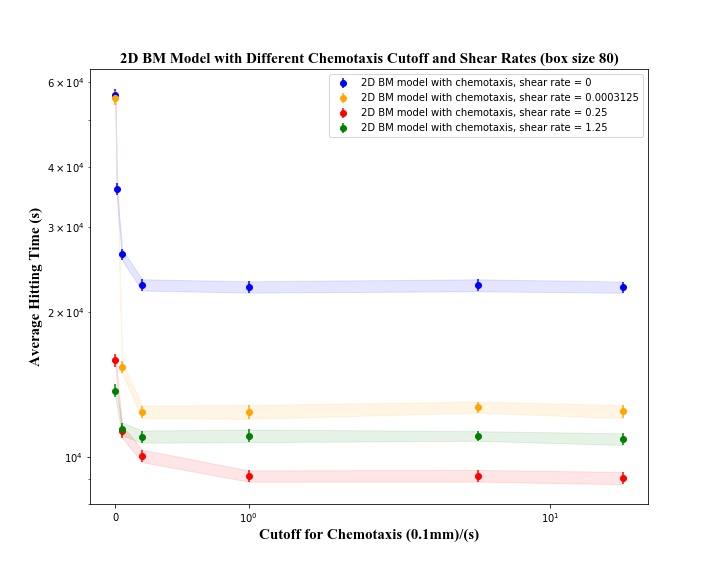}
    \caption{Average first hitting time with varying chemical cut-off $||\varphi||_{\infty}$ and shear rate $A/L$. Boxsize is $80 \ (0.1mm)$, shear amplitude cutoff is set to $ 200 \  (0.1 mm) / s$, and chemotaxis sensitivity $\chi = 500 \ (0.1 mm)$. }
    \label{fig:shear_rate_sensitivity}
\end{figure}

In Figure \ref{fig:shear_rate_sensitivity}, we provide a different perspective on the same phenomena - here the expected hitting time is plotted as a function of maximal chemotactic speed for different values of shear rate.
We see that initially increasing shear rate leads to decrease in the expected hit time, but then it starts going in the opposite direction for all but the smallest values of the maximal chemotactic speed.

\begin{figure}[htb!]
    \centering
    \hspace{-4em}
    \begin{subfigure}[b]{0.55\textwidth}
        \centering
        \includegraphics[width=\textwidth]{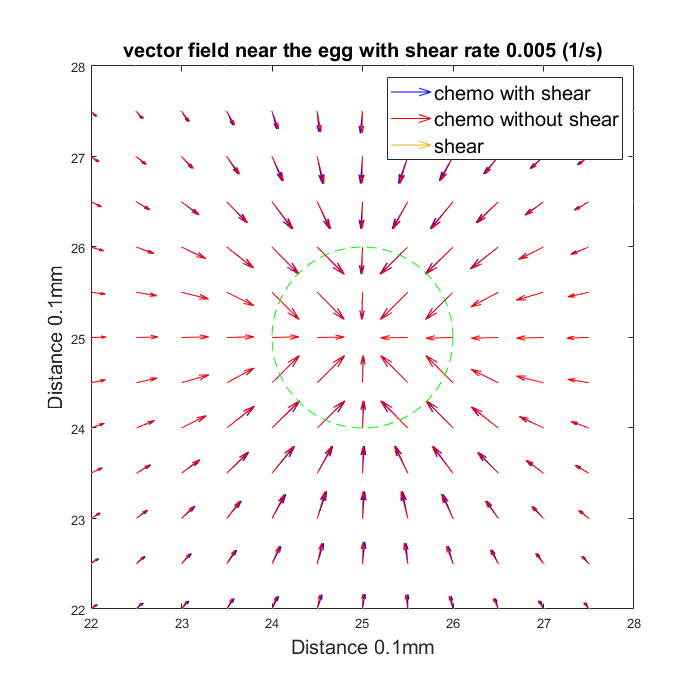}
        \caption[]%
        {{\small Shear rate = 0.03 $(s)^{-1}$}}
        \label{fig:vector_field_0.005}
    \end{subfigure}
    \hspace{-2em}%
    \begin{subfigure}[b]{0.55\textwidth}
        \centering
        \includegraphics[width=\textwidth]{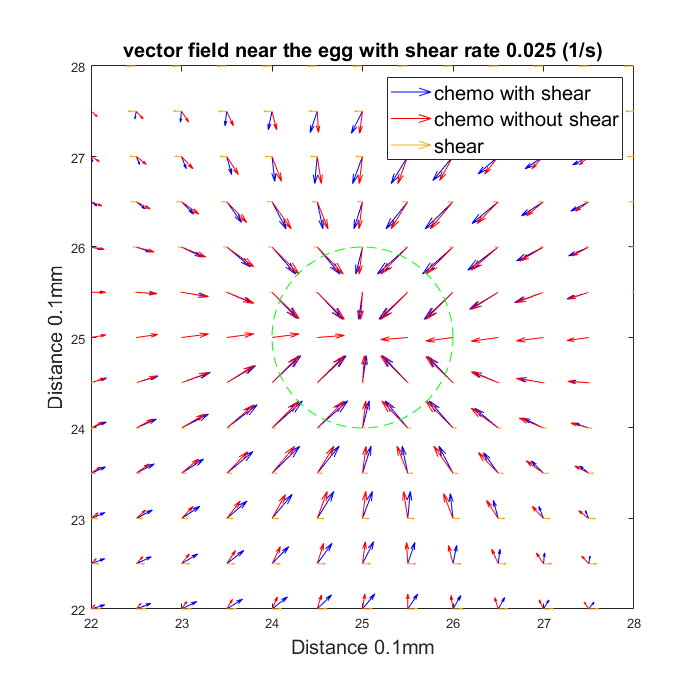}
        \caption[]%
        {{\small Shear rate = 0.15 $(s)^{-1}$}}
        \label{fig:vector_field_0.025}
    \end{subfigure}
    \vskip\baselineskip
    \hspace{-8em}
    \begin{subfigure}[b]{0.55\textwidth}
        \centering
        \includegraphics[width=\textwidth]{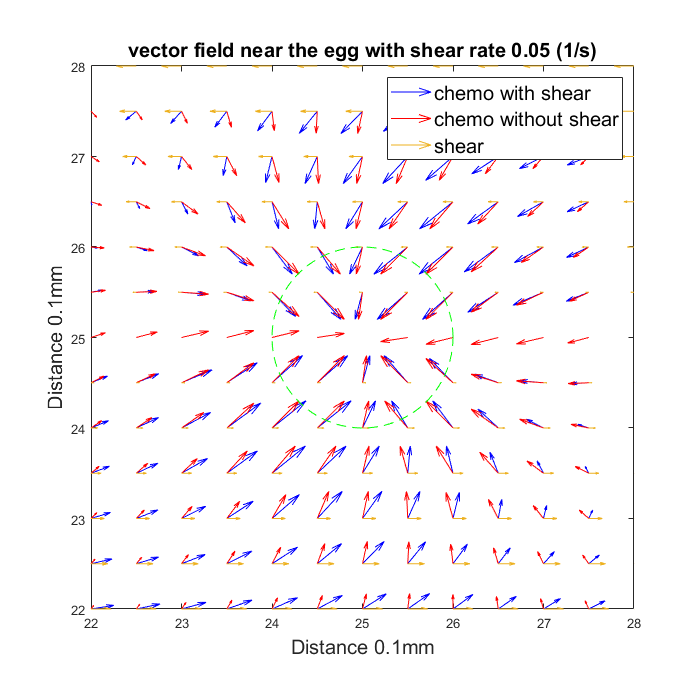}
        \caption[]%
        {{\small Shear rate = 0.3 $(s)^{-1}$}(Optimal)}
        \label{fig:vector_field_0.05}
    \end{subfigure}
    \hspace{-2em}%
    \begin{subfigure}[b]{0.55\textwidth}
        \centering
        \includegraphics[width=\textwidth]{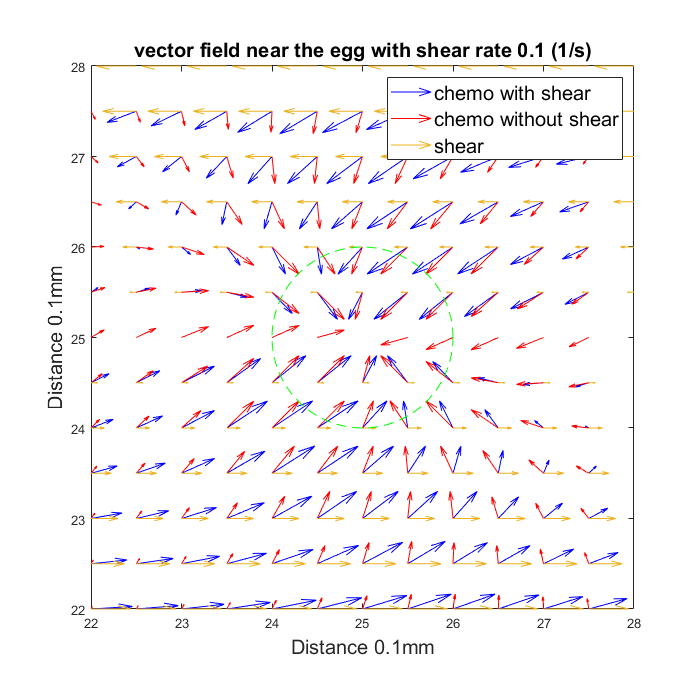}
        \caption[]%
        {{\small Shear rate = 0.6 $(s)^{-1}$}}
        \label{fig:vector_field_0.1}
    \end{subfigure}
     \hspace{-4em}
    \caption{Vector fields around the egg with different shear rates in a box of size $50 \ (0.1mm)$. Shear amplitude cutoff is set to $200 \ (0.1 mm) / s$, chemotaxis sensitivity $\chi = 500 \ (0.1 mm)$, chemotaxis cutoff $||\varphi||_{\infty}=1.25 \ (0.1 mm) /s$.}
    \label{fig:vector_fields}
\end{figure}

\ifx
\begin{figure}[htb!]
    \centering
    \begin{subfigure}[b]{0.485\textwidth}
        \centering
        \includegraphics[width=\textwidth]{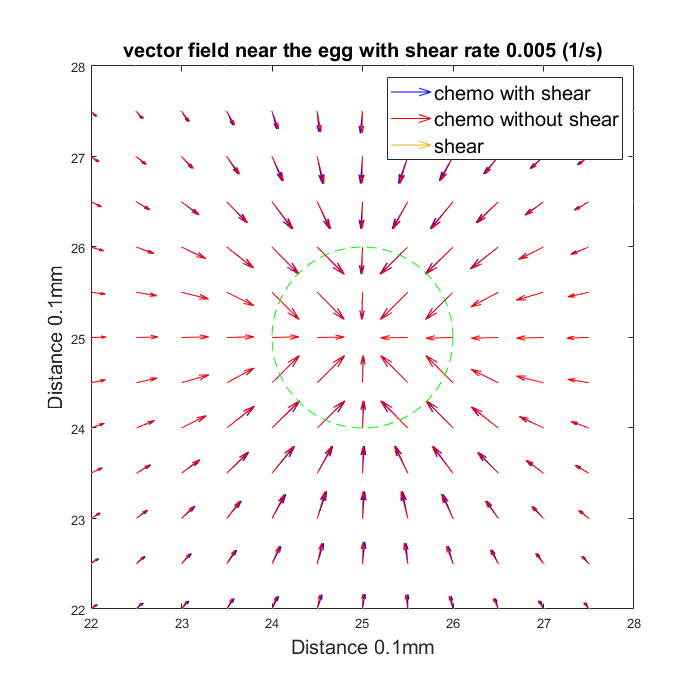}
        \caption[]%
        {{\small Shear rate = 0.03 $(s)^{-1}$}}
        \label{fig:vector_field_0.005}
    \end{subfigure}
    \hfill
    \begin{subfigure}[b]{0.485\textwidth}
        \centering
        \includegraphics[width=\textwidth]{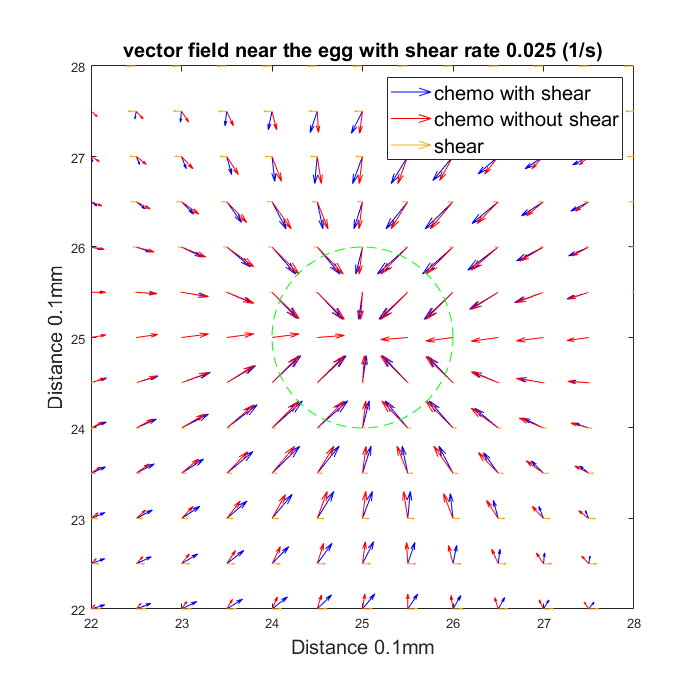}
        \caption[]%
        {{\small Shear rate = 0.15 $(s)^{-1}$}}
        \label{fig:vector_field_0.025}
    \end{subfigure}
    \vskip\baselineskip
    \begin{subfigure}[b]{0.485\textwidth}
        \centering
        \includegraphics[width=\textwidth]{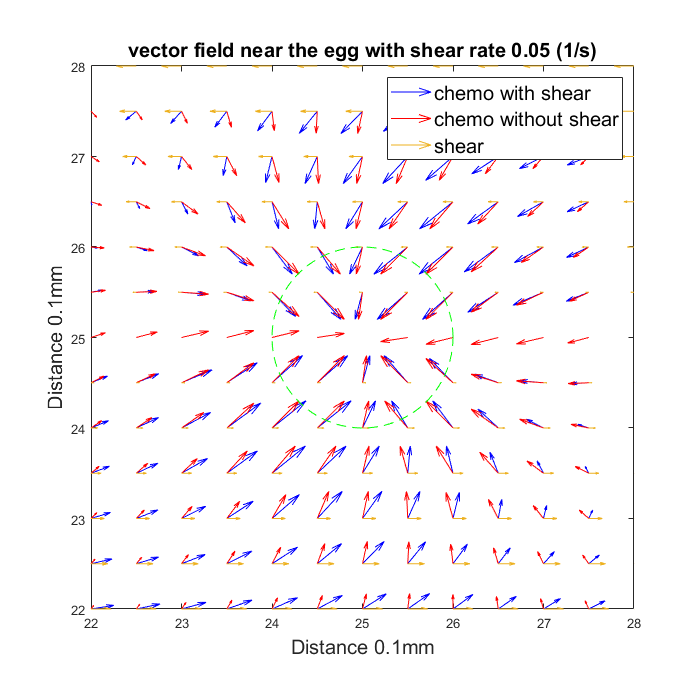}
        \caption[]%
        {{\small Shear rate = 0.3 $(s)^{-1}$}(Optimal)}
        \label{fig:vector_field_0.05}
    \end{subfigure}
    \hfill
    \begin{subfigure}[b]{0.45\textwidth}
        \centering
        \includegraphics[width=\textwidth]{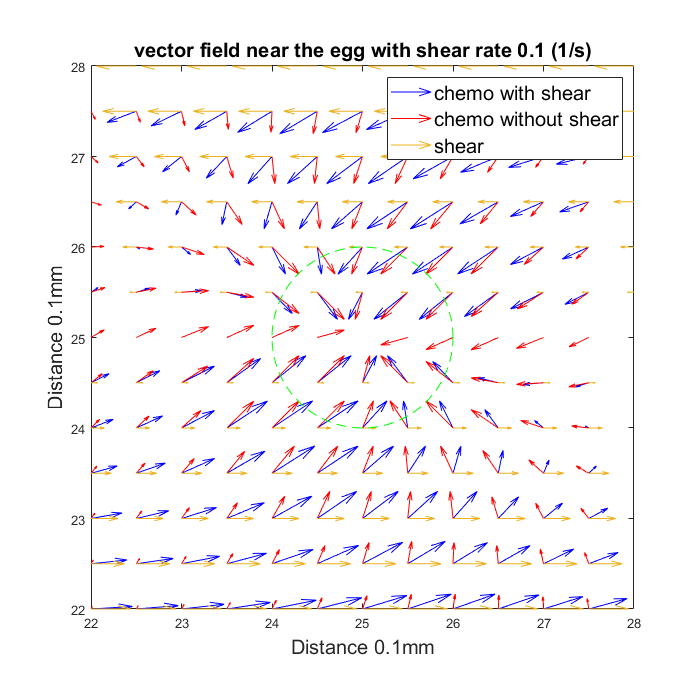}
        \caption[]%
        {{\small Shear rate = 0.6 $(s)^{-1}$}}
        \label{fig:vector_field_0.1}
    \end{subfigure}
    \caption{Vector fields around the egg with different shear rates in a box of size $50 \ (0.1mm)$. Shear amplitude cutoff is set to $200 \ (0.1 mm) / s$, chemotaxis sensitivity $\chi = 500 \ (0.1 mm)$, chemotaxis cutoff $||\varphi||_{\infty}=1.25 \ (0.1 mm) /s$.}
    \label{fig:vector_fields}
\end{figure}
\fi

To further understand the optimal shear rate effect, we performed a numeric experiment without Brownian motion:
\begin{subequations}\label{DE_full}
\begin{align}
\left(\begin{array}{rr}dX_t\\dY_t\end{array}\right)=&\left(\begin{array}{cc} {A} u(Y_t)\\0\end{array}\right)dt+\left(\begin{array}{c}V^{(1)}\\V^{(2)}\end{array}\right)dt;\label{EQ:DE}\\
V=&(V^{(1)},V^{(2)})=\varphi( {\chi} |\nabla c|)\frac{\na c}{|\nabla c|},\quad-\de c+ {A}u(y)\pa_x c=n -c.\label{EQ:chemical2}
\end{align}%\end{align}\begin{align}
\end{subequations}
This system involves the interaction of shear and chemotaxis only, and so it is deterministic.
Note that in \eqref{DE_full}, there is no guarantee that agent will find the target at all.
The result depends on the initial location of the agent. We consider a sample of agents equally spaced along the vertical axis
at $x=0$ (recall that the target is located at $(L/2,L/2)$). Instead of computing the search time, we find the percentage of
agents that do hit the target within a sufficiently large time frame.
In Figure \ref{fig:vector_fields}, we can see the vector fields near the egg with different shear rates in a box of size $50 \ (0.1 mm)$.
The arrows in yellow represent vector field created by shear flow, the arrow in red represent the vector field created by chemotaxis
(note that the effect of shear is still present when we numerically solved for chemical gradient), and the arrows in blue present the sum of shear and chemoattrant vector fields.

\begin{figure}[htb!]
\includegraphics[width=0.8\linewidth]{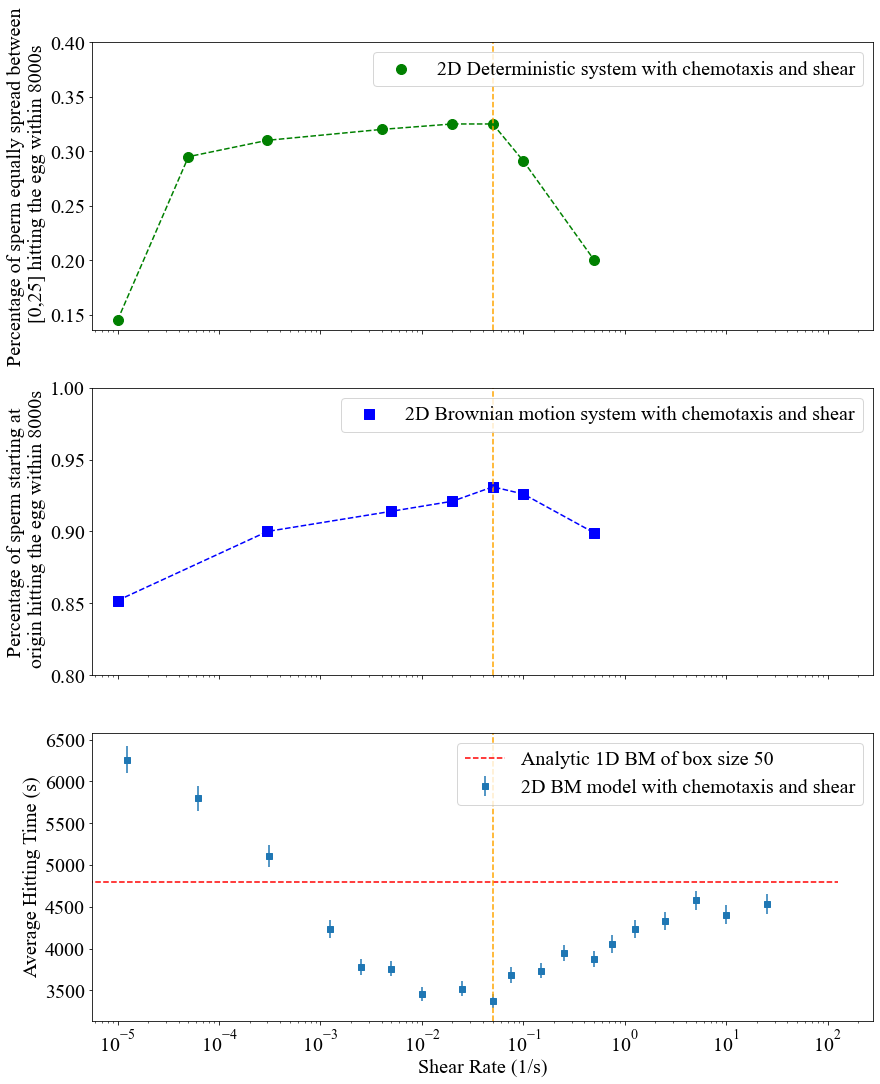}
 \centering
 \caption{Effect of chemotaxis and shear in egg-searching process in deterministic system vs. stochastic system in a box of size $50 \ (0.1mm)$. Shear amplitude cutoff is set to $200 \ (0.1 mm) / s$, chemotaxis sensitivity $\chi = 500 (0.1mm)$, chemotaxis cutoff $||\varphi||_{\infty}=1.25 \ (0.1 mm) /s$. Yellow dotted line denotes optimal configuration with shear rate $0.3
  \ (s)^{-1}$.}\label{fig:BM_Deterministic}
\end{figure}

The maximal agent success rate for this setup happens at values of shear similar to the ones leading to smallest expected hitting time in the simulations with diffusion. In
Figure \ref{fig:BM_Deterministic}, we provide a comparison of simulations without and with diffusion for $L=50\ (0.1 mm),$ $\chi=500 \ (0.1 mm)$ and $\|\varphi\|_\infty =1.25\ (0.1 mm/s).$
 In the simulation without diffusion, we place one agent every $0.01 \ (0.1 mm)$ distance apart, i.e, there are 2501 agent sampling points in the interval $[0,25],$ and let them evolve according to the system \eqref{DE_full}
for time $8000\  s$. Then we record the number of agents that successfully hit the target in this time frame to obtain Figure \ref{fig:BM_Deterministic}.
%In the deterministic system within a fixed time, We can see an initial increase in percentage of sperm hitting the egg successfully as shear rate increases, then followed by a decrease as shear rate continue to increase.
This simulation in deterministic system gives a natural parallel to the phenomenon we observed in the stochastic system (Figure \ref{Fig:Fig_50}).
Moreover, the optimal shear rates are similar in both cases, and the plateau effect is even more pronounced in the deterministic case.
Thus these effects are intrinsic features of interaction between shear flow and chemotaxis.

{\color{black} %\textbf{Siming: Yishu did two sets of figures. I think one is the trajectories in the full torus, and the other set is in the neighborhood of the egg zone. I am not sure which set to put, so I put both of them here.  }

We also show sample trajectories of the agent right before hitting the target zone (Figure \ref{fig:full_trajectory}). When the shear rate is optimal (0.3 $s^{-1}$), we observe that the searching agent can turn around and approach the target zone. On the other hand, if the shear rate overpowers the chemical attraction, the searching agent can be washed away even though it is right next to the target.
\begin{figure}[htb!]
    \centering
       \ifx
    \begin{subfigure}[b]{0.485\textwidth}
        \centering \includegraphics[width=\textwidth]{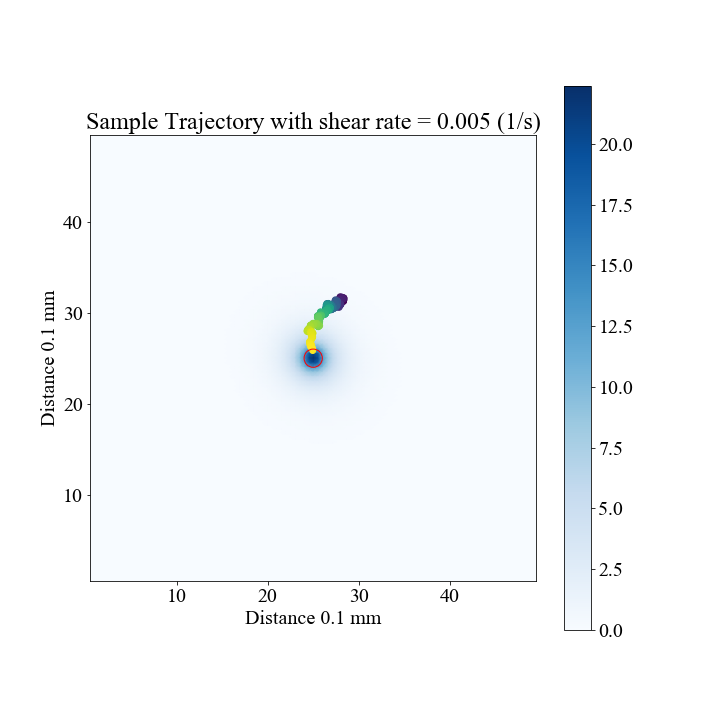}
        \caption[]%
        {{\small Shear rate = 0.03 $(s)^{-1}$}}
        \label{fig:full_trajectory_0.005}
    \end{subfigure}
    \hfill
    \begin{subfigure}[b]{0.485\textwidth}
        \centering
        \includegraphics[width=\textwidth]{sample_full_traj0025.png}
        \caption[]%
        {{\small Shear rate = 0.15 $(s)^{-1}$}}
        \label{fig:full_trajectory_0.025}
    \end{subfigure}
    \vskip\baselineskip
    \fi
    \begin{subfigure}[b]{0.485\textwidth}
        \centering
        \includegraphics[width=\textwidth]{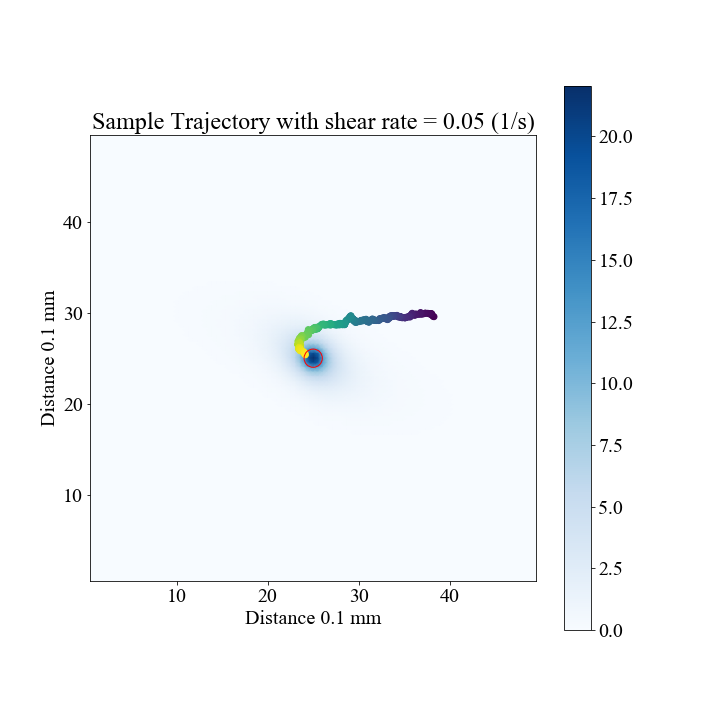}
        \caption[]%
        {{\small Shear rate = 0.3 $(s)^{-1}$}(Optimal)}
        \label{fig:full_trajectory_0.05}
    \end{subfigure}
    \hfill
    \begin{subfigure}[b]{0.485\textwidth}
        \centering
        \includegraphics[width=\textwidth]{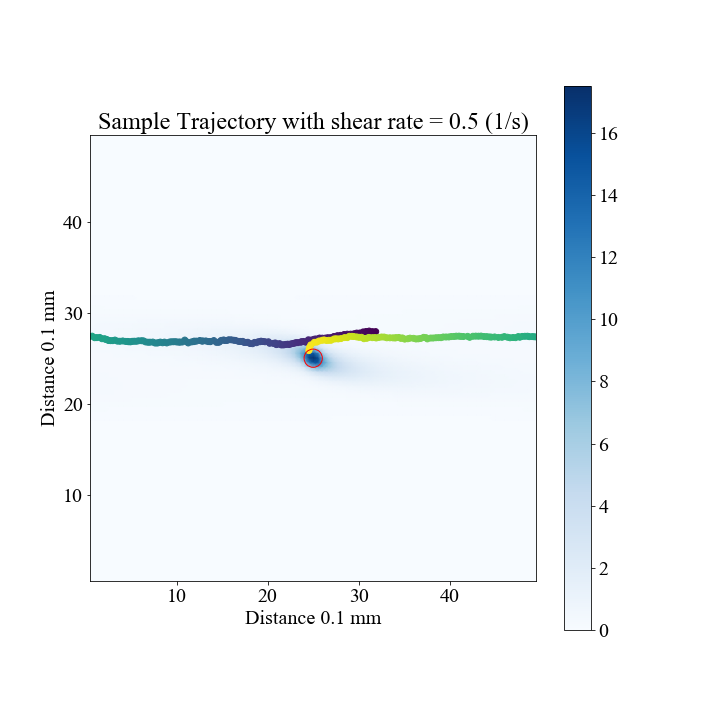}
        \caption[]%
        {{\small Shear rate = 3.0 $(s)^{-1}$}}
        \label{fig:full_trajectory_0.5}
    \end{subfigure}
    \hfill
    \begin{subfigure}[b]{0.485\textwidth}
        \centering
        \includegraphics[width=\textwidth]{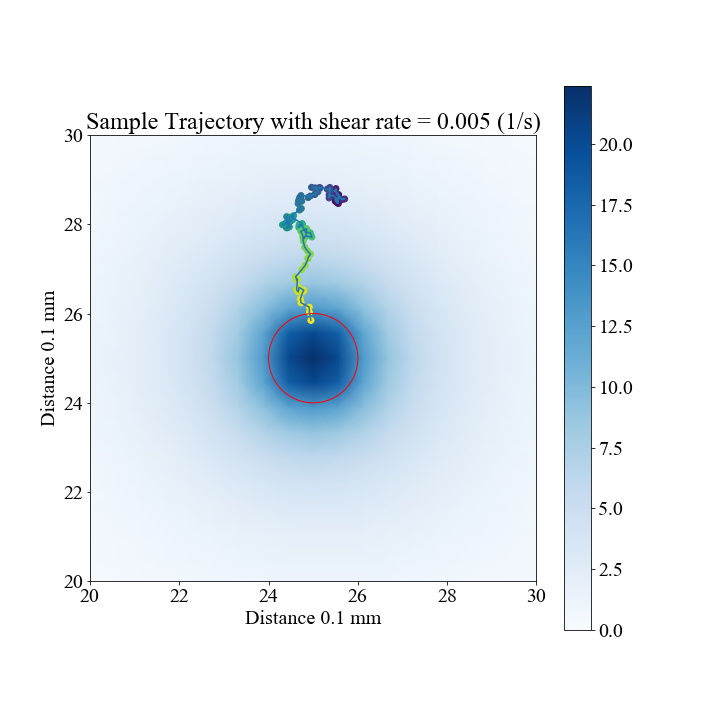}
        \caption[]%
        {{\small Shear rate = 0.03 $s^{-1}$}}
        \label{fig:local_trajectory_0.005}
    \end{subfigure}
    \hfill
    \begin{subfigure}[b]{0.485\textwidth}
        \centering
        \includegraphics[width=\textwidth]{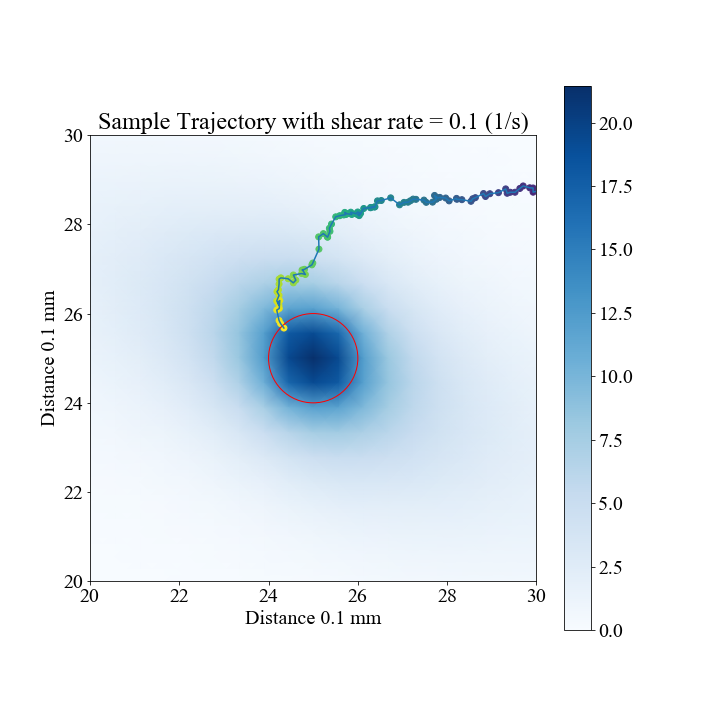}
        \caption[]%
        {{\small Shear rate = 0.6 $s^{-1}$}}
        \label{fig:local_trajectory_0.1}
    \end{subfigure}
    \caption{Sample searching trajectories around the egg with different shear rates in a box of size $50 \ (0.1mm)$. 
    %Shear amplitude cutoff is set to $200 \ (0.1 mm) / s$, 
    Chemotaxis sensitivity $\chi = 500 \ (0.1 mm)$, chemotaxis cutoff $||\varphi||_{\infty}=1.25 \ (0.1 mm) /s$.}
    \label{fig:full_trajectory}
\end{figure}

{\color{black}
To better understand what happens at shear rate values close to optimal, we plot several graphs (Figure \ref{Fig:angle_density}) that depict the probability distribution of the agent's hitting points on the surface of the target parameterized by the approach angle $\theta$ (with and without chemotaxis and at different values of shear rate). The approach angle $\theta$ is defined as the angle between the first hitting position of the target zone $E$ and the positive direction of the $x$-axis. To compute the approach angle from the discrete trajectory of the searching agent, we identify the first position of the agent after entering the target zone. Then we interpolate between this entering position and the agent's previous position in the simulation.
The interpolation line will intersect the boundary of the egg zone at a point. The approach angle is calculated using this intersection point.
%To make the original plot smoother, we also apply local averaging.
\begin{figure}[htb!]
    \centering
    \begin{subfigure}[b]{0.485\textwidth}
        \centering
        \includegraphics[scale=0.4]{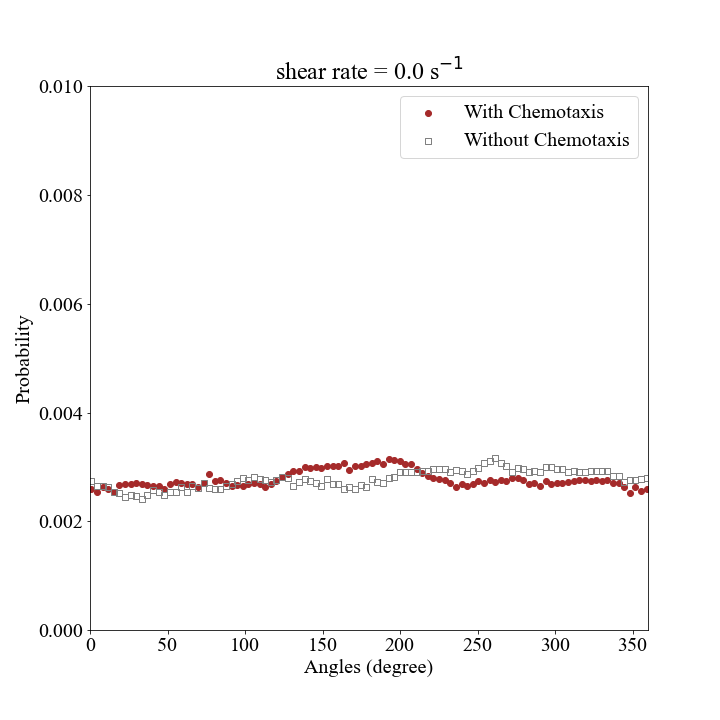}
        \caption[]%
        {{\small Shear rate = 0 $s^{-1}$}}
        \label{fig:Angle_0}
    \end{subfigure}
\hfill
    \begin{subfigure}[b]{0.485\textwidth}
\includegraphics[scale=0.4]{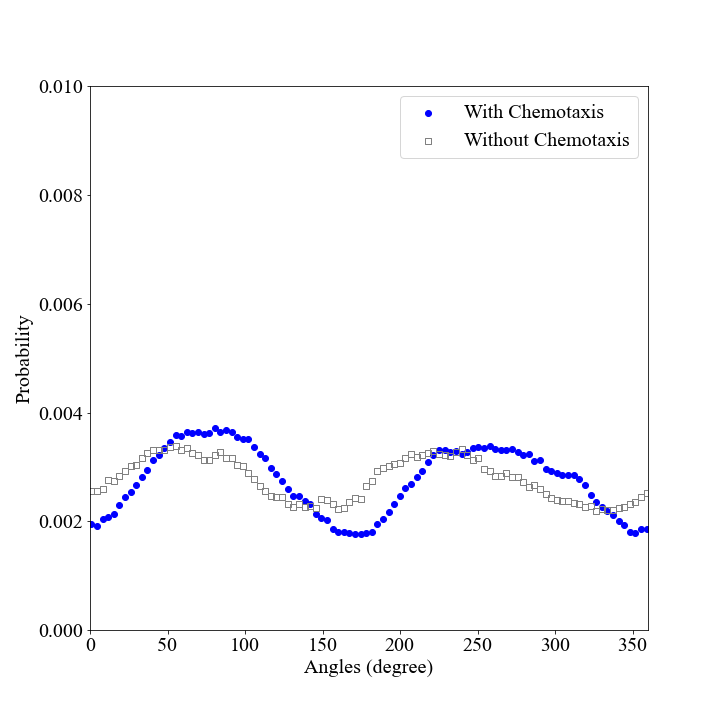}
\caption[]%
        {{\small Shear rate =  0.03 $s^{-1}$}}
        \label{fig:Angle_0005}
    \end{subfigure}

  \begin{subfigure}[b]{0.485\textwidth}     \centering
\includegraphics[scale=0.4] {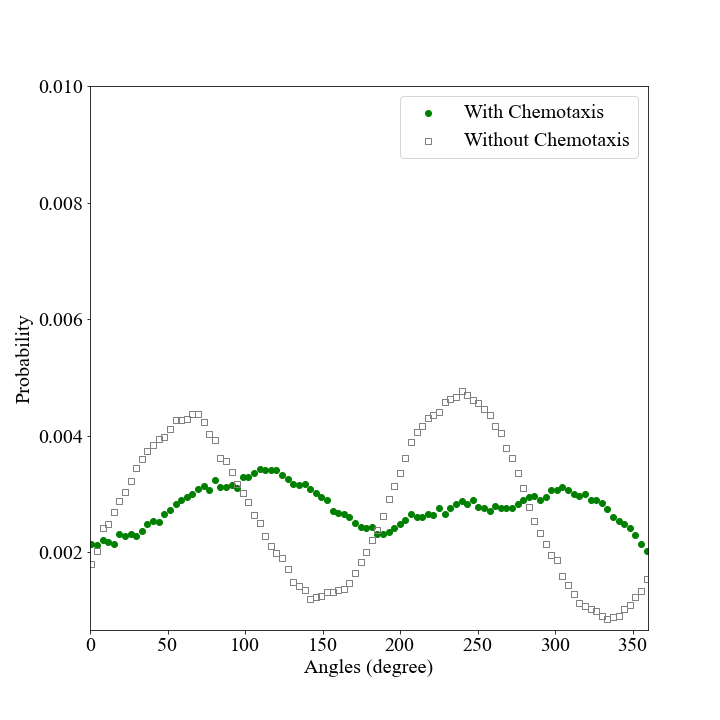}
\caption[]%
        {{\small Shear rate = 0.3 $s^{-1}$}}
        \label{fig:Angle_005}
    \end{subfigure}
    \hfill
    \begin{subfigure}[b]{0.485\textwidth}
\includegraphics[scale=0.4]{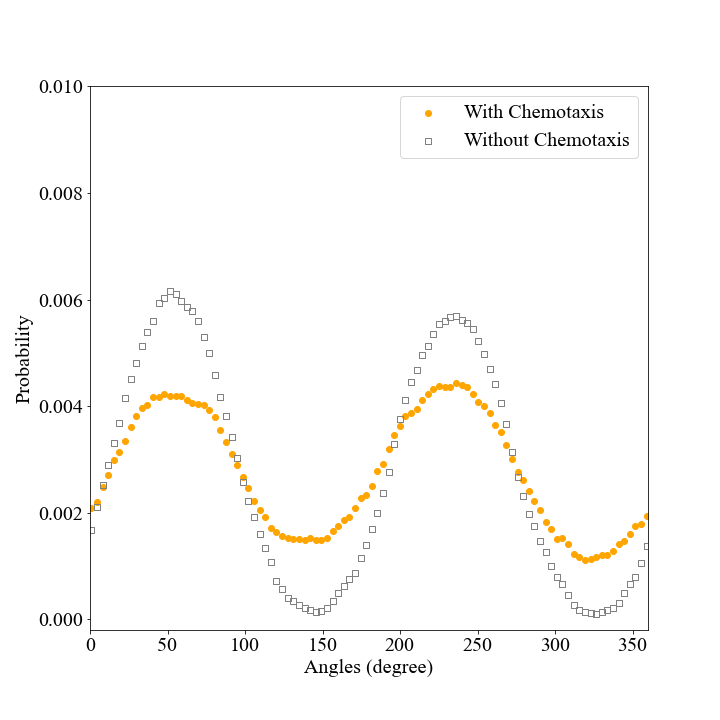}
\caption[]%
        {{\small Shear rate = 3.0  $s^{-1}$}}
        \label{fig:Angle_05}
    \end{subfigure}
 \caption{The empirical distribution of the approach angle with and without chemotaxis for varying shear rates} \label{Fig:angle_density}

\end{figure}

At zero shear rate, the distribution of the approach angles is close to uniform both with and without chemotaxis. For small values of shear rate, the distributions with and without chemotaxis deviate from uniform 
and start to shift away from each other, although they remain close. For the shear only case hitting the target on the side exposed to the shear becomes more likely than on the protected down flow side, while the angle distribution with chemotaxis
begins to shift to the right compared to shear only. At near optimal shear rates, we see an interesting phenomenon in that 
the peaks of angle distributions with and without chemotaxis are clearly misaligned. Without chemotaxis, the peaks are around $70^\circ$ and $250^\circ,$ which corresponds to the sides exposed to shear flow, 
and the minima are around $150^\circ$ and $330^\circ$ corresponding to the protected from shear sides of the target. Both maxima and minima are more pronounced than when chemotaxis is present. 
With chemotaxis, the maxima shift to about $120^\circ$ and $300^\circ,$ which in fact lie on the down flow parts of the target and are indicative of a large number of trajectories pulled towards the target 
even after passing it but ending up in the attractive chemical cloud. The minima in the chemotactic case are right around $0^\circ$ and $180^\circ.$ At high shear rates, the angle distributions with and without 
chemotaxis become aligned, with extremal points of the approach angle probability distribution with chemotaxis pulled towards those without. This is indicative of a more limited ability for the trajectories to come back from behind against the flow. 
Some trace of this ability remains though, since although the variation in the probability distribution grows for both cases, in the case with chemotaxis it is less pronounced - which corresponds to more even distribution
of hitting points on the surface of the target. 

}

%If one zooms in, we have the following
  % 17
}
% observe increasing shear rate first shortens the average hitting time of a single sperm. Then the hitting time starts to increase and converges to analytic 1D hitting time as shear rate becomes larger. Moreover, the optimal setting is when shear rate is $0.3 \ (4s)^{-1}$.

\subsection{Chemotactic maximal speed vs sensitivity}

%One observes that the expected first hitting time undergoes a significant decrease when there are both shear and chemical attraction. However, as the shear rate increases, the expected first hitting time first drops and then increases back to the $1$-D expected first hitting time. The distance between %the lowest point and the $1$-D hitting time is $280.5\  (4s)$. In Figure \ref{Fig:Fig_80}, we make a similar plot with torus size $80\  (0.1mm)$. The distance between the lowest point and the $1$-D hitting time is $1150.17\  (4s)$. In Figure  \ref{Fig:Fig_200}, the numerical experiment is done on torus %of size $200 \ (0.1mm)$. The distance between the lowest point and the $1$-D hitting time is $2220.5\  (4s)$.

We have also explored the role of two key parameters in our chemotaxis model: chemotactic sensitivity and maximal speed cutoff $\|\varphi\|_\infty.$ For these simulations, we set
shear rate to zero. The main conclusion we can draw from these simulations is that increased sensitivity appears to be more important for reduction of the expected hitting time than
maximal chemotactic speed cutoff.

\begin{figure}[htb!]
\centering
\begin{subfigure}[b]{0.7\textwidth}
   \includegraphics[width=1\linewidth]{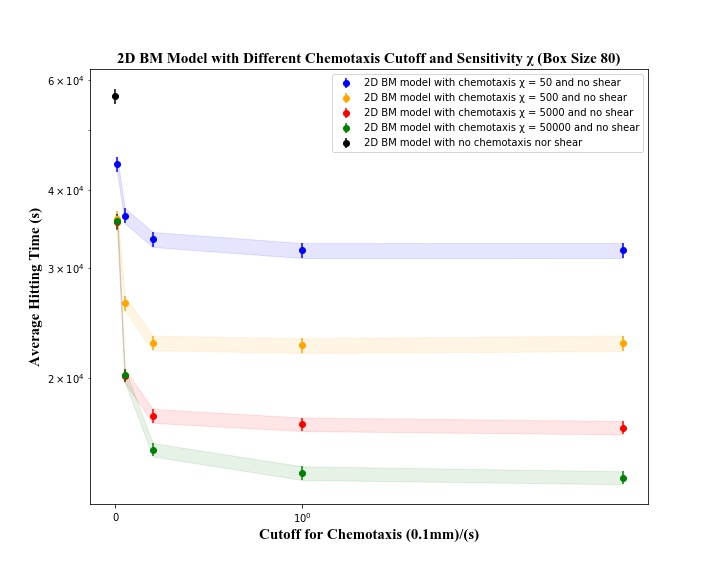}
   \caption{}
   \label{Fig:chemotaxis_sensitivity}
\end{subfigure}

\begin{subfigure}[b]{0.7\textwidth}
   \includegraphics[width=1\linewidth]{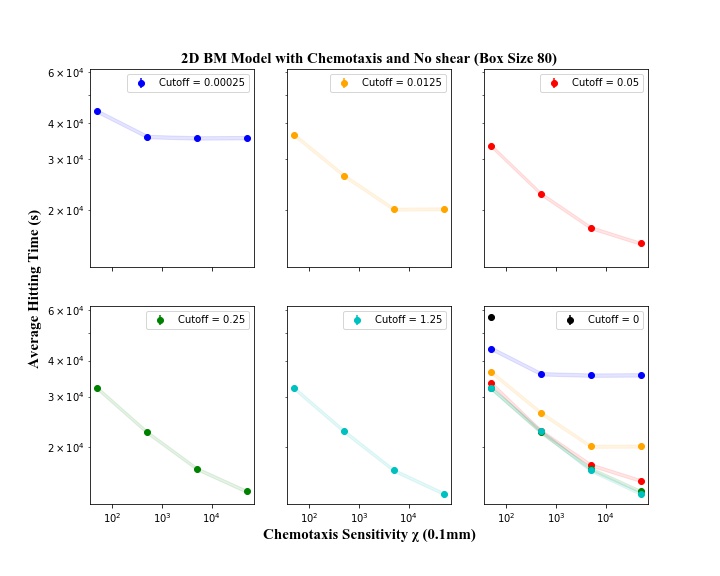}
   \caption{}
   \label{Fig:chemotaxis_sensitivity2}
\end{subfigure}

\caption[]{Average first hitting time with varying chemical cut-off $\varphi$ and chemical sensitivity $\chi$. Shear flow is not present.}
\end{figure}

In the Figure \ref{Fig:chemotaxis_sensitivity}, we fix the size of the torus to be $80 \ (0.1mm)$ and explore the relationship between the average first hitting time and the maximal chemotactic speed cut-off $||\varphi||_\infty$ for different values
of chemical sensitivity. We observe that the most of reduction in the expected hitting time happens already before reaching the maximal chemotactic speeds  \textcolor{black}{$ \sim 0.05-0.1\ (0.1mm/s)$}. The additional improvement throughout the tested speed range
of up to $ 0.125mm/s$ is marginal. It should be noted though that even marginal improvement may be important when there is active competition between different agents.
On the other hand, Figure \ref{Fig:chemotaxis_sensitivity2} shows that increase in chemical sensitivity throughout the whole tested range continues to endow significant advantage to the agents with all but the smallest maximal chemotactic speeds.

\section{Mathematical Analysis: Main Results}
%\subsection{Analysis of the Average First Hitting Time}
Throughout this section, we assume the same setting as before: the search takes place on a torus $L\Torus^2= [0,L)^2.$ The target is a disc of radius $\delta$ located at $(L/2,L/2).$
The agent performs target search in shear flow aided by chemotaxis (of the same form as described in \eqref{EQ:chemical}). The only difference with the setting of numerical experiments is
that we will not in general assume a certain starting point for the search, but instead will consider expected hitting times with arbitrary initial position.
We also choose not to normalize any parameters, and carry $\delta$ and $\nu$ through the estimates.

%\subsection{Large shear rate limit without chemotaxis}

Let us first consider the case of large $A$ limit when chemotaxis is not present (and the box size $L$ is fixed).
One expects that strong shear effectively reduces one dimension. The result we state below is in the spirit of Freidlin-Wentzel theory \cite{FreidlinWentzell},
is likely known and is certainly in the folklore. However we could not find a convenient reference to quote and include the proof for the sake of completeness.

When no chemical aggregation is present, i.e., $\chi=0$, the SDE for the searching agent is simple, i.e.,
\begin{align}\label{Passive_transport_SDE}
\left(\begin{array}{rr}dX_t\\dY_t\end{array}\right)=&\left(\begin{array}{cc} {A} u(Y_t)\\0\end{array}\right)dt+\sqrt{2\nu}\left(\begin{array}{cc}d B_t^{(1)}\\ dB_t^{(2)}\end{array}\right),\\
(X_t,Y_t)=&(x_0,y_0).
\end{align}
Let $\mathbb{E}^{(x_0,y_0)}T_{2D}^A$ be the expected time of the process \eqref{Passive_transport_SDE} on torus $\Torus^2=[0,L)^2$ starting at $(x_0,y_0)$
to hit the target $B((L/2,L/2);\delta)$, namely,
\begin{align}\label{Defn_T_2D_A}
\mathbb{E}^{(x_0,y_0)}T_{2D}^A=\mathbb{E}^{(x_0,y_0)}\min_t\{t|(X_t, Y_t)\in B((L/2,L/2);\delta), \quad  (X_t,Y_t)\text{ solves the SDE }\eqref{Passive_transport_SDE}\}.
\end{align}
To capture the large shear behavior of the averaged first hitting time, we define the one-dimensional first hitting time
\begin{align}\label{Defn_T_1D}
T_{1D}=\min_t\left\{t\bigg|Y_t\in \left[\frac{L}{2}-\delta, \frac{L}{2}+\delta\right],\, (X_t,Y_t)\text{ solves \eqref{Passive_transport_SDE}. }\right\}. %Y_0=y_0
\end{align}
As the shear strength $A$ increases, one expects that the searching agent traverses the horizontal direction fast, and $\mathbb{E}T_{2D}^A$ approaches $\mathbb{E}T_{1D}$.

%In the appendix, we present a proof of the following theorem
\begin{theorem}\label{thm:First_hitting_time_passive_transport}
Consider the equation \eqref{Passive_transport_SDE}.
Suppose the shear profile $u(y)\in C^2(\Torus)$ is such that $u(L/2)=0$.
Moreover, assume that %the absolute value $|u|$ is bounded from below in the sense that
\begin{align}
\label{Strictly_Monotone}\min_{y\in[L/2-2\delta,L/2+2\delta]}|u'(y)|\geq u_d>0.%\min_{y\in[\frac{3L}{8}, \frac{5L}{8}]}
\end{align}
%\begin{align}\label{u_m_lower_bound}
%\min_{y\in [\frac{L}{2}-2\delta,\frac{L}{2}-\frac{\delta}{2}]\cup[\frac{L}{2}+\frac{\delta}{2}, \frac{L}{2}+2\delta]}|u(y)|\geq u_m>0.
%\end{align}
%Further assume that the initial $y$-position of the agent is away from the level of the egg zone, i.e.,  $y_0\notin [L/2-\delta, L/2+\delta]$.
Then the average first hitting time $\mathbb{E}^{(x_0,y_0)}T_{2D}^A$ \eqref{Defn_T_2D_A} approaches $\mathbb{E}^{y_0}T_{1D}$ \eqref{Defn_T_1D}, i.e.,
\begin{align}
\lim_{A\rightarrow \infty}\mathbb{E}^{(x_0,y_0)}T_{2D}^A=\mathbb{E}^{y_0}T_{1D}.
\end{align}
\end{theorem}
%\textcolor{red}{Drop $u'(L/2)>0$. Introduce that we are on the torus of size $L$ and the size of the egg $\delta\ll L$. Drop the condition $y_0\notin [L/2-2\delta,L/2+2\delta]$ and add a line in the proof. We don't have the first trip and the proof is direct. }\textcolor{black}{Drop: `and $u'(L/2) > 0$.'    %I am not sure about dropping the condition on $y_0\notin [L/2-2\delta, L/2+2\delta]$. The speed of the shear is zero at $y=L/2$! So if  $y_0=L/2$, it might take $\sim \nu^{-1}$ to reach $E_\delta$. If we shift the shear such that $u\geq 1/C$ on $y\in[L/2-2\delta, L/2+2\delta$, then it seems that it is %okay? }%\footnote{}
\begin{remark}
In our simulation, $u(y) = \sin(2\pi(y-L/2)/L),$ and {\color{black} $u_d\sim L^{-1}$.} %{\color{black} HS: I have replaced the $u_m \sim \delta/L{ (u_d\sim L^{-1})}$ based on the referee's concern. He thought that the $u_m$ is not defined before.}
%2. The assumption that $y_0\notin [L/2-\delta, L/2+\delta]$ cannot be entirely dropped in our setting (although it can be weakened).
%Note for example that when $y_0=0,$ $T_{1D}=0,$ but the expected hitting time $\mathbb{E}^{(x_0,y_0)}T_{2D}^A$ does not necessarily converge to zero as $A \rightarrow \infty:$
%think of a case where $u(y)=0$ on $[L/2-\delta/4, L/2+\delta/4].$ One can, of course adopt a different or additional assumption on $u$ that would rule out such behavior,
%but we try to keep statements simple and intuitive.
\end{remark}
\begin{remark}
It is not hard to show that $\lim_{A\rightarrow \infty}T_{2D}^A(\omega)=T_{1D}(\omega)$ almost surely. Indeed, by Blumenthal's 0-1 law, if we let
$\tau=\inf\{t\geq 0:B_t>0\},$ then $\mathbb{P}_0(\tau=0)=1$ see e.g. \cite{Durrett}. So once we reach one of the levels $y -L/2 = \pm \delta$ at $T_{1D}(\omega),$
by almost sure continuity of Brownian motion, with probability one there is an interval of times arbitrarily close to the time $T_{1D}(\omega)$ during which we will dip into the target zone $|y-L/2| < \delta.$
 Thus as $A \rightarrow \infty,$ $T_{2D}^A(\omega)$ will converge to $T_{1D}(\omega).$
%For a reference of this result, we refer to the interested readers to the book ``Brownian motion and Martingales in analysis'' by Richard Durrett or Theorem 2.8 in Aldous online notes. So for arbitrarily small time, the Brownian motion become strictly negative, and the agent will enter the strip %$\{y||y-L/2|<\delta\}$.  From that point on, by taking $A$ large compared to the realization of the Brownian trajectory, we have that the time goes to $T^{1D}(\omega)$ almost surely. }
\vspace{3mm}
\end{remark}

If there exists chemical attraction $(\chi\neq 0)$, the analysis is more involved.
As the magnitude of the shear flow $Au(y)$ increases, the chemical gradient $\na c$ homogenizes in the horizontal direction. Hence the attraction vector field $V$ is also homogenized in the $x$-direction. To explicitly capture the homogenization effect, we consider the elliptic equation on  $(x,y)\in L\Torus\times L\Torus$:%in the bounded channel \times \left[\frac{1}{4}L, \frac{3}{4}L\right]
\begin{align}\label{Chemical_Equation}
-\de c+Au(y)\pa_x c=n-c.
\end{align}
As before, we assume that the  target density $n$ is stationary and its support is localized near the center of the torus, i.e.,
\begin{align}
\{y|(x,y)\in \mathrm{support}\{n\}\}= B_\delta(L/2,L/2) % \left[\frac{L}{2}-\delta, \frac{L}{2}+\delta\right].
\label{vertical_support_of_n}
\end{align}
Moreover, the shear flow $u(y)$ will be assumed to be strictly monotone in the neighborhood of the support of $n$, i.e., satisfy \eqref{Strictly_Monotone}.
%{\color{red} Do we really have to assume that much? It would be nice to have an assumption not in terms of $L$}

Next we introduce the horizontal-homogenization. Functions $f$ on $L\Torus\times L\Torus$ can be decomposed into $x$-average $\lan f\ran $ and the remainder $f_\sim$:
\begin{align}\label{x-average_and_remainder_of_chemical}
\lan f\ran(y)=\frac{1}{L}\int_{0}^L f(x,y)dx,\quad f_\sim(x,y)=f(x,y)-\lan f\ran (y).
\end{align}
%\textcolor{red}{Define this notation for general $f$. }
The remainder of the chemical  $c_\sim$ will be  homogenized in the $x$-direction in the sense that $c_\sim,\, \na c_\sim$ decay to zero as $A$ approaches infinity.   The results are summarized in the next theorem.
\begin{theorem}\label{thm:Horizontally_homogenized}
Consider the solutions $c$ to the equation \eqref{Chemical_Equation}. The  shear flow profile $u\in C^2(\Torus)$ has only finitely many critical points and is non-degenerate in the sense that if $u'(y_0)=0$ at a point $y_0$, then $u''(y_0)\neq 0$.
%{\color{red} what does degenerate mean in our context? We never defined it.
%Is it satisfying \eqref{Strictly_Monotone}? Then we should just say so.}
Further assume that the shear profile is strictly monotone near the egg zone \eqref{Strictly_Monotone} and the egg density is localized \eqref{vertical_support_of_n}. Then the chemical density $c_\sim$ and its derivatives %$\na c_\sim$
up to the second order are approaching zero as $A$ approaches infinity, i.e.,
\begin{align}
||\pa_x^{i}c_{\sim}||_\infty\leq \frac{C|\log A|||\pa_x^{i}n_{\sim}||_\infty}{A^{2/3}},\quad ||\pa_x^{j}\pa_y c_{\sim}||_\infty\leq \frac{C|\log A|^{2}||\pa_x^j\na n_{\sim}||_\infty}{A^{1/6}},\quad i\in\{0,1,2\},\, j\in\{0,1\}.\label{Quantitative_Horizontally_homogenization}
\end{align}% Due to the application of the new results in \cite{AlbrittonBeekieNovack21}, we can improve the log.
Here the constant $C$ depends on the parameters $\nu,\, L$.
%{\color{red} please check and restate. Right now there is no result for $c_\sim.$}
\end{theorem}
%\begin{remark}To derive the same homogenization estimates on the torus $L\Torus\times [0,L]$, one can apply even reflections of the data $(n,c)$ and the shear profile $u(y)$ around the horizontal lines $y=\frac{1}{4}L$ and $y=\frac{3}{4}L$.
%\end{remark}

Now we consider the convergence of the average first hitting time in the full generality ($\chi\neq 0$). We use the notation $T^{A;\chi}_{2D}$ to denote the first hitting time of the SDE \eqref{SDE_full}, namely,
\begin{align}\label{Defn_T_2D_A_chi}
T_{2D}^{A;\chi}=\min_t\{t|(X_t, Y_t)\in B((L/2,L/2);\delta), \quad  (X_t,Y_t)\text{ solves the SDE }\eqref{SDE_full}\}.
\end{align}
Next we define the effective $1D$ system
\begin{align}\label{1D_effective_SDE}
dY_t=V_{\mathrm{eff}}^{(2)}dt+\sqrt{\nu}dB_t,\quad V_{\mathrm{eff}}^{(2)}=\varphi(\chi|\pa_y \lan  c\ran|)\frac{\pa_y \lan c \ran}{|\pa_y \lan c\ran|},\quad Y_0=y_0.
\end{align}
Observe that $\lan c \ran (y)$ solves a simple PDE
\[ \partial^2_y \lan c \ran = \lan n \ran - \lan c \ran. \]
We define the $1D$-first hitting time as follows
\begin{align}\label{T_1D_chi}
T_{1D}^\chi=\min_t\left\{t\bigg|Y_t\in \left[\frac{L}{2}-\delta, \frac{L}{2}+\delta\right],\, Y_t \text{ solves \eqref{1D_effective_SDE}. } \right\}.
\end{align}
The value of $\mathbb{E}^{y_0}T_{1D}^\chi$ can be calculated using Dynkin's formula and integration factor method; we will outline this computation in Section \ref{Sec:Proof_thm_3}. %{\color{red} make an explicit more precise reference?}
\begin{theorem}\label{theorem_first_hitting_time_full_system}
Consider the dynamics \eqref{SDE_full}.  Assume that the density $n\in C^2(\Torus^2)$, and the cutoff function ${\varphi(\cdot)}\in C^2(\rr_+)$ vanishes at the origin, i.e., $\varphi(0)=0$. %\textcolor{black}{(How much regularity to assume in order to get the convergence of the semigroup?)}
The shear profile $ u\in C^2(\Torus)$ satisfies all assumptions in Theorem \ref{thm:First_hitting_time_passive_transport} and \ref{thm:Horizontally_homogenized}.
%(\eqref{u_m_lower_bound}, \eqref{vertical_support_of_n}, \eqref{Strictly_Monotone}). %smooth ,\infty   Furthermore, the conditions  are fulfilled.
%Let the initial $y$-position of the agent satisfy $y_0\notin [L/2-2\delta, L/2+2\delta]$.
Then the expected first hitting time $\mathbb{E}^{(x_0,y_0)}T_{2D}^{A;\chi}$  \eqref{Defn_T_2D_A_chi} approaches the expected $1D$-first hitting time $\mathbb{E}^{y_0}T_{1D}^\chi$ \eqref{T_1D_chi} as the magnitude $A$ tends to infinity, i.e.,
\begin{align}
\lim_{A\rightarrow\infty}\mathbb{E}^{(x_0,y_0)}T_{2D}^{A;\chi}=\mathbb{E}^{y_0}T_{1D}^\chi, \quad \forall (x_0,y_0)\in (L\Torus)^2.
\end{align}
\end{theorem}
%\textcolor{red}{For any $(x_0,y_0)\in (L\Torus)^2$. Drop the $\varphi'(0)=0$ because in (A.58) we don't need the $\varphi'(0)$. Check whether this assumption is used elsewhere?} \textcolor{black}{Dropped. Checking. }
\begin{remark}
In fact, explicit convergence rate will be derived in  Section  \ref{Sec:Proof_thm_3}.
\end{remark}
%\section{Conclusion}
%In this paper, we carry out numerical experiments and mathematical analysis to understand the average first hitting time of the SDE \eqref{SDE_full}. In the numerical study, we observe the following.
%\begin{itemize}
%\item{Even though the fast shear does not deplete the reaction rate in our numerical study, we still observe non-trivial relation between the average first hitting time and the shear strength. For moderately small shear, the hitting time decreases drastically. For strong enough shear, the average hitting time increases and approaches that of the one-dimensional average hitting time. }
%\item{We observe that increasing the chemical sensitivity $\chi$ is more effective than increasing the aggregation speed $||\varphi||_\infty$ when one tries to decrease the average first hitting time. }
%\end{itemize}
%We prove the following theorem.
%\begin{itemize}
%\item{The average first hitting time approaches the average one-dimensional first hitting time as the shear strength $A$ approaches infinity.}
%\item{The density of the searching agent gets homogenized in the horizontal direction with an enhanced rate.}
%\end{itemize}

%\appendix
%\section{Appendix}
\section{Proof of Theorem \ref{thm:First_hitting_time_passive_transport} }

%\subsection{Proof of Theorem \ref{thm:First_hitting_time_passive_transport} }
%\subsection{Expected First Hitting Times}\label{Sec:1D_2D_time}
%in Dimension Two
%In this section, we estimate the average time for the searching agent to find the egg zone, under the assumption that chemical attraction is not present, i.e., $\chi=0$.

If $\chi=0$, it is well known how to calculate the expected first hitting time $\mathbb{E}^{y_0}T_{1D}$ \eqref{Defn_T_1D}. Since only the $y$-component of the agent's position determines whether the agent hit the target region $\left[\frac{L}{2}-\delta,\frac{L}{2}+\delta\right]$, it is enough to consider the SDE $dY_t=\sqrt{2\nu}B_t^{(2)}$. The agent starts at $y_0,$ and performs the Brownian motion $\sqrt{2\nu}dB_t$ on $[0,L)$ with periodic boundary conditions until it hits the interval $[\frac{L}{2}-\delta,\frac{L}{2}+\delta].$ We can recast this problem equivalently as the exit time from $[-L/2+\delta,L/2-\delta]$ for a Brownian particle starting at $y_0$ (without loss of generality we assume that $y<L/2$). The expected first exit time is well known and can be computed explicitly; we provide a brief sketch of the argument.
First recall the Dynkin formula (\cite{Oksendal}): for $f\in C_0^2$, suppose $\tau$ is a stopping time, $\mathbb{E}^{y_0}[\tau]<\infty$, then
\begin{align}\label{Dynkin}
    \mathbb{E}^{y_0}[f(Y_\tau)]=f(y_0)+\mathbb{E}^{y_0}\left[\int_0^\tau H f(Y_s)ds\right],
\end{align}
where in our case
\begin{align}
    dY_t=\sqrt{2\nu}dB_t,\quad
    Hf= \nu \pa_{yy}f,\quad Y_0=y_0.
\end{align}

%To derive the formula for the average exit time, we define $\tau$ to be the first exit time.
To apply the formula, we consider the  solution $f$ to the partial differential equation:
\begin{align}\label{Poisson_equation}
     {\nu}\pa_{yy}f=-1,\quad f\left(\pm \frac{L}{2} \mp \delta\right)=0.
\end{align}
Combining the equation and the formula \eqref{Dynkin},  we have obtained the relation
\begin{align}
    0=\textcolor{black}{\mathbb{E}^{y_0}[f(Y_\tau)]}=f(y_0)+\mathbb{E}^{y_0}\left[\int_0^\tau H f(Y_s)ds\right]=f(y_0)-\mathbb{E}^{y_0}\left[\tau\right],
\end{align}
which in turn yields that
\begin{align}
    f(y_0)=\mathbb{E}^{y_0}[\tau].
\end{align}
Directly solving the equation yields that
\begin{align}\label{form}
\mathbb{E}^{y_0}T_{1D}= \mathbb{E}^{y_0}[\tau]= f(y_0)=-\frac{1}{2\nu}y^2_0 + \frac{1}{2\nu}\left(\frac{L}{2}-\delta \right)^2.
\end{align}
\ifx
The goal of this section is to consider the large $A$ regime. We will prove the following theorem.
\begin{theorem}
Suppose $u$ is a smooth function such that $u(L/2)=0$ and $u'(L/2) > 0.$ Moreover, we assume that $\delta$ is sufficiently small relative to the scale of variations in $u,$  so that $|u(y)| \geq u_m >0$ for $y \in [L/2 \pm \delta/2,L/2 \pm 2\delta].$
Let $\mathbb{E}^{(x_0,y_0)}[T_{2D}^A]$ be the expected time of the process \eqref{Passive_transport_SDE} on torus $\Torus^2=[0,L)^2$ starting at $(x_0,y_0)$
to hit the target $B((L/2,L/2);\delta).$ Then as $A \rightarrow \infty,$ $\mathbb{E}^{x_0}[T_{2D}^A]$ converges to the one dimensional hitting time $\mathbb{E}^{y_0}T_{1D}$ \eqref{form}. %2\nu^{-1}(\frac{L}{2}-\delta)^2
\end{theorem}
\fi
Now we prove Theorem \ref{thm:First_hitting_time_passive_transport}.
%{\color{red} Please adjust the proof for the new assumption and no exceptional starting points. Need to introduce $u_m$ somewhere, also for further reference.}
\begin{proof}[Proof of Theorem \ref{thm:First_hitting_time_passive_transport}]
For simplicity, we will provide below an argument for the agent starting at $(0,0).$ It can be generalized to $(x_0,y_0)$ with $y_0 \notin [\frac{L}{2}-\delta, \frac{L}{2}+\delta]$ in a straightforward manner. For $y_0\in  [\frac{L}{2}-\delta, \frac{L}{2}+\delta]$, the argument is similar, but extra modifications are required.  We will comment on the adjustments at the end of the proof.
%(H: I have adjusted the position of this paragraph. The reason is that the setup ($S^\pm, u_m$) below are different for $y_0\in[L/2-\delta, L/2+\delta]$.)}

First note that the expected 2D hitting time is larger than the 1D hitting time. This can be seen through considering the $y$-component. If $Y_t$ does not reach the region $ [\frac{L}{2}-\delta, \frac{L}{2}+\delta]$, then the agent does not find the target.
Hence the 2D hitting time is bounded below by the 1D hitting time.

To derive an upper bound, let us focus on the searching strips
\[ S^+ =\left\{(x,y)\bigg|y\in \left[\frac{L}{2}+\delta-2A^{-1/3},\frac{L}{2}+\delta\right]\right\}\,\,\, {\rm and}\,\,\, S^-=\left\{(x,y)\bigg| y \in\left[\frac{L}{2}-\delta,\frac{L}{2}-\delta+2A^{-1/3}\right]\right\}. \]
These two strips are adjacent to the boundary of the target zone $y=\frac{L}{2} \pm \delta$.
Due to our assumptions $u(L/2)=0$ and \eqref{Strictly_Monotone}, and assuming that $A$ is sufficiently large,
for the values of $y$ within the searching strips $S^\pm$ the magnitude of the velocity $u$ has a lower bound. We denote it by
\begin{align}\label{u_m}
u_{m}:=\min_{y\in S^\pm}|u(y)|  \geq \frac{\delta u_d}{2}>0.
\end{align}
%In our simulation, $u_m \sim \delta/L.$
%Similarly, on $S_0^-$, the velocity is bounded from above,  $u(y)\leq -u_m<0$. Further define the following refined searching strip $S^+$ and $S^- $:
%\begin{align}
%S^+ =\left[\frac{L}{2}+\frac{1}{2}\delta -2A^{-1/3},\frac{L}{2}+\frac{1}{2}\delta\right];\\
%S^-=\left[\frac{L}{2}-\frac{1}{2}\delta,\frac{L}{2}-\frac{1}{2}\delta+2A^{-1/3}\right].%7
%\end{align}
Denote the center levels of $S^\pm$ as $y^\pm=\frac{L}{2}\pm \delta \mp A^{-1/3}$ respectively. % And the center level of $S^-$ as $y_c^-=\frac{L}{2}-\frac{1}{2}\delta +A^{-1/3}$.

\begin{figure}[hbt!]
\centering
\includegraphics[scale=.51]{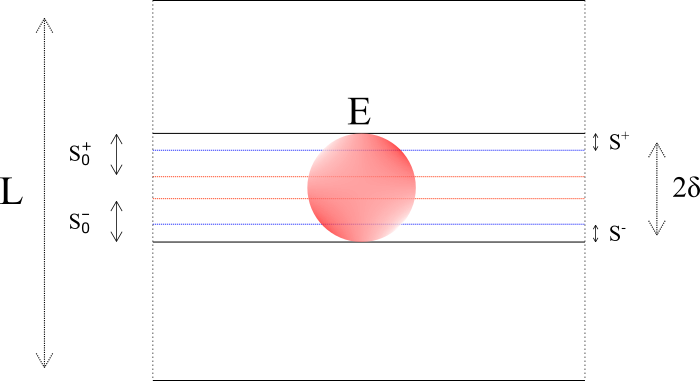}
\caption{Setup}
\end{figure}
%{\color{red} Please make a picture of a round egg. A square egg is weird.}

%Recall that our initial position is $(0,0)$.
The expectation of the first hitting time $\tau_{0}$ of the second component getting from $(0,0)$ to one of the center levels $y^\pm$
can be explicitly computed similarly to \eqref{form}. %is the same as the first hitting time to the egg zone, and it can be explicitly calculated.
Applying the formula analogous to \eqref{form} but replacing $L/2-\delta$ with{ $L/2-\delta+A^{-1/3}$}, % \textcolor{black}\nu^{-1}(\frac{L}{2}-\delta)^2{\color{red} are you sure we want $\nu=1$? in 1D we had $1/2$}
 we obtain that
\begin{align}\label{tau_0}
\mathbb{E}[ \tau_{0}]=\frac{1}{2\nu}\left(\frac{L}{2}-\delta+A^{-1/3}\right)^2.
\end{align}
%Similar calculation yields the average first hitting time $\tau_0$  from $(x_0,y_0)$ to the level $y\in \{y_c^+, y_c^+\}$. Note that the distance between $y_0$ to this level is $D+A^{-1/3}$ and the length of the interval in consideration is $\mathcal{L}=L-\delta-2A^{-1/3}$. Combining these and the %formula \eqref{form} yields that
%\begin{align}
%\mathbb{E}\tau_0=(D+A^{-1/3})(L-\delta+A^{-1/3}-D).\label{tau_0}
%\end{align}
The two hitting times $\eqref{form}_{y_0=0}$, \eqref{tau_0} differ by a small term of order ${O}(A^{-1/3})$.%\mathcal

Once the agent hits one of the center levels of either $S^+$ or $S^-$, we will focus on that specific strip.
Without loss of generality, we assume that the agent first reaches $y^+$. If $A$ is very large, there is a high chance that the agent will hit the target before
it exits the strip $S^+.$ Indeed, if the agent remains for time $\tau^{in} \geq \frac{2L}{A u_m}$ inside $S^+,$ it has enough time to traverse the entire torus and hit the target.
On the other hand, by the reflection principle (see, e.g., equation (3.8) in  \cite{Durrett96}), we have that
\begin{align*}
\mathbb{P}\left(\tau^{in}\leq \frac{2L}{A u_m}\right)=2\mathbb{P}\left(\sqrt{\nu}B_{\frac{2L}{A  u_m}}\geq A^{-1/3}\right).
\end{align*}
Since
%\begin{align*}
$\mathbb{E}[(B_t)^{2}]= t,$
%\end{align*}
the following estimate holds:
\begin{align*}
\mathbb{P}\left(\tau^{in}\leq \frac{2L}{A   u_m}\right)\leq 2\mathbb{P}\left(\sqrt{\nu}B_{\frac{2L}{A  u_m}}\geq A^{-1/3}\right) \leq 2\frac{\mathbb{E}\left[\left(B_{\frac{2L}{A u_m}}\right)^{2}\right]}{\nu^{-1}A^{-2/3}} \leq \frac{4\nu L}{A^{1/3}u_m}.
\end{align*}

Another possibility is that the $\tau^{in}\geq \frac{2L}{Au_m}$, but the agent does not traverse through all the searching strip. The probability of this event is again small. Denoting $X_e$ as the entering position and denoting $d_\rr$ as the distance on the universal cover $\rr$ of the torus $\Torus$, we estimate the probability as follows, %i.e.,
\begin{align*}
\mathbb{P}\left(\tau^{in}\geq \frac{2L}{Au_m},\, d_\rr(X_{\frac{2L}{Au_m}},X_e)\leq L\right)\leq \mathbb{P}\left(\sqrt{\nu}B_{ \frac{2L}{Au_m}}\geq  L\right)\leq \frac{2\nu}{A u_mL}.%\tau^{in} MATH
\end{align*}
We define the event
\begin{align}\label{F_i}
F_i:=\left\{\tau^{in}\leq \frac{2L}{Au_m}\right\}\cup \left\{\tau^{in}\geq\frac{2L}{Au_m}, d_\rr(X_{\frac{2L}{Au_m}},X_e)\leq L\right\}.
\end{align}
We observe that $F_i^c$ guarantees a success in the  $i$-th trip to $S^\pm$. % (H: the original claim that ``the event $F_i$ is the event that the search fails in the $i$-th trip to $S^\pm$" is not quite right.)} Hence the probability of $F_i$ is bounded by
\begin{align}
\label{shortprob}\mathbb{P}(F_i)\leq \mathbb{P}\left(\left\{\tau^{in}\leq \frac{2L}{Au_m}\right\}\right)+\mathbb{P}\left( \left\{\tau^{in}\geq\frac{2L}{Au_m}, d_\rr(X_{\frac{2L}{Au_m}},X_e)\leq L\right\}\right)\leq \frac{4\nu L}{A^{1/3}u_m}+\frac{2\nu}{A u_mL}.%r\cup
\end{align}
%\textcolor{black}{Here we just need that $\mathbb{P}(F_i)<1$.}

Now in the event $F_i$ %(H: delete `that $\tau^{in} \leq \frac{L}{A  u_m}$' because that expression is not quite what we have)},
we wait for time $\tau^{out}$ till the agent gets back either to the level $y^+$ or $y^-$.
%(for simplicity we do not consider the option that it gets to $y^-$ before that, which would also be a favorable to search event).
For simplicity we can ignore the option of reaching $y^-$ and focus just on $y^+.$
Since at the stopping time $\tau^{in},$
the agent is at the boundary of $S^+,$ analogously with the previous 1D argument, we can compute the expectation of time $\tau^{out}$ by solving the boundary
value problem %{\color{red} please double check the coefficient in front of $f''$}
\[ \nu f''(z) =-1, \,\,\, f(z=0)=f(z=L)=0, \]
where $\mathbb{E}[\tau^{out}]=f(A^{-1/3}).$ Solving for $f$ yields
\[ f(z) = \frac{1}{2\nu}(-z^2 + Lz), \]
and thus $\mathbb{E}[\tau^{out}]=\frac{1}{2\nu}( LA^{-1/3}-A^{-2/3}).$ Now we can iterate, and by Markov property we obtain %{\color{black}(H: I have adjusted the $\mathbb {E}[\cdot]$ notation)}
\begin{align*}
\mathbb{E}^{(x_0,y_0)}[T_{2D}^A] \leq & \mathbb{E}[\tau_0]+\mathbb{P}\left(F
_1^c\right)\frac{2L}{A u_m}%\\\tau^{in}\geq \frac{L}{A u_m}&
+\mathbb{P}\left(F_1\right)\left(\mathbb{E}[\tau^{out}]+\mathbb{P}\left(F_2^c\right)\frac{2L}{A u_m}\right)\\%\tau^{in}\leq \frac{L}{A u_m}   \tau^{in}\geq \frac{L}{A u_m}
&+\mathbb{P}\left(F_1\right)\mathbb{P}\left(F_2\right)\left(\mathbb{E}[\tau^{out}]+\mathbb{P}\left(F_3^c\right)\frac{2L}{A u_m}\right)+...\\  %\tau^{in}\leq \frac{L}{A  u_m} \tau^{in}\leq \frac{L}{A u_m}\tau^{in}\geq \frac{2L}{A u_m}
=&\mathbb{E}[\tau_0]+\mathbb{P}\left(F_1^c\right)\frac{2L}{A u_m}+\sum_{i=1}^\infty\left(\prod_{j=1}^{i}\mathbb{P}\left(F_j\right)\right)\left(\mathbb{E}[\tau^{out}]+\mathbb{P}\left(F_{i+1}^c\right)\frac{2L}{A u_m}\right).%\\\tau^{in}\geq \frac{L}{A u_m}&
%\tau^{in}\leq \frac{L}{A u_m}\tau^{in}\geq \frac{L}{A u_m}j
\end{align*}
Using our estimates on $\mathbb{E}[\tau^{out}]$, and  the fact that $\mathbb{P}(F_i)<1$ for sufficiently large $A$ \eqref{shortprob}, we find that
\begin{align*}
\mathbb{E}^{(x_0,y_0)}[T_{2D}^A] \leq &\mathbb{E}[ \tau_{0}] + C\left(\frac{ L}{\nu A^{1/3}}+\frac{2L}{A   u_m}\right)\\
\leq& \mathbb{E}[\tau_0]+O(A^{-1/3}) =  \mathbb{E}^{y_0}[T_{1D}]+O( A^{-1/3}).%$% +\frac{L}{A u_m}
\end{align*}
%\textcolor{red}{$\mathbb{P}(F_i)<1$ is enough to make the argument work. Because $\mathbb{E}\tau^{out}\leq \frac{1}{A^{1/3}}$.}
Thus the upper bound approaches $\mathbb{E}^{y_0}[T_{1D}]$ as $A\rightarrow \infty$.

 Finally, we comment on the case where the starting position is in the egg zone level, i.e.,  $y_0\in[L/2-\delta,L/2+\delta]$. In this case, the $1$-dimensional first hitting time is zero. Hence the goal is to show that the average $2$-dimensional first hitting time converges to zero as $A\rightarrow\infty$. We distinguish between two possible cases: case a) $y_0\in [L/2-\delta,L/2+\delta]\backslash[L/2-A^{-1/5}, L/2+A^{-1/5}]$; case b) $y_0\in [L/2-A^{-1/5}, L/2+A^{-1/5}]$.

In the first case, without loss of generality, we focus on the upper component  $y_0\in(L/2+A^{-1/5}, L/2+\delta]$. We redefine the searching strip to be $S^+:=[y_0-2A^{-1/3} , y_0]$,  and note that the center level of $S^+$ is $y^+:=y_0-A^{-1/3}$. Applying the same argument as before, we have that the average first hitting time from $y_0$ to $y^+$ is bounded above by $\mathbb{E}\tau_0\leq CA^{-1/3}$. Inside the searching strip $S^+$, the assumptions $u(L/2)=0$, and \eqref{Strictly_Monotone} yields that  the absolute value of fluid velocity has positive lower bound if $A$ is large enough, i.e., \begin{align*}
Au_m:=&A\min_{y\in S^+}|u(y)|\geq Au_d(A^{-1/5}-2A^{-1/3})\geq A^{4/5}u_d/2.
\end{align*}  Similarly to the previous argument, we  consider the events $F_i:=\{\tau^{in}\leq \frac{4L}{A ^{4/5}  u_d}\}\cup\{\tau^{in}\geq \frac{4L}{A^{4/5}u_d},\, d_\rr(X_{\frac{4L}{A^{4/5}u_d}},X_e)\leq  L\}$, which are adjustments to definition \eqref{F_i}. Then the probability of $F_i$ can be estimated as follows. First of all,
\begin{align*}
\mathbb{P}\left(\tau^{in}\leq \frac{4L}{A ^{4/5}  u_d}\right)\leq 2\mathbb{P}\left(\sqrt{\nu}B_{\frac{4L}{A^{4/5}  u_d}}\geq A^{-1/3}\right) \leq 2\frac{\mathbb{E}\left[\left(B_{\frac{4L}{A^{4/5} u_d}}\right)^{2}\right]}{\nu^{-1}A^{-2/3}} \leq \frac{8\nu L}{A^{2/15}u_d}.
\end{align*}
Then
\begin{align*}
\mathbb{P}\left(\tau^{in}> \frac{4L}{A^{4/5}u_d},\, d_\rr(X_{\frac{4L}{A^{4/5}u_d}},X_e)\leq L\right)\leq \mathbb{P}\left(\sqrt{\nu}B_{ \frac{4L}{A^{4/5}u_d}}\geq  L\right)\leq \frac{4\nu}{A^{4/5} u_dL}.%\tau^{in} MATH
\end{align*}
Hence,
\begin{align*}
\mathbb{P}(F_i)\leq\frac{8\nu L}{A^{2/15}u_d}+\frac{4\nu}{A^{4/5} u_dL}.
\end{align*}
Thus $\lim_{A\rightarrow \infty}\mathbb{P}(F_i)=0$. Now the same iterative argument as above yields the result
\begin{align*}
0\leq\mathbb{E}^{(x_0,y_0)}[T_{2D}^A]\leq &\mathbb{E}[\tau_0]+\mathbb{P}\left(F_1^c\right)\frac{2L}{A^{4/5} u_d/2}+\sum_{i=1}^\infty\left(\prod_{j=1}^{i}\mathbb{P}\left(F_j\right)\right)\left(\mathbb{E}[\tau^{out}]+\mathbb{P}\left(F_{i+1}^c\right)\frac{2L}{A^{4/5} u_d/2}\right)\\
\leq &C(\nu, L, u_d)A^{-2/15}.
\end{align*}
Therefore, $\lim_{A\rightarrow \infty}\mathbb{E}^{(x_0,y_0)}[T_{2D}^A]=0$.%?

If $y_0\in [L/2-A^{-1/5}, L/2+A^{-1/5}]$, then we define the searching strip to be $S:=[L/2+A^{-1/5}-2A^{-1/3}, L/2+A^{-1/5}]$. The average first hitting time from the starting position to the center level $y^+:=L/2+A^{-1/5}-A^{-1/3}$ is bounded by $\mathbb{E}[\tau_0]\leq CA^{-1/5}$. The speed of the shear inside the searching strip is bounded from below by $Cu_dA^{4/5}$. Now the same argument as above completes the proof.
\ifx\fi

\end{proof}

\section{Proof of Theorem \ref{thm:Horizontally_homogenized}}

%\textcolor{black}{Taylor book, Kato book?}
We first present the enhanced dissipation estimates from the works \cite{BCZ15}, \cite{Wei18} and \cite{ColomboCotiZelatiWidmayer20} adapted to our large torus setting ($|\Torus|=L\gg 1$).

\begin{theorem} Consider solutions $\eta_{\sim}(t,x,y)$ to the passive scalar equations
\begin{align}
\pa_t \eta_{\sim}+Au\left(\frac{y}{L}\right)\pa_x \eta_{\sim}=\nu\de \eta_{\sim},\quad \eta_{\sim}(t=0,x,y)=\eta_{0;\sim}(x,y), \quad\, (x,y)\in \Torus^2\label{passive_scalar_1}
\end{align}
subject to initial data  $\eta_{0;\sim}\in L^\infty(\Torus^2)$ and zero average constraint $\int_{\Torus}\eta_{0;\sim}(x,y)dx=0,\,\forall y\in \Torus$.

\noindent
\textbf{Case a)} Assume that the shear flow profile $u$ is a Lipschitz function with finitely many critical points. Furthermore, if the derivative of the profile exists at a point $y\in\Torus$, then it is strictly bounded away from zero, i.e.,  $|u'(y)|>u_d>0$.
%{\color{red} This is extremely confusing. Strictly speaking, in the periodic case this can hold only if there are some points where the derivative does not exist,
%but then the minimum is undefined. I added a quick fix and a remark later, but may be better to discuss before the theorem what exactly do we mean by statements like this one} {\color{black} Anwser: I am not sure. Finitely many critical points assumption is natural even for the nondegenerate shear flow %setting. So I think making this assumption in the Lipschitz setting is not too bad. If we just do two kinks, then the main theorems can only work for $\sin(y)$-like flows, which is not as general as the current version. May I ask for your opinion?}
If the parameter $A$ is large enough in the sense that $\frac{\nu}{AL}\leq c_0$ for a small constant $c_0=c_0(u)>0$, then there exist  \textcolor{black}{constants $C,\, \kappa >0$ depending only on the shear profile $u$} such that the following enhanced dissipation estimate holds:
\begin{align}\label{ED}
|| \eta_{\sim}( t)||_\infty\leq C||\eta_{ 0;\sim}||_\infty e^{-\kappa \nu^{1/3}A^{2/3}L^{-4/3}|\log \frac{\nu}{AL}|^{-1}t},\quad \forall t\in [0,\infty).
\end{align}

\noindent
\textbf{Case b)} The shear flow profile $u\in C^2(\Torus)$ is non-degenerate in the sense that if $u'(y_0)=0$, then $u''(y_0)\neq 0$. Moreover, there are only finitely many critical points.
 If the parameter $A$ is large enough in the sense that $\frac{\nu}{AL}\leq c_0$ for a small constant $c_0= c_0(u)>0$, then there exist \textcolor{black}{constants $C,\, \kappa >0$ depending only on the shear profile $u$} such that the following enhanced dissipation estimate holds :
\begin{align}\label{ED_non_degenerate}
|| \eta_{\sim}( t)||_\infty\leq C||\eta_{ 0;\sim}||_\infty e^{-\kappa \nu^{1/2}A^{1/2}L^{-3/2} |\log \frac{\nu}{AL}|^{-1}t},\quad \forall t\in [0,\infty).
\end{align}
%\end{itemize} \ the mean-zero constraint
\end{theorem}
\begin{remark}
%1. Note that a Lipschitz function is a.e. differentiable by Rademacher theorem. The minimum in the case a) is taken over all points where the derivative exists. \\
%2.
We will need the estimate \eqref{ED} to derive a sufficiently strong bound on the chemical gradient $||\pa_y c_\sim||_\infty$. We are not able to show that the  chemical gradient converges to zero as $A\rightarrow \infty$ using only estimate  \eqref{ED_non_degenerate}.
\end{remark}
\begin{proof}
We divide the proof into several steps.

\noindent
\textbf{Step \# 1: Rescaling argument. }
If we rescale the variables in \eqref{passive_scalar_1} by setting $X:=\frac{x}{L}$, $Y:=\frac{y}{L}$ and $\tau=\frac{A}{L}t$, we end up with the following:
\begin{align}\label{PS_XY_coordinate}
\pa_\tau\wh \eta_{\sim}+u(Y)\pa_X\wh{\eta}_{\sim}=\wh \nu\de_{X,Y}\wh{\eta}_{\sim}, \quad \int_0^1 \wh \eta_\sim(x,y)dx=0,\quad (X,Y)\in \Torus^2.
\end{align}
Here $\wh \eta(\tau,X,Y):=\eta(t,x,y)$, and $\wh \nu:=\frac{\nu}{AL}$.
Hence if we obtain the following estimates:

\noindent case a)
\begin{align}\label{ED_case_a}
||\wh\eta_{\sim}(\tau)||_{L_{X,Y}^\infty}\leq C ||\wh\eta_{0;\sim}||_{L_{X,Y}^\infty} e^{-\kappa_{0}  \widehat{\nu}^{1/3}|\log\wh\nu|^{-1}\tau },\quad \forall \tau\in[0,\infty),%23
\end{align}

\noindent case b)
\begin{align}\label{ED_case_b}
||\wh\eta_{\sim}(\tau)||_{L_{X,Y}^\infty}\leq C||\wh\eta_{0;\sim}||_{L_{X,Y}^\infty} e^{-\kappa_0\widehat{\nu}^{1/2}|\log \wh \nu|^{-1}\tau},\quad \forall \tau\in[0,\infty)
\end{align}for some universal $\kappa_0,$  then by rescaling back to the original variables, we obtain \eqref{ED} and  \eqref{ED_non_degenerate}.

\noindent
\textbf{Step \# 2: $L^2$-estimates. } Consider the passive scalar equation \eqref{PS_XY_coordinate}.
\ifx on a unit torus $\Torus^2$,
\begin{align}
\pa_\tau \widehat \eta_{\sim}+u(Y)\pa_X \widehat\eta_{\sim}= \widehat{\nu} \de \widehat\eta_{\sim},\quad \widehat{\eta}_{\sim}(\tau=0,X,Y)=\widehat{\eta}_{0;\sim}(X,Y), \quad (X,Y)\in \Torus^2,
\end{align}
where the sizes of the torus and the shear profile is normalized, i.e., $|\Torus|=1,\, ||u||_\infty\leq 1$.
\fi
We will show that if the viscosity $\widehat{\nu}$ is small enough, i.e., $\widehat{\nu}\leq c_0(u)$ for some constant $c_0>0$ depending only on $u$, then the enhanced dissipation estimate holds: in case a),
\begin{align}
||\widehat {\eta}_{\sim}(\tau)||_{L^2}\leq C ||\widehat{\eta}_{0;\sim}||_{L^2} e^{-\kappa_{0}\widehat{\nu}^{1/3}\tau},\quad \forall \tau\in[0,\infty);\label{ED_literature_Lip}
\end{align} in case b),
\begin{align}
||\widehat {\eta}_{\sim}(\tau)||_{L^2}\leq C ||\widehat{\eta}_{0;\sim}||_{L^2} e^{-\kappa_{0}\widehat{\nu}^{1/2}\tau},\quad \forall \tau\in[0,\infty).\label{ED_literature_nondegenerate}
\end{align}
{\color{black}Here the constants $C,\, \kappa_0$ depend only on the shear profile $u$.} The estimate \eqref{ED_literature_nondegenerate} \textcolor{black}{appears in Theorem 1.1 of \cite{AlbrittonBeekieNovack21}}. We also refer the interested readers to \cite{BCZ15} and \cite{Wei18}. %\cite{BCZ15}We recall the enhance dissipation estimate derived in the papers ,  and \cite{AlbrittonBeekieNovack21}.

To prove the \eqref{ED_literature_Lip}, we first consider the mixed $x$-Fourier transform of the passive scalar equation
\begin{align*}
\pa_t \wh\eta_k=&\wh \nu \pa_{YY}\wh \eta_k-\wh \nu |k|^2 \wh \eta_k-iu(Y)k\wh \eta_k=:\mathcal{L}_k\wh\eta_k.
\end{align*}
\textcolor{black}{%To apply the main theorem of \cite{Wei18}, w
We also consider the following resolvent equation associated with $\mathcal{L}_k$:
\begin{align}
-\wh\nu\pa_{YY} w_k+\wh \nu |k|^2 w_k+ik(u(Y)-\lambda)w_k=F. \label{Resolvent}
\end{align}
To prove \eqref{ED_literature_Lip}, we will use the following inequality:
 for $\forall \lambda\in \rr$,
\begin{align}\label{Goal}
||w_k||_2^2\leq C\wh\nu^{-2/3}|k|^{-4/3}||F||_2^2.
\end{align}%\textcolor{red}
The constant $C$ depends only on the shear profile $u$, and is independent of $\lambda,\nu,|k|$. The explicit derivation of the connection between \eqref{Goal} and \eqref{ED_literature_Lip}
is carried out on pages 7-8 of the paper \cite{He21} and here we omit further details, other than note that the main theorem of \cite{Wei18} plays an important role.}
To derive the estimate \eqref{Goal}, we test the equation \eqref{Resolvent} with $\overline{w_k}$ and $-i(u-\lambda)k\overline{w}_k$ and take the real part to obtain the following bounds:
\begin{align}
\wh\nu||\pa_Y w_k||_2^2+\wh\nu|k|^2&||w_k||_2^2\leq ||F||_2||w_k||_2;\label{Real} \\
|k|^2\int_{\Torus}(u -\lambda)^2 |w_k|^2 dY=& -\text{Re}\left(ik\int _\Torus F(u -\lambda) \overline{w_k}dY\right)-\text{Re}\left(ik\wh\nu\int_\Torus \pa_Y w_k\overline{w_k}\pa_Y(u-\lambda)dY\right).
\end{align}%color{red}color{red}
Direct application of H\"older inequality and Young's inequality yields
\begin{align*}
|k|^2||(u-\lambda) w_k||_2^2\leq 4||F||_2^2+\frac {1}{2}|k|^2||(u-\lambda)w_k||_2^2+\wh\nu||u'||_\infty|k| ||\pa_Y w_k||_2||w_k||_2.
\end{align*}
After simplification, we obtain,\begin{align}
||(u-\lambda) w_k||_2^2\leq 8|k|^{-2}||F||_2^2+2\wh\nu ||u'||_\infty|k|^{-1}  ||\pa_Y w_k||_2||w_k||_2.
\label{Imagine_test}
\end{align}
Now we define the following partition of domain
\begin{align}
E:=\{Y||u(Y)-\lambda|\leq \wh\nu^{1/3}|k|^{-1/3}\}, \quad E^c:=\{Y| |u (Y)-\lambda |>\wh\nu^{1/3}|k|^{-1/3}\}.\quad
\end{align}
%{\color{black}Siming: I will delete: ``On the set $E$, if the derivative of the Lipschitz profile $u$ exists, then it is bounded away from zero, i.e., $|u'|\geq \frac{1}{C}>0$.  Moreover, the total number of critical points where the derivative does not exist is finite.''}
We claim that the size of the set $E$ is bounded by $\textcolor{black}{|E|}\leq C(u)\wh\nu^{1/3}|k|^{-1/3}$. Here the constant $C(u)$ depends on the Lipschitz norm of the shear profile, the minimum of $|u'|(y)$ (whenever it exists) and the total number of critical points of $u$.
%{\color{red} Siming, this is false. Imagine a Lip function that is piecewise linear and slope equals one everywhere where it is defined. Given any $\epsilon,$ you can find such Lip function
%that is within some fixed value for all $y$, it just needs to look like a saw. Using this, you can build lots of counterexamples for your claim} {\color{black}
Specifically, if there are only finitely many critical points ($N$), then the total area of $E$ cannot exceed $C\frac{N^2\wh\nu^{1/3}|k|^{-1/3}}{\min\{|u'(y)||y\in \Torus,\,u'(y) \text{ exists. }\}}$.
\textcolor{black}{
The proof of the above claim is as follows. There are three possible scenarios: a) $\text{dist}(\lambda, \text{Range}(u))>\widehat{\nu}^{1/3} |k|^{-1/3}$; b) $\text{dist}(\lambda, \text{Range}(u))\in(0,\widehat{\nu}^{1/3}|k|^{-1/3}]$; c) $\text{dist}(\lambda, \text{Range}(u))=0$. In scenario a), the set $E$ is empty, so the bound holds trivially. In scenario b), by definition of $N$, there can be at most $N$ critical points $y_i$ in the set $\{y||u(y)-\lambda|\leq \widehat{\nu}^{1/3}|k|^{-1/3}\}$. Around each such critical point $y_i$, there is a connected component $F(y_i)$ of the set $\{y||u(y)-\lambda|\leq \widehat{\nu}^{1/3}|k|^{-1/3}\}$ enclosing $y_i$. The total number of the connected components $F(y_i)$ is bounded by $N$. Note that the connected component $F(y_i)$ can contain other critical points, but there can be at most $N$ of them. Further recall that if the derivative of $u$ exists, $|u'|\geq \frac{1}{C}$. As a result, the size of each connected component $F(y_i)$ is at most $N C\widehat{\nu}^{1/3}|k|^{-1/3}$. Thus, summing up the lengths of all connected components, we obtained the bound $|E|\leq C\frac{N^2\widehat\nu^{1/3}|k|^{-1/3}}{\min\{|u'(y)||y\in \mathbb{T} ,\,u'(y) \text{ exists.}\}} $. We note that a more careful accounting would reduce $N^2$ in the bound to $N,$ but we do not pursue it for simplicity.
In the last scenario, we can consider the intersection points $z_i$ such that $u(z_i)=\lambda$. There can be at most $N$ of these points, since the number of times a profile $u$ can cross a given value is bounded
by the number of critical points. Around each intersection point $z_i$, we can consider the connected component $G(z_i)$ of the set $\{y||u(y)-\lambda|\leq \widehat{\nu}^{1/3}|k|^{-1/3}\}$. The lengths of the components $G(z_i)$ are then estimated similarly to $F(y_i)$ in scenario b), arriving  at the same bound $|E|\leq C\frac{N^2\widehat\nu^{1/3}|k|^{-1/3}}{\min\{|u'(y)||y\in \mathbb{T} ,\,u'(y) \text{ exists.}\}}$.
}

To estimate $||w_k||_{L^2(E)}^2$, we use the Gagliardo-Nirenberg inequality, and then the estimate \eqref{Real} to get that
%\textcolor{red}{Rmk: And the constant $C$ is universal constant. Drop $GN$. and $C$ changes from line to line. and $C_{\infty}, C_{H^1}$ should be dropped. }\textcolor{black}{I have added the notation explanation in section 1.2. And the $C_{...}$ are simplified. }
\begin{align}
||w_k||_{L^2(E)}^2\leq &C\wh\nu^{1/3}|k|^{-1/3}||w_k||_{L^\infty(\Torus)}^2\\
\leq &C \wh\nu^{1/3}|k|^{-1/3}||w_k||_{L^2(\Torus)}||\pa_Y w_k||_{L^2(\Torus)}+C\wh\nu^{1/3}|k|^{-1/3}||w_k||_{L^2(\Torus)}^2\\
\leq& CB \wh\nu^{2/3}|k|^{-2/3}||\pa_Y w_k||_{L^2(\Torus)}^2+\left(\frac{1}{B}+C\wh\nu^{1/3}|k|^{-1/3}\right)||w_k||_{L^2(\Torus)}^2\\
\leq& CB \wh\nu^{-1/3}|k|^{-2/3}||F||_{L^2(\Torus)}||w_k||_{L^2(\Torus)}+\left(\frac{1}{B}+C\wh\nu^{1/3}\right)||w_k||_{L^2(\Torus)}^2\\
\leq&{C(B) }\wh\nu^{-2/3}|k|^{-4/3}||F||_{L^2(\Torus)}^2+\left(\frac{2}{B}+C\wh\nu^{1/3}\right)||w_k||_{L^2(\Torus)}^2.\label{E_est}
\end{align}%{\color{black} HS: I have replaced all the $||\cdot||_2$ by $||\cdot||_{L^2(\Torus)}$. I also do the same for the next two paragraphs.}
Next we estimate the contribution from the $E^c$ region. We apply the relations \eqref{Real},  \eqref{Imagine_test} to estimate
\begin{align}
||w_k||_{L^2(E^c)}^2\leq &C\wh\nu^{-2/3}|k|^{ 2/3}\int_{\Torus} (u-\lambda)^2 |w_k|^2dY\\
\leq&C\wh\nu^{-2/3}|k|^{-4/3}||F||_{L^2(\Torus)}^2+{C }\wh\nu^{1/3}|k|^{-1/3}||u' ||_{L^\infty(\Torus)}||\pa_Y w_k||_{L^2(\Torus)}||w_k||_{L^2(\Torus)}\\
\leq& C\wh\nu^{-2/3}|k|^{-4/3}||F||_{L^2(\Torus)}^2+{CB}\wh\nu^{2/3}|k|^{-2/3}||u'||_{L^\infty(\Torus)}^2||\pa_Y w_k||_{L^2(\Torus)}^2+\frac{1}{B}||w_k||_{L^2(\Torus)}^2\\
\leq& C\wh\nu^{-2/3}|k|^{-4/3}||F||_{L^2(\Torus)}^2+{CB }\wh\nu^{-1/3}|k|^{-2/3}||u'||_{L^\infty(\Torus)}^2||F||_{L^2(\Torus)}||w_k||_{L^2(\Torus)}+\frac{1}{B}||w_k||_{L^2(\Torus)}^2\\
\leq&{C(B)}\wh\nu^{-2/3}|k|^{-4/3}||u'||_{L^\infty(\Torus)}^{{4}}||F||_{L^2(\Torus)}^2 +\frac{2}{B}||w_k||_{L^2(\Torus)}^2.\label{E_c_est}
\end{align}
Combining \eqref{E_est} and \eqref{E_c_est}, we obtain that
\begin{align}
||w_k&||_{L^2(\Torus)}^2=||w_k||_{L^2(E)}^2+||w_k||_{L^2(E^c)}^2\nonumber\\
\leq& C(B){(1+||u'||_{L^\infty(\Torus)}^4)} \wh\nu^{-2/3}|k|^{-4/3}||F||_{L^2(\Torus)}^2+{\left(\frac{4}{B}+C\wh\nu^{1/3}\right)}||w_k||_{L^2(\Torus)}^2.
\end{align}%\frac{1}{B}
Now choosing $B\geq 8$ and $\wh\nu$ small enough  yields the estimate \eqref{Goal} and hence \eqref{ED_literature_Lip}. This concludes Step \# 2.
%Through a standard rescaling argument, one obtains the estimate \eqref{ED}.

 \noindent
 \textbf{Step \# 3: $L^\infty$-enhanced dissipation estimate. }
We derive an  $L^2$-$L^\infty$-estimate of the passive scalar semigroup $S_{s;s+\tau}$, which represents the solution operator of the equation \eqref{PS_XY_coordinate} from time $s$ to $s+\tau$. Consider the time interval $[s, s+\kappa^{-1}d(\wh{\nu})^{-1}|\log \wh{\nu}|]$, where $\kappa\in(0,\kappa_0)$ denotes a constant \textcolor{black}{($\kappa_0$ is defined in \eqref{ED_literature_Lip}, \eqref{ED_literature_nondegenerate})} and $d(\wh\nu)=\wh\nu^{1/3}$ in case a) and {\color{black} $d(\wh \nu)=\wh \nu^{1/2}$ in case b)}. First we  prove the following estimate for passive scalar equation% should be $d(\wh\nu)=\wh\nu^{1/2}|\log \wh\nu|^{-2}$
\begin{align}\label{Nash}
||S_{s,s+\tau} \eta_{\sim}(s)||_\infty\leq \frac{C }{(\wh{\nu} \tau)^{1/2}}||\eta_{\sim}(s)||_2.
\end{align}
The proof of this estimate \eqref{Nash} is a combination of Nash inequality and a duality argument. \textcolor{black}{We refer the interested readers to the proof of Lemma 3.1 and 3.3 in \cite{FannjiangKiselevRyzhik06}.} %follow the  the proof of equation (4.4) in \cite{IyerXuZlatos}, \cite{CKRZ08}, Lemma 5.4 in \cite{Zlatos2010}. or Lemma 5.4 in \cite{Zlatos10}
Now we decompose the interval $[s, s+ \kappa^{-1}d(\wh{\nu})^{-1}|\log \wh{\nu}|]$ into two equal-length sub-intervals and  apply the estimates \eqref{ED_literature_Lip} and \eqref{ED_literature_nondegenerate} to derive the following:
\begin{align}
||S&_{s,s+ \kappa^{-1}d(\wh{\nu})^{-1}|\log \wh{\nu}|}\eta  _{\sim}(s)||_\infty\\
=&||S_{s+\frac{1}{2}\kappa ^{-1}d(\wh{\nu})^{-1}|\log \wh{\nu}|,s+ \kappa ^{-1}d(\wh{\nu})^{-1}|\log \wh{\nu}|}S_{s,s+\frac{1}{2} \kappa ^{-1}d(\wh{\nu})^{-1}|\log \wh{\nu}|}\eta_{\sim}(s)||_\infty\\
\leq &\frac{C }{(\wh{\nu} \frac{1}{2}\kappa ^{-1}d(\wh{\nu})^{-1}|\log \wh{\nu}|)^{1/2}}||S_{s,s+\frac{1}{2} \kappa ^{-1}d(\wh{\nu})^{-1}|\log \wh{\nu}|}\eta_{\sim}(s)||_2\\
\leq&\frac{C}{( \kappa ^{-1}\wh{\nu} d(\wh{\nu})^{-1}|\log\wh{\nu}|)^{1/2}}||\eta_{\sim}(s)||_2 {e^{-\frac{1}{2}\kappa_0  d(\wh{\nu}) \kappa ^{-1}d(\wh{\nu})^{-1}|\log \wh{\nu}|}}\\
\leq &\frac{1}{64}||\eta_{\sim}(s)||_2{\leq \frac{1}{2}  ||\eta_{\sim}(s)||_\infty}.\label{L2_Linfty_semigroup}%\textcolor{red}3
\end{align}
In the last line, we choose $\kappa$ and then $\wh{\nu}$ small enough compared to universal constants so that the coefficient is small. We further note that the $L^\infty$-norm of $\eta_{\sim}$ is dissipative along the dynamics. To conclude, we iterate the argument on consecutive intervals to derive the estimate.
\end{proof}
To prove Theorem \ref{thm:Horizontally_homogenized}, we also need a useful formula.
\begin{lem}\label{lem:integral_formula}
a) Consider the elliptic equation on $\Torus^2$
\begin{align}\label{eq_1}
-\mathcal{L}\rho=\mathcal{S},\quad \mathcal{L}=  \de -Au(y)\pa_x - \mathbf{1},
\end{align}
with $u \in C^1(\Torus^2)$ and $\mathcal{S} \in C(\Torus^2).$
The solution $\rho$ can be  represented as follows:
\begin{align}\label{formula}
\rho=\int_0^\infty e^{t\mathcal{L}}\mathcal{S} dt.
\end{align}
Here $e^{t\mathcal{L}}$ is the semigroup generated by the operator $\mathcal{L}$.

b) Consider the solution $\rho$ to the equation on $ \Torus \times [0,\LL]$,
\begin{align}\label{eq_2}
- \cL^\dagger\rho = \mathcal{S},\,&\quad\mathcal{L}^\dagger:=\de +Au(y)\pa_x,\quad
\rho(x,y=\{0,\mathbb{L}\})=0,
\end{align}
where $u \in C^1(\Torus \times [0,\LL])$ and $\mathcal{S} \in C(\Torus \times [0,\LL]).$
%Here $\mathcal{L}^\dagger$ is the differential operator $\frac{1}{2}\nu \de+Au(y)\pa_x$ subject to Dirichlet boundary condition and
Let $e^{t\cL^\dagger}$ be the semigroup associated with $\cL^\dagger$. %If the forcing $\mathcal{S}$ is regular enough, i.e., $\mathcal{S}\in C^\infty(\Torus\times [0,\mathbb{L}])$\textcolor{black}{(??)},
Then the solution can be represented as follows:
\begin{align}
\rho=\int_0^\infty e^{t\mathcal{L	^\dagger}}\mathcal{S}dt.\label{formula_[0,L]}
\end{align}
\end{lem}
\begin{proof} Let us first prove \eqref{formula}. Note that $e^{t\mathcal{L}} \mathcal{S}$ is the solution to the passive scalar equation
\begin{align}\label{pseq}
\pa_t (e^{t\mathcal{L}}\mathcal{S})=\cL (e^{t\mathcal{L}}\mathcal{S}), \quad e^{0\mathcal{L}}\mathcal{S}=\mathcal{S}.%n_\sim n_\sim n_\sim n_\sim  n_\sim
\end{align}
Note that $e^{t\cL}$ is $C^2(\Torus^2)$ for every $t>0$ by parabolic regularity.
By considering the time evolution of the maximum value, we observe that the solutions to the passive scalar equation decay to zero exponentially in time for all $(x,y)\in\Torus^2$.
%i.e., $e^{\infty\mathcal{L}}\mathcal{S}=0$.
By integrating the above equation in time on both sides, we obtain that %and use the fact that ,
\begin{align*}
0-\mathcal{S}=&\int_0^\infty \cL(e^{t\cL}\mathcal{S}) dt=\cL\left(\int_0^\infty e^{t\cL}\mathcal{S} dt\right)=\cL \rho.
\end{align*}%n_\simn_\simn_\simc_\simc_\simc_\simn_\sim
Hence $\rho$ solves $-\cL \rho =\mathcal{S}$.

%\textcolor{black}{E. B. (Brian) Davis.  Heat semigroup. }
Next we consider case b). Since  the first eigenvalue of the differential operator $-\mathcal{L}^\dagger$ defined on the domain $\Torus\times [0,\mathbb{L}]$
with Dirichlet boundary conditions at $y=0,\LL$ is strictly positive, % (\textcolor{red}{No gap here, say the first eigenvalue is strictly positive.})
we have that the solutions to the passive scalar equation associated with $\mathcal{L}^\dagger$ decay to zero exponentially in time.
The convergence to the initial data as $t\rightarrow 0^+$ is trickier as $\mathcal{S}$ does not have to satisfy the Dirichlet boundary condition.
Nevertheless, it is true that
\begin{equation}\label{idconv}
e^{t \cL^\dagger} \stackrel{t \rightarrow 0^+}{\longrightarrow} \mathcal{S}
\end{equation}
% as $t\rightarrow 0^+$
 for every $(x,y) \in \Torus \times (0,\LL).$
One way to prove this is to use Feynman-Kac formula. Let, as before, $(X_t,Y_t)$ be the diffusion process
$dX_t = Au(Y_t)dt+ \sqrt{\nu}dB^{(1)}_t,$ $dY_t=  \sqrt{\nu}dB^{(2)}_t.$ Then the solution of the passive scalar equation \eqref{pseq}
satisfies
\begin{equation}\label{FK}
(e^{t\mathcal{L}}\mathcal{S})(x,y,t) = \mathbb{E}^{(x,y)}
\left[ \mathcal{S}(X_t, Y_t) 1_{t<\tau^{(x,y)}} \right],
\end{equation}
where $\tau^{(x,y)}$ is the hitting time of the Dirichlet boundary $y = 0, \LL$ for a trajectory $(X_t, Y_t)$ starting at $(x,y).$ The formula \eqref{FK} is certainly well-known, though we could not find a convenient direct reference for it. It is not difficult to derive from its variant involving a potential rather than Dirichlet boundary condition \cite{Oksendal}
by taking the potential to be constant outside our domain and taking this constant to infinity.
On the other hand, the formula \eqref{FK} implies \eqref{idconv} via elementary estimates provided that $\mathcal{S}$ is continuous.
%Since the initial data $\mathcal{S}$ is regular, the %convergence of the semigroup to the initial data as $t\rightarrow 0^+$ is justified.
Hence we apply the same argument as in the proof of integration formula \eqref{formula} to derive \eqref{formula_[0,L]}.

% \textcolor{black}{
% \begin{remark} In case b), if the $\mathcal{S}$ is not zero at the two boundaries $y=0,\mathbb{L}$, we can approximate $\mathcal{S}$ by $\mathcal{S}_{in}\in C^\infty (\Torus\times [0,\mathbb{L}])$, such that $\mathcal{S}_{in}$ vanishes on the boundary, and $||\mathcal{S}_{in}-\mathcal{S}||_{L^2}\leq %\frac{1}{A}$. By the linearity of the equation, we have that $\rho=\rho_1+\rho_2$ such that
% \begin{align}
% \mathcal{L}^\dagger \rho_1=\mathcal{S}_{in},  \mathcal{L}^\dagger \rho_2=(\mathcal{S}-\mathcal{S}_{in}).
% \end{align}
% By energy estimate, we have $||\rho_2||_2^2\leq \frac{C}{A}$. And the $\rho_1$ part can be estimated through the estimate \eqref{ED_channel} as in \eqref{T_sim_2}.
%\end{remark}
%}
 \ifx \textcolor{black}{Check: I am not sure whether we need to worry about the convergence of the semigroup $e^{t\mathcal{L}^\dagger}$ as time approaches zero? Probably it will converge to the initial data in the $L^2$-sense so that the integral formula holds in terms of $L^2$? Note that the right hand side of \eqref{TT_sim_eq} might not vanish at the boundary so the pointwise convergence might be wrong. One idea ($in$ = interior, $b$=boundary)
\begin{align}
\frac{1}{2}\nu \de \TT_\sim+Au(y)\pa_x \TT_\sim= RHS=F\eta_{in;F}+F\eta_{b;F},\quad \eta_{in;F}+\eta_{b;F}=1, \\
\eta_{in;F},\eta_{b;F}\in C^\infty,\,||F\eta_{b;F}||_{L^2}\leq A^{-100}.
\end{align}
Now $\TT_\sim=\TT_{\sim;in}+\TT_{\sim;b}$, and \begin{align}
\frac{1}{2}\nu \de \TT_{\sim;in}+Au(y)\pa_x \TT_{\sim;in}=F\eta_{in};\\
\frac{1}{2}\nu \de \TT_{\sim;b}+Au(y)\pa_x T_{\sim;b}=F\eta_{b;F}.
\end{align}
Now the $\TT_{\sim;in}$ can be estimated through the original argument. And standard energy estimate yields that $||\TT_{\sim;b}||_2\leq A^{-100}$.}\fi
\end{proof}

\begin{proof}[Proof of Theorem \ref{thm:Horizontally_homogenized}]
\ifx%%%%%%% New sketch:
To apply the aforementioned theorem to the finite channel, we apply the even reflection of the data $n,c$ and the shear profile with respect to the upper boundary and extend it to the torus. As a result, the strictly monotone shear profile becomes Lipschitz. Next we estimate the solution to the elliptic problem
\begin{align*}
-\de c_\sim + Au(y)\pa_x c_\sim=n_\sim-c_\sim.
\end{align*}
Define the operator $L$ to be
\begin{align*}
L:=\de -Au(y)\pa_x -\mathbf{1}_d.
\end{align*}
We rewrite the solution $c_{\sim}$ as follows
\begin{align*}
c_{\sim}=\int_0^\infty e^{tL} n_{\sim}dt.
\end{align*}
By rescaling in time, we observe that the enhanced dissipation  estimate \eqref{ED} also holds for the system $\pa_t \eta_\sim=L\eta_\sim$.  Hence we obtain that
\begin{align*}
||c_{\sim}||_2\leq \int_0^\infty|| e^{t L}||_{L^2\rightarrow L^2}||n_{\sim}||_2 dt\leq\frac{CL^{4/3}}{\nu^{1/3}A^{2/3}}||n_{\sim}||_2.% /\infty\infty\infty\infty \infty
\end{align*}
Similarly, the $\pa_x c_\sim$ has the same bound. Now we estimate the $\pa_y c_\sim$. Note that it solves the equation
\begin{align*}
\de \pa_y c_\sim+Au(y)\pa_x\pa_y c_\sim+A u'(y)\pa_x c_\sim =\pa_y n_\sim-\pa_y c_\sim.
\end{align*} Applying the same arguments as above yields that
\begin{align*}
||\na c_{\sim}||_2=&\norm{\int_0^\infty e^{tL}\left(-Au'(y)\pa_x c_\sim+\pa_y n_\sim\right)dt }_2\\
\leq &\int_0^\infty||e^{tL}||_{L^2\rightarrow L^2}\left(A||u'||_\infty\frac{L^{4/3}||\pa_x n_\sim||_2}{\nu^{1/3}A^{2/3}}+ ||\pa_y n_\sim||_2 \right) dt \\
\leq&\frac{C(\nu,L,||u'||_\infty)}{A^{1/3}}||\na n_{\sim}||_2. %\infty|\log A|^{2}\infty
\end{align*}
Hence the heterogeneous part of the $\na c$ is small. This concludes the proof.
\fi%%%%%%
We divide the proof into two steps.

\noindent
\textbf{Step \# 1: Estimate of the solution $h_\sim:=e^{t\mathcal{L}}n_\sim$ of the evolution equation. }
%Here we use the separation of time scale to derive the $A^{2/3}$ suppression with the non-degenerate shear flow.
First of all, we consider the equation
\begin{align}\label{PS_sim}
\pa_t h_{\sim}+Au(y)\pa_x h_\sim= \de h_\sim-h_\sim,\quad h_\sim(t=0,\cdot)=n_\sim(\cdot),
\end{align}
and an approximate system
\begin{align*}
\pa_t h_1+Au_1(y)\pa_x h_1= \de h_1-h_1, \quad h_1(t=0, \cdot)=n_\sim(\cdot).
\end{align*}
Here the Lipschitz shear profile $u_1(y)$ is identical to $u(y)$ near the egg zone, i.e., $y\in[\frac{L}{2}-2\delta,\frac{L}{2}+2\delta]$. The profile $u_1$ only differs from the original shear profile $u$ near the critical points of $u:$
we replace every critical point with a piecewise linear profile. In particular, we choose $u_1(y)$ such that $|u'_1(y)|\geq \frac{1}{C}>0$ for every $y$ where the derivative exists (and it may fail to exist only in a finite
number of points coinciding with the critical points of $u$). Now we compare the two solutions $h_\sim$ and $h_1$,
\begin{align}
\pa_t (h_\sim-h_1)+A(u(y)-u_1(y))\pa_x h_\sim+&Au_1(y)\pa_x(h_\sim-h_1)=\de (h_\sim-h_1)-(h_\sim-h_1), \nb\\
\label{h_sim-h_1_eq}\quad &h_\sim(0,\cdot)-h_1(0,\cdot)=0.
\end{align}
Here we show that for $t\leq A^{-1/4}$, the difference is small. First, let us establish
 that the time integration of the contribution $(u-u_1)\pa_x h_\sim$ is small on this period.
 Note that the difference $u(y)-u_1(y)$ is supported away from the initial data and the diffusion is limited by smallness of the time interval $[0, A^{-1/4}]$; hence one expects that this term is small.

\ifx
\textcolor{black}{To rigorously derive the decay, we consider the quantity $Q:=\lan (\pa_xh_\sim)^2\ran$. Direct calculation yields that the quantity $Q$ solves the following equation:
\begin{align*}
\pa_t Q=\nu\de_y Q-2\nu\lan (\pa_y Q)^2\ran.
\end{align*}\
By comparison principle, $Q$ is dominated by the solution to the heat equation $\pa_t H=\nu\de_y H$ subject to initial data $H(0,\cdot)=\lan (\pa_x n_\sim)^2\ran$.  }
\fi

To rigorously derive the decay, we consider the equation \eqref{PS_sim} on the universal cover $L\Torus\times\rr$ and use $h_\sim^c$ and $n_\sim^c$ to denote the solution and the initial data.
Taking the horizontal Fourier transform of the equation leads to %\textcolor{black}{(Check the Fourier transform constants!)} \eqref{PS_sim} and end up with
\begin{align*}
\pa_t {  k} \wh h_{k}^c+Au(y) \frac{2\pi i  k}{L}k \wh h_{k} ^c=-\frac{4\pi^2k^2}{L^2} {  k} \wh h_{k}^c+ {   k} \pa_{yy}\wh h_{k}^c- k \wh h_k^c. %\frac{i2\pi k}{L}
\end{align*}
%\textcolor{red}{Why do we need this faster decay. (A. 62) is the key.  Explain! I added a remark after Theorem 4.}
If we calculate the time evolution of $|k\wh h_{k}^c|^2$, we obtain that
\begin{align*}
\pa_t |k\wh h_{k}^c|^2=&-\frac{8\pi^2k^2}{L^2}|k\wh h_{k}^c|^2+\pa_{yy}|k\wh h_{k}^c|^2-2|k\pa_y \wh h_{k}^c|^2-2|k\wh h_{k}^c|^2\\
\leq& -\frac{8\pi^2 k^2}{L^2}|k\wh h_{k}^c|^2+\pa_{yy}|k\wh h_{k}^c|^2-2|k\wh h_{k}^c|^2.
\end{align*}
Since the solution is positive, by comparison principle, we have that
\begin{align}\label{point_wise_bound}
|k \wh h_{k}^c|^2(t,y )\leq e^{-(\frac{8\pi^2 |k|^2}{L^2}+2)t}\frac{1}{\sqrt{4\pi t}}\int_\Real e^{-\frac{|z-y|^2}{4 t}}|k\wh n_{k}^c|^2(z)dz.%2 1 |k|^2
\end{align}
For all $A \geq 1,$ $\textcolor{black}{0<}t\leq A^{-1/4}$ and all $y$ such that $\mathrm{dist}_\rr(y,\mathrm{support}|\wh n_k^c|)\geq \delta$, we \textcolor{black}{apply the monotone convergence theorem, the fact that the size of the target is small $\delta\leq L/8$, and the periodicity of $\wh n_k^c$ to obtain} the following bound: %Lebesgue dominated   for any $A>1$:%L/C_0  L^2  L^{2 }  L L
{\color{black}\begin{align*}
|k \wh h_{k}^c|& (t,y )\leq\sqrt{e^{-(\frac{8\pi^2|k|^2}{L^2}+2)t}\frac{1}{\sqrt{4\pi t}}\sum_{\ell=-\infty}^\infty\int_{0}^L e^{-\frac{|z+\ell L-y|^2}{4t}}|k \wh n_k|^2(z)dz}\\
 \leq& C(\delta, L)e^{-\frac{1}{2}(\frac{8\pi ^2 |k|^2}{L^2}+2)t}\frac{1}{\sqrt[4]{4\pi t}}e^{-\frac{\delta^2}{8t}}||k \wh n_k||_{L^2_y(\Torus)}\leq C(\delta,L)e^{- \frac{ \pi |k| \delta}{L}}\frac{A^{1/16}}{\sqrt[4]{4\pi}}e^{-\frac{\delta^2}{16}A^{1/4}}||k\wh n_k||_{L_y^2(\Torus)} .
\end{align*}}\textcolor{black}%{\textbf{(Siming: There are some problems with the constant in front of $\delta$. I change all the numbers here.)}}
{\color{black} Here in the last step, we apply the relation $\frac{4\pi^2|k|^2}{L^2}t+\frac{\delta^2}{16 t}\geq  \frac{2\sqrt{4}\pi |k|\sqrt{t}}{L}\frac{\delta}{\sqrt{16}\sqrt{t}} \geq  \frac{\pi|k|\delta}{L}$. }
The above estimate holds, in particular, for all $y$ such that $\mathrm{dist}_\Torus(y,\{z|(x,z)\in\mathrm{supp}(n_\sim)\}) \geq \delta >0$.
Now the $L^\infty_x$-norm of $\pa_x h_\sim^c$ can be estimated as follows:
\begin{align}\label{pa_x_h_f_infty_est_away_from_support}
||\pa_x h_\sim^c(\cdot ,y)||_{L_x^\infty}\leq &C\sum_{k\neq 0}|k\wh h_{k}^c(y)|\leq \frac{CA^{1/16}}{e^{C^{-1}\delta^2 A^{1/4}}}\sum_k {e^{- \frac{ \pi |k|\delta}{L}}}
%e^{- \frac{ \pi |k| }{4C_0^2}}
||k \wh n_k||_{L^2_y}\leq \frac{C A^{1/16}L^2}{\delta^2} e^{-C^{-1}\delta^2 A^{1/4}}|| n_\sim||_{L^2(\Torus^2)} ,\\
\quad& \mathrm{dist}_\Torus(y,\{z|(x,z)\in\mathrm{supp}(n_\sim)\}) \geq \delta >0, \quad \forall t\leq A^{-1/4}.
\end{align}
Now combining \eqref{pa_x_h_f_infty_est_away_from_support}, and the fact that $|u(y)-u_1(y)|$ is non-zero only for $\mathrm{dist}_\Torus(y,\{z|(x,z)\in \mathrm{supp}(n_\sim)\})\geq \delta>0$,  we have that
\begin{align*}
|A (u(y)-u_1(y))\pa_x h_\sim(y) |\leq \frac{C A^{17/16}L^2}{\delta^2} e^{-C^{-1}\delta^2 A^{1/4}}||n_\sim||_{L^2(\Torus^2)},\quad \forall y\in \Torus.
\end{align*}
From the equation \eqref{h_sim-h_1_eq} and direct application of comparison principle, we have that for $t\leq A^{-1/4}$, the difference $h_\sim -h_1$ is small:
\begin{align*}
||h_\sim -h_1||_{L_{x,y}^\infty}(t)\leq& \textcolor{black}{\int_0^t\|A (u(\cdot)-u_1(\cdot))\pa_x h_\sim(s,\cdot) \|_{L_y^\infty(\Torus)} ds}\\
\leq& A^{-1/4}  C L^2 \delta^{-4} A^{17/16}e^{-C^{-1}\delta^2 A^{1/4}}||n_\sim||_{L^2(\Torus^2)}\\
\leq &C L^2 \delta^{-4} e^{-C^{-1}\delta^2 A^{1/4}}||n_\sim||_{L^2(\Torus^2)},\quad \forall t\in[0, A^{-1/4}].
\end{align*}
Recalling that the $||h_1 ||_{L_{x,y}^\infty}$ undergoes enhanced dissipation with rate $|\log A|^{-1}A^{2/3}$ \eqref{ED},  we have obtained the following estimate
\begin{align}\label{h_sim_t_leq_A1/4}
||h_{\sim}||_\infty(t) \leq &||h_1||_{L_{x,y}^\infty}(t)+||h_\sim - h_1||_{L_{x,y}^\infty}(t)\\
\leq &C||n_\sim||_{\infty} e^{-\kappa \frac{\nu^{1/3}A^{2/3}}{L^{4/3}}|\log \frac{\nu}{AL}|^{-1}t}+C(L,\delta) e^{-C^{-1}\delta^2 A^{1/4}}||n_\sim||_{L^2(\Torus^2)},\quad \forall t\in[0,A^{-1/4}].
\end{align}
For $t\geq A^{-1/4}$, we apply the enhanced dissipation estimate \eqref{ED_case_b} for non-degenerate shear flow
\begin{align}\label{h_sim_t_geq_A1/4}
||h_{\sim}(t)||_\infty \leq C ||n_\sim||_\infty e^{-\kappa |\log\frac{\nu}{ AL}|^{-1}\frac{\nu^{1/2}A^{1/2}}{L^{3/2}}t},\quad \forall t\in[A^{-1/4},\infty).
\end{align}

\noindent
\textbf{Step \# 2: Estimate of the chemical $c_\sim$.}
Now we apply the formula $c_\sim =\int_0^\infty e^{t\cL} n_\sim dt$, $\cL=\frac{1}{2}\nu\de-Au(y)\pa_x-1$, and the estimates \eqref{h_sim_t_leq_A1/4}, \eqref{h_sim_t_geq_A1/4} to derive the following
\begin{align*}
||c_\sim||_\infty\leq& \left(\int_0^{A^{-1/4}}+\int_{A^{-1/4}}^\infty \right) ||e^{t\cL}n_\sim||_\infty dt\\
\leq& C\int_0^{A^{-1/4}}\left( ||n_\sim||_\infty e^{-\kappa \frac{\nu^{1/3}A^{2/3}}{L^{4/3}}|\log \frac{\nu}{AL}|^{-1} t}+ L^2 \delta^{-2}e^{-C^{-1}\delta^2 A^{1/4}}|| n_\sim||_2 \right) dt\\
&+C\int_{A^{-1/4}}^\infty ||n_\sim||_\infty e^{-\kappa \frac{\nu^{1/2}A^{1/2}}{L^{3/2}}|\log\frac{\nu}{AL}|^{-1}t}dt \\
\leq & \frac{C(L,\delta,\nu)|\log A|}{A^{2/3}} ||n_\sim||_\infty.
\end{align*}
%In the last line, we pick $A$ large enough compared to $\nu,L$.
Similar argument yields that
\begin{align*}
||\pa_x ^{i+1}\ c_\sim||_\infty \leq \frac{C(\nu,\delta,L)|\log A|}{A^{2/3}} ||\pa_x^{i+1} n_\sim||_\infty,\quad i\in\{0,1\}.
\end{align*}
The norm $||\pa_y c_\sim||_\infty$ can be estimated similarly. Applying the enhanced dissipation estimate of the non-degenerate shear flow \eqref{ED_non_degenerate}, we have that for $A$ large enough,
\begin{align*}
||\pa_y c_\sim||_\infty \leq \int_0^\infty ||e^{t\cL}(-Au'(y)\pa_x c_\sim+\pa_y n_\sim)||_\infty dt\leq\frac{C|\log A|^2}{A^{1/6}} ||\na n_\sim||_\infty.
\end{align*}
Similar arguments yields the estimate
\begin{align*}
||\pa_x\pa_yc_\sim||_\infty\leq\frac{C|\log A|^2}{A^{1/6}} ||\pa_x\na n_\sim||_\infty.
\end{align*}
This concludes the proof of \eqref{Quantitative_Horizontally_homogenization}.
%R, and we consider the simplified SDE The simplified SDE can be estimated with the argument below.
\end{proof}

\section{Proof of Theorem \ref{theorem_first_hitting_time_full_system}}\label{Sec:Proof_thm_3}

Before proving Theorem \ref{theorem_first_hitting_time_full_system}, we establish an auxiliary convergence result concerning solutions of partial differential equations
with Dirichlet boundary conditions. Consider the solutions of the 2D elliptic system
\begin{align}\label{T_eq}
\nu\de \TT+V\cdot \na \TT+Au(y)\pa_x \TT=-1,\\
\TT(x,y=0)=\TT(x,y=\LL)=0.
\end{align}
%\textcolor{black}{There might be sign problem in front of $V$ and $A$. But I skip it for the moment. (It seems to be fine right now?)}
Here $V$ is defined as in \eqref{SDE_full}, and the vertical size of the domain satisfies $\LL\in [L/2,2L]$.
The horizontal size remains $L$ and the boundary conditions in $x$ are periodic.
The equation \eqref{T_eq} naturally arises when we consider the average first exit time from the domain $\{(x,y)|y\in[0,\mathbb{L}]\}$. The solution to the equation \eqref{T_eq} can be decomposed into the $x$-average $\lan \TT\ran$ and the remainder $\TT_\sim$.  We further consider the 1D  system
\begin{align}\label{T_1D}
\nu\pa_{yy}\TT_{1D}+&V_{\mathrm{eff}}^{(2)}\pa_y \TT_{1D}=-1,\quad\TT_{1D}(y=0)=\TT_{1D}(y=\LL)=0,\\
V_{\mathrm{eff}}=& (0,V_{\mathrm{eff}}^{(2)}),\quad\,V_{\mathrm{eff}}^{(2)}(y)=\varphi(\chi|\pa_y \lan  c\ran(y)|)\frac{\pa_y \lan c \ran(y)}{|\pa_y \lan c\ran(y)|}.\label{T_1D_2}%,\, \lan V\ran(y)
\end{align} %\textcolor{red}{$V_{\mathrm{eff};2}$ or $V_{\mathrm{eff}}^{\lan 2\ran}$}
Here $\varphi\in C^2(\rr_+)$ is defined in \eqref{SDE_full}, and satisfies the constraint $ \varphi(0)=0$. Since the $x$-average $\lan c\ran$ solves the equation
\begin{align*}
(-\pa_{yy}+1)\lan c\ran=\lan n\ran,
\end{align*}
direct estimate yields that the $V_{\mathrm{eff}}$ satisfies the following bound
\begin{align}\label{V_eff_est}
||V_{\mathrm{eff}}||_\infty+||\pa_y V_{\mathrm{eff}}||_\infty\leq C(||\lan n\ran||_{W^{1,\infty}},\chi,||\varphi||_{C^2}).
\end{align}
Now we prove the convergence proposition.
%\textcolor{red}{Proposition 1. Comment on the shear (shear is non-degenerate. )}
\begin{pro}\label{pro:elliptic_eq}
Consider the solutions to \eqref{T_eq} and \eqref{T_1D}. Assume that the size of the domain $|\mathbb{L}|$ is bounded from below, i.e., $|\mathbb{L}|\in[\frac{1}{2}L,L]$ and the shear profile $u \in C^2$ is non-degenerate in the sense that
there are finitely many critical points and if $u'(y_0)=0$, then $u''(y_0)\neq 0$. Further assume that the chemical density $c_\sim$ satisfies estimates  \eqref{Quantitative_Horizontally_homogenization}.
If the shear strength $A$ is large enough, i.e., $A\geq A_0(\nu, L,\chi, ||\varphi||_{C^2}, ||n_\sim||_{W^{2,\infty}})$, then the following estimate holds:
\begin{align}
%||\TT_\sim||_\infty\leq &\frac{C|\log A|^{2 }}{A^{1/12{}}};\label{remainder_est}\\
||  \TT-\TT_{1D}||_\infty\leq& \frac{C|\log A|^2}{A^{1/6}}.\label{TT-TT1D_est}
\end{align}
Here the constant $C$ may depend on $\nu, \, L,\,\chi,\,||\varphi||_{C^2},\,||u'||_\infty$, and $\,||n_\sim||_{W^{2,\infty}(\Torus^2)}$.
%34 \textcolor{black}{(might not be a complete list yet)}
\end{pro}
%\textcolor{red}{Nondegenrate explain! Refer to case b) in the previous lemma. }
\begin{proof}
We organize the proof into several steps.

\noindent
\textbf{Step \# 1: Quantitative estimates on the solutions.} First we derive a bound over the deviation of the chemical attraction vector fields
\begin{align}\label{V-V_eff_Linfty}
||V-V_{\mathrm{eff}}||_\infty\leq \frac{C(L,\nu^{-1},\chi,||\varphi||_{C^2})|\log A|^2}{A^{1/6}} ||\na n_{\sim}||_\infty;\\
||V_\sim||_\infty+||\pa_x V||_\infty\leq \frac{C(L,\nu^{-1},\chi,||\varphi||_{C^2})|\log A|^2}{A^{1/6}}||\pa_x\na n_{\sim}||_\infty .\label{V_x_Linfty}
\end{align}
The bound \eqref{V-V_eff_Linfty} is a natural consequence of the estimate \eqref{Quantitative_Horizontally_homogenization} and the mean value theorem, %Here we use the fact that $||\na c_\sim||_\infty \leq C(\log A)^2A^{-1/6}(||\na n_{\sim}||_\infty+||\pa_{xxx} n_\sim||_2)$, and
\begin{align}
||V-V_{\mathrm{eff}} ||_\infty
\leq&\norm{\frac{\varphi(\chi|\na c|)}{|\na c|}(\na c-\na \lan c\ran)+\left(\frac{\varphi(\chi|\na c|)}{|\na c|}-\frac{\varphi(\chi|\na\lan  c\ran |)}{|\na \lan c\ran|}\right)\na \lan c\ran}_\infty\nb\\
%\leq& \textcolor{black}{C \norm{\na \left(\varphi(\chi|\cdot|)\frac{\cdot }{|\cdot|}\right)}_\infty||\na c-(0,\pa_y \lan c\ran)||_\infty= C(\chi,||\varphi||_{C^1})||\na(\lan c\ran+c_\sim)-(0,\pa_y\lan c\ran)||_\infty}\\
\leq& C(\chi,||\varphi||_{C^2})||\na c_\sim||_\infty\leq C(||\varphi||_{C^2},\chi,\nu,L)\frac{|\log A|^2}{A^{1/6}}||\na n_\sim||_\infty.
\end{align}
Next we estimate $||\pa_x V_\sim||_\infty$ with \eqref{Quantitative_Horizontally_homogenization}, and the fact that $\varphi(0)=0$,
\begin{align*}
||\pa_x V_\sim||_\infty\leq & \norm{\left(\frac{\varphi(\chi z)}{z}\right)'\bigg|_{z=|\na c|}\left(\frac{\na c}{|\na c|}\cdot \pa_x\na c\right)\na c}_{\infty}+\norm{\frac{\varphi(\chi\cdot)}{|\cdot|}}_\infty||\pa_x\na c||_\infty\\
\leq&C(||\varphi||_{C^2},\chi)||\pa_x\na c_\sim||_\infty\leq C(||\varphi||_{C^2},\chi,\nu,L)\frac{|\log A|^2}{A^{1/6}}||\pa_x\na n_\sim||_\infty.
\end{align*}
Note that by fundamental theorem of calculus, $||V_\sim||_\infty\leq ||\pa_x V_\sim||_\infty L$. Hence we have  obtained \eqref{V_x_Linfty}.

Now we estimate the $L^\infty$ and $H^1$ norms of the solutions $\TT$ \eqref{T_eq} and $\TT_\sim=\TT-\lan \TT\ran$. To derive the $L^\infty$ bound, we consider the following barrier
\begin{align*}
\nu W''(y)+V^{(2)}_{\mathrm{eff}}(y) W'(y)=5,\quad W(0)=W(\LL)=0.
\end{align*}
Here $V_{\mathrm{eff}}^{(2)}$ is the second component of $V_{\mathrm{eff}}$ \eqref{T_1D_2}.
By elliptic maximum principle, we observe that $W\leq 0$.
This equation is explicitly solvable with integration factors. The solution $W(y)$ and its derivative $W'(y)$ are bounded
\begin{align}\label{W_est}
||W'||_{L^\infty([0,\LL])}+||W||_{L^\infty([0,\LL])}\leq C(\nu, |\mathbb{L}|, ||\varphi||_\infty).%||V_\mathrm{eff}||_\infty
\end{align}
Consider the sum $\TT+W$, which satisfies
\begin{align*} \nu\de (\TT+W)+&{Au(y)\pa_x(\TT+W)}+V^{(2)} \pa_y (\TT+W)\\
&+(V^{(2)}_{\mathrm{eff}}-V^{(2)})\pa_y W +V^{(1)}\pa_x (\TT+W)=4.
\end{align*}
Since $W\leq 0$ and $\TT\geq 0$ by maximum principle, it is enough to derive the upper bound for $\TT+W$.
Rearranging the terms, we get
\begin{align*}
\nu\de (\TT+W)+{Au(y)\pa_x(\TT+W)}+V^{(2)} \pa_y (\TT+W) &+V^{(1)}\pa_x (\TT+W)=4+(V^{(2)}-V^{(2)}_{\mathrm{eff}})\pa_y W,\\
(\TT+W)(x,y=0)& =(\TT+W)(x,y=\LL)=0, \quad\forall x\in L\Torus.%_{\mathrm{eff}}T
\end{align*}
Combining the $L^\infty$ estimate \eqref{W_est} and the $||V-V_{\mathrm{eff}}||_\infty$ estimate \eqref{V-V_eff_Linfty}, applying the maximum principle for elliptic equations, and choosing $A_0(\nu,L,\chi, ||\varphi||_{C^1}, ||n_\sim||_{W^{1,\infty}})$ large enough,
we obtain that $T+W\leq 0$ and therefore,
\begin{align}\label{T_infty}
||\TT||_\infty\leq ||W||_\infty\leq C(\nu,L,||\varphi||_\infty). % ||V_{\mathrm{eff}}||_\infty_\infty
\end{align}
Once the $L^\infty$ bound is derived, the $L^2$ energy estimate yields the $H^1$ bound. Indeed, multiplying the equation \eqref{T_eq} by $T$ and integrating in space, we apply the H\"older inequality, Young's inequality  and integration by parts to obtain
\begin{align*}
\nu||\na \TT||_{L^2}^2\leq& ||\TT||_{L^1}+||V||_{L^\infty}||\na \TT||_{L^2}||\TT||_{L^2}\textcolor{black}{+\iint A u(y)\pa_x\left(\frac{\TT^2}{2}\right) dxdy}\\
\leq& L^2 ||\TT||_{L^\infty}+\frac{1}{4}\nu||\na \TT||_{L^2}^2+\frac{1}{\nu}L^2||V||_{L^\infty}^2|| \TT||_{L^\infty}^2.
\end{align*}
Note that the second term on the right hand side can be absorbed by the left hand side. We recall the estimate of $||\TT||_{L^\infty}$ \eqref{T_infty} and the $L^\infty$ norm bound $||V||_{L^\infty}\leq ||\varphi||_{L^\infty}$, and end up with
\begin{align}
||\na \TT_\sim||_2^2+||\pa_y \lan \TT\ran||_2^2\leq& C||\na \TT||_2^2\leq C(\nu, L,\chi,||\varphi||_{C^1})(1+||\TT||_\infty^2)\nb\\
\leq &C(\nu,L,\chi, ||\varphi||_{C^1}).\label{T_H1}%_{H^1}^2\textcolor{red}{use\, C},\textcolor{red}{Drop???||n||_{W^{1,\infty}}}
\end{align}
Next we estimate higher regularity norm of the solution. By taking the $\pa_x$ derivative of the equation \eqref{T_eq} and testing it with $\pa_x \TT$, we obtain
\begin{align*}%\nonumber
\nu||\na \pa_x \TT||_2^2\leq &(||\lan V\ran||_\infty+||V_\sim||_\infty)||\pa_x\TT_\sim||_2||\na \pa_x \TT||_2+ ||\pa_x V_\sim||_\infty||\pa_x \TT_\sim||_2||\na \lan \TT\ran||_2\\
&+||\pa_xV_\sim||_\infty||\pa_x\TT_\sim||_2||\na \TT_\sim||_2\\
\leq& C||\pa_x \TT_\sim||_2\bigg(||\lan V\ran||_\infty^2||\pa_x \TT_\sim||_2+||V_\sim||_\infty^2||\pa_x \TT_\sim||_2+||\pa_x V_\sim||_\infty||\na \lan \TT\ran||_2\\&\quad\quad\quad\quad+||\pa_xV_\sim||_\infty||\na \TT_\sim||_2\bigg)+\frac{1}{4}\nu||\na \pa_x \TT||_2^2.
%\label{aux623}
\end{align*}%\textcolor{red}{Do one more step on the terms $\lan V\ran \pa_x\TT\na\pa_x\TT$}  \textcolor{black}{Done.}
By recalling the estimates \eqref{V_x_Linfty}, \eqref{T_H1}, and the fact that $||V||_\infty\leq ||\varphi||_\infty$, we infer the following estimate:
\begin{align}
%||\pa_x\na \TT_\sim||_2^2=
||\pa_x\na  \TT||_2^2\leq C(\nu,L,\chi, ||\varphi||_{C^2},||n||_{W^{2,\infty}}).\label{T_H2}%\textcolor{red}{Use\, }C_{H^2}^2
\end{align}
This concludes the first step.

\noindent
\textbf{Step \# 2: Convergence of solutions.} First, we observe that $\TT_\sim$ solves the following equation
\begin{align}
\nu \de \TT_\sim+Au(y) \pa_x\TT_\sim=-\lan V\ran\cdot \na \TT_\sim-V_\sim\cdot \na \lan \TT\ran-(V_\sim\cdot \na \TT_\sim)_\sim,\quad
\TT_\sim\big|_{y=0,\LL}=0.\label{TT_sim_eq}
\end{align}
Recall our notation $\mathcal{L}^\dagger$ for the differential operator $\mathcal{L}^\dagger=\frac{1}{2}\nu\de +Au(y)\pa_x$ subject to Dirichlet boundary conditions at $y=0,\LL$.  We also recall Theorem 1.1 in \cite{AlbrittonBeekieNovack21}, which provides enhanced dissipation estimates for the solutions to passive scalar equations subject to shear flows and Dirichlet boundary conditions in the channel. The explicit estimate is identical to \eqref{ED_literature_nondegenerate}, so we omit the details. Combining this and the argument in the proof of Theorem \ref{thm:Horizontally_homogenized} yields the following enhanced dissipation estimate for $\mathcal{L}^\dagger:$ %{\color{red} It would be better to state what exactly the main result of the paper we refer to says, and include a specific label like Theorem XX from ....}
% the main result  \cite{He21} and Direct adaptation of the argument in \cite{He21} and the proof of Theorem \ref{thm:Horizontally_homogenized} yields
 %\footnote{\textcolor{black}{In \cite{BCZ15}, the enhanced dissipation estimate in the channel subject to the Neumann boundary condition is proven. By applying similar argument, one can derive the estimate in the Dirichlet case. In this case, the dangerous boundary term discussed in \cite{BCZ15} is not present. Or one can use the resolvent estimates as in \cite{Wei18}, \cite{He21}.}}
\begin{align}\label{ED_channel}
||e^{t\mathcal{L}^\dagger} \eta_\sim||_{L^2}\leq C e^{-\kappa  \nu^{1/2}A^{1/2}L^{-3/2}t}||\eta_\sim||_{L^2},\quad \forall t\in[0,\infty). %|\log \frac{\nu}{AL}|^{-2}
\end{align}
%\textcolor{red}{$L^2$ bound?}
%\textcolor{black}{Since the coefficients and forcing of the equation \eqref{TT_sim_eq} is smooth $C^\infty$, we have that the right hand side of \eqref{TT_sim_eq} is smooth by elliptic regularity. }
Now we apply the estimate \eqref{ED_channel} and Lemma \ref{lem:integral_formula} to derive that
\begin{align*}
||\TT_\sim||_2\leq &\int_0^\infty ||e^{t\mathcal{L}^\dagger}(-(V_\sim\cdot \na \TT_\sim)_\sim-\lan V\ran\cdot \na \TT_\sim-V_\sim\cdot \na \lan\TT\ran)||_2dt\\
\leq &\frac{ C(\nu,L)}{A^{1/2}}(||V_\sim||_\infty||\na \TT_\sim||_2+||\lan V\ran ||_\infty ||\na \TT_\sim||_2+||V_\sim||_\infty ||\pa_y \lan \TT\ran ||_2).%|\log A|^2
\end{align*}Applying estimate \eqref{T_H1}, % as follows\begin{align}||\pa_y \lan \TT\ran||_2^2\leq C||\na \TT||_2^2\leq C C_{H^1}^2. \end{align}
and the fact that $||V_\sim||_\infty+||\lan V\ran||_\infty\leq C||\varphi||_{C^1}$, we obtain that
\begin{align}\label{T_sim_2}
||\TT_\sim||_2\leq \frac{1}{A^{1/2}} C( \nu^{-1},L,\chi,||\varphi||_{C^1}).%|\log A|^2
\end{align}

Next we derive the $\dot H _x^1$-estimate of $\TT$. Taking the $\pa_x$ derivative of \eqref{T_eq} and applying formula \eqref{formula} and the estimate \eqref{ED_channel}, we have that
\begin{align*}
||\pa_x\TT_\sim||_2\leq &\int_0^\infty ||e^{t\mathcal{L}^\dagger}(-\pa_x(V_\sim\cdot \na \TT_\sim)_\sim-\lan V\ran\cdot \na\pa_x \TT_\sim-\pa_x V_\sim\cdot \na \lan\TT\ran)||_2dt\\
\leq &\frac{ C(\nu^{-1},L)}{A^{1/2}}(||\pa_x V_\sim||_\infty||\na \TT_\sim||_2+||V_\sim||_\infty||\pa_x \na \TT_\sim||_2+||\lan V\ran ||_\infty ||\na\pa_x  \TT_\sim||_2\\
&+||\pa_x V_\sim||_\infty ||\pa_y \lan \TT\ran ||_2).%|\log A|^2|\log A|^2
\end{align*}By estimates  \eqref{V_x_Linfty}, \eqref{T_H1}, \eqref{T_H2}, and we obtain
\begin{align}\label{pa_xTT_small}
||\pa_x\TT_\sim||_2\leq \frac{1}{A^{1/2}}C(\nu^{-1},L,\chi,||\varphi||_{C^2},||n ||_{W^{2,\infty}}).%22
\end{align}

To derive the estimate of $||\TT_\sim||_\infty$, we recall that
\begin{align}||\TT_\sim||_{L^\infty(L\Torus\times[0,\LL])}\leq C(L)||\pa_{xy}\TT_{\sim}||_{L^2(L\Torus\times[0,\LL])},\label{Sobolev_xy}
\end{align}
which is a direct consequence of the fundamental theorem of calculus and H\"older inequality. % is small?
Now we estimate the quantity $||\pa_{xy}\TT_\sim||_2$.  Similarly to the derivation of  \eqref{T_H2}, by taking the $\pa_x$-derivative of the equation \eqref{T_eq}, testing it against $\pa_x \TT_\sim$, and recalling estimates $||V||_\infty\leq ||\varphi||_\infty,$  \eqref{V_x_Linfty}, \eqref{T_H1},   \eqref{pa_xTT_small}, we obtain: %\textcolor{red}{Add similar to ref, we have that, add a comment about that we pick a simpler bound on powers of A. }\textcolor{black}{Done}
\begin{align*}
\nu||\na \pa_x \TT_\sim||_2^2
\leq& C(||\pa_x V_\sim||_\infty||\na \lan \TT\ran||_2+||\pa_x \TT_\sim||_2||\lan V\ran||_\infty^2+||\pa_xV_\sim||_\infty||\na \TT_\sim||_2\\&+||V_\sim||_\infty^2||\pa_x \TT_\sim||_2)||\pa_x \TT_\sim||_2\\
\leq &\frac{1}{A^{1/2}}C (\nu,L,\chi,||\varphi||_{C^2},||n_\sim||_{W^{2,\infty} }).
\end{align*}
Here we choose the non-optimal factor $A^{-1/2}$ for simplicity; we could have replaced by $|\log A|^2 A^{-2/3}.$
Hence by \eqref{Sobolev_xy}, we have that \begin{align}
||T_\sim||_\infty\leq C(\nu,L,\chi,||\varphi||_{C^2},||n_\sim||_{W^{2,\infty}})\frac{1}{A^{1/4}}\label{remainder_est}.
\end{align}

Next we note that $\lan \TT\ran$ solves the equation
\begin{align*}
\nu\pa_{yy}\lan \TT\ran+\lan V^{(2)}\ran\pa_y \lan \TT\ran+\lan V_\sim\cdot \na \TT_\sim\ran=-1,\quad \lan\TT\ran(y=0)=\lan\TT\ran(y=\LL)=0.
\end{align*} The difference between $\lan \TT\ran $ and $\TT_{1D}$ \eqref{T_1D} satisfies
\begin{align*}
\nu\pa_{yy} (\lan \TT\ran-\TT_{1D})+ \lan V^{(2)}&\ran\pa_y(\lan \TT\ran -\TT_{1D})+(\lan V^{(2)}\ran-V^{(2)}_{\mathrm{eff}})\pa_y \TT_{1D}+\lan V_\sim\cdot \na \TT_\sim \ran  =0,\\
(\lan \TT\ran-\TT_{1D})(y=0)=&(\lan\TT\ran-\TT_{1D})(y=\LL)=0.
\end{align*}
By integration factor method, one can derive the bound
\begin{align}\label{pa_y_TT_1D_infty}
||\pa_y \TT_{1D}||_\infty\leq C(\nu,L,\chi,||\varphi||_{C^1}).
\end{align}
Define $\mathcal{F}:=\textcolor{black}{-(\lan V^{(2)}\ran-V_{\mathrm{eff}}^{(2)})}\pa_y \TT_{1D}-\lan V_\sim\cdot \na \TT_\sim \ran$, then by
\eqref{V-V_eff_Linfty}, \eqref{V_x_Linfty}, \eqref{T_H1},  and \eqref{pa_y_TT_1D_infty}, \begin{align}\label{F_infty}
||\mathcal{F}||_2\leq \frac{|\log A|^2}{A^{1/6}} C(\chi,||\varphi||_{C^2},\nu^{-1},L,||u'||_\infty,||n_\sim||_{W^{2,\infty}}).
\end{align}
Application of the integration factor method yields that:%\leq C\frac{1}{A^{...}} o show
\begin{align*}
(\lan \TT\ran-\TT_{1D})(y)=&\frac{2}{\nu}\int_0^y\int_0^w\mathcal{F}(z)e^{-\int_z^w \frac{2}{\nu}\lan V^{(2)}\ran d\tau}dzdw+C_1 \int_0^y e^{-\frac{2}{\nu}\int_0^w \lan V^{(2)}\ran d\tau}dw;\\
C_1=&\textcolor{black}{-\frac{2}{\nu}\frac{\int_0^{\LL}\int_0^w\mathcal{F}(z)e^{-\frac{2}{\nu}\int_z^w \lan V^{(2)}\ran d\tau}dzdw}{\int_0^{\LL} e^{-\frac{2}{\nu}\int_0^w\lan V^{(2)}\ran d\tau}dw}}.
\end{align*}
%\textcolor{black}{I have replaced the $L$ by $\LL$ in the last line because we are on the domain $[0,\LL]$. }
Hence, by \eqref{F_infty}, and $||V||_\infty\leq C||\varphi||_{C^1}$, we obtain
\begin{align*}
||\lan \TT\ran-\TT_{1D}||_{L_y^\infty}\leq \frac{|\log A|^{2}}{A^{1/6}}C(\chi,||\varphi||_{C^2},\nu^{-1},L,||u'||_\infty,||n_\sim||_{W^{2,\infty}}).
\end{align*}%45 5  \cap H_x^3
Combining it with \eqref{remainder_est}, we  obtain the $L^\infty$ estimate \eqref{TT-TT1D_est}.
%By Moser iteration, we have the estimate \eqref{TT-TT1D_est}.
\end{proof}

\begin{proof}[Proof of Theorem \ref{theorem_first_hitting_time_full_system}]
%{\color{red} Please adjust this proof for the result avoiding constraint on the initial point}
%\textcolor{red}{Add a comment about the agent starting inside the egg strip.}\textcolor{black}{It is harder than expected because near the egg center $u(L/2)=0$.} 2 2
In the main part of the proof, we focus on the case where the starting position $(x_0,y_0)$ is outside the target zone $[\frac{L}{2}-\delta,\frac{L}{2}+\delta]$. %(H: Change from $[\frac{L}{2}-2\delta,\frac{L}{2}+2\delta]$ to $[\frac{L}{2}-\delta, \frac{L}{2}+\delta]$.)
We will provide comments concerning the case where the starting position is inside $[\frac{L}{2}-\delta,\frac{L}{2}+\delta]$ at the end.
\ifx
The $1$-dimensional first hitting time is defined as follows
\begin{align}\label{T_1D_chemo}
T_{1D}^\chi=\min\left\{t\bigg|Y_t=\frac{L}{2}\pm \delta,\quad Y_t \text{ solves }\right\}.
\end{align}
\footnote{
\textcolor{red}{To show the convergence, we need to show that $|\mathbb{E}T_{1D}^{\mathrm{eff}}( Y_t\in S^\pm)-\mathbb{E}T_{1D}^{\mathrm{eff}}(Y_t=y^\pm )|\leq CA^{-q}$ and $|\mathbb{E}T_{2D}^{A}(Y_t=y^\pm)-\mathbb{E}{T}^{\mathrm{eff}}_{1D}(Y_t=y^\pm)|\leq C|\log A|^2 A^{-1/6}$. The first estimate can be derived by integration factor. The second one can be shown by the PDE.
}}
 The expected $1$-D hitting time can be explicitly calculated with the Dynkin formula and the equation \eqref{T_1D}.
\begin{align*}
T_{2D}^{A;\chi}=\min_t\{t| (X_t,Y_t)\in B\left(\left(\frac{L}{2}, \frac{L}{2}\right);\delta\right),\quad (X_t,Y_t) \text{ solves }\eqref{SDE_full}\},
\end{align*}
and the $1$-dimensional time \eqref{T_1D_chemo},
\fi
 To relate the 2-dimensional expected first hitting time \eqref{Defn_T_2D_A_chi} to the 1-dimensional hitting time  \eqref{T_1D_chi},  we define another time
\begin{align*}
T_{2D;0}^{A;\chi}=\min_t\left\{t\big||Y_t-{L}/{2}|\leq \delta\right\}.
\end{align*}
We observe that $\mathbb{E}^{(x_0,y_0)}T_{2D;0}^{A;\chi}\leq \mathbb{E}^{(x_0,y_0)} T_{2D}^{A;\chi}$. Moreover, by the Dynkin's formula,  $\mathbb{E}^{(x_0,y_0)}T_{2D;0}^{A,\chi}$ solves the following PDE
\begin{align*}
\nu\de \TT+V\cdot \na \TT+Au(y)\pa_x \TT=&-1,\quad\
\TT(x,y=\frac{L}{2}+\delta)=\TT(x,y=\frac{L}{2}-\delta)=0,
\end{align*}
%By the Dynkin's formula, $\mathbb{E}^{(x_0,y_0)}T_{2D;0}^{A;\chi}$ solves
which is the partial differential equation $\eqref{T_eq} _{\LL=L-2\delta}$ modulo suitable shifting of $y$-coordinate. Similarly, the average $\mathbb{E}^{y_0} T_{1D}^\chi$ solves the ordinary differential equation $\eqref{T_1D}_{\LL=	L-2\delta}$  modulo suitable shifting of $y$-coordinate. By Proposition \ref{pro:elliptic_eq}, we have that \begin{align}\label{lower_bound}
\mathbb{E}^{y_0} T_{1D}^\chi\leq \mathbb{E}^{(x_0,y_0)}T_{2D;0}^{A;\chi}+|\mathbb{E}^{(x_0,y_0)}T_{2D;0}^{A;\chi}-\mathbb{E}^{y_0} T_{1D}^\chi|\leq \mathbb{E}^{(x_0,y_0)}T_{2D}^{A;\chi}+\frac{C|\log A|^2}{A^{1/6}}.
\end{align}

Next we apply an idea similar to one in the proof of Theorem \ref{thm:First_hitting_time_passive_transport} to estimate the upper bound of the average first hitting time $T_{2D}^{A;\chi}$. We decompose the searching process into individual trips. For the $0$-th trip, the agent reaches the level $y=\frac{L}{2}\pm \delta$. The expected time of the $0$-th trip is $\mathbb{E}^{(x_0,y_0)}T_{2D;0}^{A;\chi}$. In the first trip, the agent moves from $y=\frac{L}{2}\pm {\delta}$ to $y^\pm=\frac{L}{2}\pm\delta\mp A^{-1/3}$.
The average time for the agent to go from $y=\frac{L}{2}+\delta$ to $y=\frac{L}{2}+\delta-A^{-1/3}$ is estimated as follows.
To set up application of Proposition \ref{pro:elliptic_eq}, we consider the following ODE on $[0,  {L}-2\delta+2A^{-1/3} ]$:
\begin{align}\label{phi623}
\nu \phi''(y)+  {V}^{(2)} _{\mathrm{eff}}(y)\phi'(y)=-1,\quad\
\phi(0)=\phi \left({L}-2\delta+2A^{-1/3}  \right)=0.
\end{align}
%Here $\widetilde{V}_{\mathrm{eff}}$ is the even reflection of the vector field ${V}_{\mathrm{eff}}$ around the line $y=\frac{L}{4},\,\frac{3L}{4}$.
By the Dynkin's formula, the expected 1D hitting time from $y=y^\pm\pm A^{-1/3}$ to the point $y^\pm$ is $\phi(A^{-1/3})$. The equation \eqref{phi623} has a solution \begin{align*}
\phi(y)= &-\frac{1}{\nu}\int_0^y\int_0^w e^{-\int_z^w \frac{1}{\nu}V_{\mathrm{eff}}^{(2)} d\tau}dzdw+C_1 \int_0^y e^{-\int_0^w \frac{1}{\nu}V_{\mathrm{eff}}^{(2)} d\tau}dw;\\
C_1=& \frac{\int_0^{L-2\delta+2A^{-1/3}}\int_0^we^{-\frac{1}{\nu}\int_z^w V_{\mathrm{eff}}^{(2)}d\tau}dzdw}{\nu\int_0^{L-2\delta+2A^{-1/3}} e^{-\frac{1}{\nu}\int_0^w V_{\mathrm{eff}}^{(2)}d\tau}dw}.
\end{align*}% so that $||\phi'||_\infty\leq C(\nu, L,|| \varphi||_{C^1})$.
As a result, we have that
\begin{align}\label{tau_out_chemo}
\phi(A^{-1/3})\leq A^{-1/3}C(\nu, L,\chi,||\varphi||_{C^2}).
\end{align}
Hence by Proposition \ref{pro:elliptic_eq}, the expected first hitting time from $y^\pm\pm A^{-1/3}$ to $y^\pm$ is less than $C(\nu^{-1}, L,\chi,||\varphi||_{C^2},||u'||_\infty,||n||_{W^{2,\infty}})|\log A|^2 A^{-1/6}$. As a result, as $A\rightarrow \infty$, the time spent for the first trip converges to zero.

% (\underline{H: The original expressions are not quite accurate. So I rephrase this paragraph})
Once the agent reaches the level $y^\pm$, we study the dynamics of the agent within the searching zone strips $(y^\pm \pm A^{-1/3})$ in each trip.
%First we estimate the probability that the agent leaves the searching zone within time $5L u_m^{-1}A^{-1}$.
To this end, we consider the process
\begin{align*}
dY_t&= V^{(2)}(X_t,Y_t) dt+\sqrt{\nu}dB_t^{(2)}.%\chi \lan\ran_{\mathrm{eff}}
\end{align*}
Here $(X_0,Y_0)$ is the starting position and $Y_0=y^\pm$.  Recall that $\tau_{in}$ is the first hitting time by the agent of the boundary of the strip and $u_m$ \eqref{u_m} is the minimum of $|u|$ in the strip $(y^\pm \pm A^{-1/3})$, i.e., $u_m=\min|u(y)|\geq u_d\delta/2$.  %the magnitude of the shear
Define $F_i$ to be the event
\begin{align}\label{F_i_thm_3}
F_i:=\left\{\tau^{in}\leq \frac{5L}{Au_m}\right\}\cup \left\{\tau^{in}\geq\frac{5L}{Au_m}, d_\rr(X_{\frac{5L}{Au_m}},X_0)\leq L\right\}.
\end{align}
We observe that if $F_i^c$ happens, then the search is successful in the $i$-th trip. Similarly to the derivation of \eqref{shortprob}, we decompose the event $F_i$ \eqref{F_i_thm_3} into several subcases, and estimate the probability of them individually.
%(H.: Here the definition is similar to \eqref{F_i} but the constant is slightly different. Here we have `$5$' instead of `$2$'.)

% {\color{red} I think it is better to recall the definition of $F_i$ here} %We estimate the first hitting time for $Y_t$ to reach the boundary.
\ifx
Without loss of generality, we shift the starting point to zero.
Now starting from the origin, and once the  $|Y_t|^2$ reaches the level $A^{-2/3}$, the agent leave the searching zone. By the Ito formula, the time evolution of $|Y_t|^2$ is the following
\begin{align*}
d|Y_t|^2=2Y_t (\chi \lan V_2\ran_{\mathrm{eff}}(Y_t)dt+\sqrt{\nu}dB_t^{(2)})+\nu dt. ?%2
\end{align*}
Now we have that
\begin{align*}
\mathbb{E} |Y_t|^2\leq 0+2\chi\frac{1}{A^{1/3}}||\lan V_2\ran_{\mathrm{eff}}||_\infty t+\nu t.
\end{align*}
Now for $\forall t\leq \frac{10L}{u_m A}$, we have that
\begin{align*}
\mathbb{P}(|Y_t|^2\geq A^{-2/3})\leq A^{2/3}\mathbb{E}|Y_t|^2\leq \frac{20L}{u_m A^{1/3}}(\chi|| V^{(2)}_{\mathrm{eff}}||_\infty +\nu).
\end{align*}
\fi%\lan\ran
%Hence
%\begin{align*}
%\mathbb{P}\left(\tau^{in}\leq \frac{5L}{u_m A}\right)\leq \frac{20L}{u_m A^{1/3}}(\chi||\lan V_2\ran_{\mathrm{eff}}||_\infty +\nu).
%\end{align*}\textbf{This is not justified}
%
%
%\textcolor{red}{Can we show that the advection can only cover half the distance and the Brownian motion have to over the remaining half?}
Since $||  V ||_\infty\leq C,$ %$||V||_\infty \frac{5L}{u_m A}\leq \frac{1}{2A^{1/3}}$, hence%_{\mathrm{eff}}_{\mathrm{eff}}
\begin{align}
\mathbb{P}\left(\tau^{in}\leq \frac{5L}{u_m A}\right)\leq&\mathbb{P}\left(|Y_{\tau^{in}}|=\bigg|\int_0^{\tau^{in}} V^{(2)} (X_t,Y_t)dt+\sqrt{\nu}\int_0^{\tau^{in}}dB_t\bigg|\geq A^{-1/3},\quad \tau^{in}\leq \frac{5L}{u_m A}\right)\nb\\
\leq&\mathbb{P}\left(\bigg|\int_0^{\tau^{in}} V^{(2)} (X_t, Y_t)dt\bigg|\geq \frac{1}{2A^{1/3}}\text{ for some }\tau^{in}\leq \frac{5L}{u_m A}\right)\nb\\
&+\mathbb{P}\left(\sqrt{\nu}|B_{\tau^{in}}|\geq\frac{1}{2A^{1/3}}\text{ for some  }\tau^{in}\leq \frac{5L}{u_m A}\right)=:P_1+P_2.\label{P_1_P_2}%_{\mathrm{eff}}_{\mathrm{eff}}
\end{align}
The first term $P_1$ can be estimated using the fact that $||V||_\infty\leq C$ as follows:%_{\mathrm{eff}}_{\mathrm{eff}}
\begin{align*}
P_1\leq \mathbb{P}\left(||V||_\infty\frac{5L}{u_m A}\geq \frac{1}{2A^{1/3}}\right)=0;
\end{align*}
the probability is zero if $A$ is chosen large enough compared to $||V||_\infty,\, L,\, u_m$. %_{\mathrm{eff}}  has probability zero because the time is too short and t
The second term $ P_2$ in \eqref{P_1_P_2} can be estimated with reflection principle of Brownian motion as follows
\begin{align*}
%\mathbb{P}\left(\tau^{in}\leq \frac{5L}{u_m A}\right)
P_2\leq\mathbb{P}\left(\min_{t}\{t|\sqrt{\nu}B_t=\frac{1}{2}A^{-1/3}\}\leq\frac{5L}{u_m A}\right)\leq 2\mathbb{P}\left(\sqrt{\nu}B_{\frac{5L}{Au_m}}\geq \frac{1}{2}A^{-1/3}\right)\leq \frac{10\nu L}{A^{1/3}u_m}.
\end{align*}
As $A\rightarrow \infty$, this approaches $0$.
%\textcolor{red}{Next we estimate the probability of the event $F_i^c$ \eqref{shortprob} \begin{align}
%\mathbb{P}(F_i^c)\geq \mathbb{P}( d_\rr(X_\frac{2L}{Au_m}, X_0)>  L)\geq \mathbb{P}(|\sqrt{\nu}B_{\frac{2L}{Au_m}}|\leq \frac{L}{2})\geq 1/10.
%\end{align}
%The remaining part of the proof is similar to the proof of Theorem \ref{thm:First_hitting_time_passive_transport}. Hence we omit the details. } \ifx
Now the probability
\begin{align*}
\mathbb{P}\left(\tau^{in}\geq \frac{5L}{Au_m},\, d_\rr(X_{\frac{5L}{Au_m}},X_0)\leq L\right)\leq\mathbb{P}\left(\sqrt{\nu}B_{\frac{5L}{u_m A}}\geq L\right)\leq \frac{C(\nu, u_m, L,|| V||_\infty )}{A}.%^{1/3^2}  _2 \lan\ran_{\mathrm{eff}}, \chi
\end{align*}
Note that in the above, due to $\|V\|_\infty \leq C,$ we can absorb any contribution due to chemotactic advection.
Now we obtain an estimate similar to \eqref{shortprob}. Application of the argument parallel to that in the proof of Theorem \ref{thm:First_hitting_time_passive_transport} leads to
\begin{align}
\mathbb{E}^{(x_0,y_0)}T_{2D}^{A;\chi}\leq \mathbb{E}^{y_0} T_{1D}^\chi+C(\nu^{-1},L,\chi,||u'||_\infty,||\varphi||_{C^2},||n||_{W^{2,\infty{}}})|\log A|^2A^{-1/6}.
\end{align}
Combining it with the lower bound \eqref{lower_bound},  the convergence of the expected first hitting time follows.%\fiupper bound

Finally, we comment on the case where $y_0\in[L/2-\delta,L/2+\delta]$. We apply the same adjustments as in the proof of Theorem \ref{thm:First_hitting_time_passive_transport}, and recall the  definitions of the searching strips therein. The main difference is that the expected first hitting time from the starting point $(x_0,y_0)$ to the center of the searching strip is bounded as follows, $$\mathbb{E}[\tau_0]\leq C(\nu^{-1}, L,\chi,||\varphi||_{C^2},||u'||_\infty,||n||_{W^{2,\infty}})|\log A|^2 A^{-1/6}.$$ The explicit estimate is similar to the treatment of the first trip above. The remaining estimates are similar to the ones yielding Theorem \ref{thm:First_hitting_time_passive_transport} and we omit them for the sake of brevity.

\end{proof}

{\color{black}
Finally, we prove that the average effective searching time $\mathbb{E}^{y_0}T_{1D}^\chi$ in Theorem \ref{theorem_first_hitting_time_full_system} is less than the $1D$-Brownian motion hitting time $\mathbb{E}^{y_0}T_{1D}$ in Theorem \ref{thm:First_hitting_time_passive_transport}.
\begin{pro}
If the egg density $\lan n\ran(y)$ is symmetric about the point $y=L/2$ and supported inside the strip $y\in[L/2-\delta,L/2+\delta]$, then $\mathbb{E}^{y_0}T_{1D}^\chi\leq \mathbb{E}^{y_0}T_{1D}$ for $y_0\notin[L/2-\delta, L/2+\delta]$.
\end{pro}
\begin{proof}We decompose the proof into two steps. To simplify the notation, we consider the problem in a shifted coordinate system so that $y\in [-L/2, L/2]$, and $\text{support}(\lan n\ran)\in[-\delta,\delta]$.

\noindent \textbf{Step \# 1:} We show that the chemical gradient has a favorable sign, i.e., $y\pa_yc_{\text{eff}}\leq 0$ for $\forall y\in[-L/2, -\delta]\cup[\delta, L/2]$. The chemical equation on the shifted domain $[-L/2,L/2]=L\Torus$ reads as follows
\begin{align}\label{c_eff_Torus}
-\nu \pa_{yy}c_{\text{eff}}+c_{\text{eff}}=\lan n\ran,\quad \lan n\ran(y)=\lan n\ran(-y).
\end{align}%In general, the solution to the elliptic problem  involves an infinite sum and it is hard to keep track of the monotonicity property of $\pa_y c_{\text{eff}}$. Our idea is to
We consider the candidate $c_{\text{eff}}=c_{\text{eff};\rr}+\mathcal{R}(y)$. Here $c_{\text{eff};\rr}$ solves \eqref{c_eff_Torus} on the real line $\rr$ subject to the source $\lan n\ran$ supported in $[-\delta, \delta]$. The homogeneous remainder $\mathcal{R}(y)$ corrects the boundary conditions. The variation of parameters method yields that:
\begin{align}
c_{\text{eff};\rr}(y)=& \frac{1}{2\sqrt{\nu}}\int_{y}^{\infty} e^{\frac{1}{\sqrt{\nu}}(y-z)}\lan n\ran (z)dz+\frac{1}{2\sqrt{\nu}}\int_{-\infty}^{y}e^{-\frac{1}{\sqrt{\nu}}(y-z)}\lan n\ran(z)dz, y\in [-L/2, L/2].
\end{align}Since $c_{\text{eff};\rr}$ is even, we have $c_{\text{eff};\rr}(y=-L/2)=c_{\text{eff};\rr}(y=L/2)$, and $\pa_y c_{\text{eff};\rr}(y=-L/2)=-\pa_y c_{\text{eff};\rr}(y=L/2)$. Further note that $\lan n\ran (y)=0$ for $y$ close to $\pm L/2$, so the equation \eqref{c_eff_Torus} yields that $\pa_{y}^{2}c_{\text{eff};\rr}(L/2)=\frac{1}{\nu} c_{\text{eff};\rr}(L/2)=\frac{1}{\nu} c_{\text{eff};\rr}(-L/2)=\pa_{y}^{2m}c_{\text{eff};\rr}(-L/2)$. Similarly all even order derivatives of $c_{\text{eff};\rr}(L/2)$ matches at $y=\pm L/2$. Next we define the even corrector
\begin{align}
\mathcal{R}(y)=&\frac{G}{\sqrt{\nu}}\cosh(\frac{y }{\sqrt{\nu}}),\ \
\quad G = \frac{1}{2}e^{-\frac{L}{2\sqrt{\nu}}}\int_{ -\delta}^{ \delta}e^{-\frac{z}{\sqrt{\nu}}}\lan n\ran(z)dz\frac{1}{\sinh(L/(2\sqrt{\nu}))}.&
\end{align} Since $\pa_{y}^{2m}\mathcal{R}$ is even, we have $\pa_{y}^{2m}c_{\text{eff}}(L/2)=\pa_{y}^{2m}c_{\text{eff}}(-L/2)$ for $\forall m\in \mathbb{N}$. Next we observe that the choice of $\mathcal{R}$ guarantees that the derivative $\pa_y c_{\text{eff}}(\pm L/2)$ is zero. Similar arguments as above yields that  $\pa_{y}^{2m+1}c_{\text{eff}}(L/2)=\pa_{y}^{2m+1}c_{\text{eff}}(-L/2)=0, $  for all $m\in\mathbb N$. Hence $c_{\text{eff}}$ is indeed a solution to \eqref{c_eff_Torus}.
We can compute the derivative of the chemical density
%\begin{align}
%\pa_y c_{\text{eff}}(y)=& \frac{1}{2\nu} \int_{y}^{\infty} e^{\frac{1}{\sqrt{\nu}}(y-z)}\lan n\ran (z)dz-\frac{1}{2\nu}\int_{-\infty}^{y}e^{-\frac{1}{\sqrt{\nu}}(y-z)}\lan n\ran(z)dz+\frac{A}{\nu}\sinh(\frac{y-L/2}{\sqrt{\nu}}).%
%\end{align}F we have that the gradient is
for $y\in[-L/2,-\delta]$,
\begin{align}
\pa_y c_{\text{eff}}(y)=& \frac{1}{2\nu}e^{\frac{ y}{\sqrt{\nu}}} \int_{ -\delta}^{ \delta} e^{\frac{-z }{\sqrt{\nu}}}\lan n\ran (z)dz+\frac{1}{2\nu}\int_{ -\delta}^{ \delta}e^{-\frac{z}{\sqrt{\nu}}}\lan n\ran(z)dz\frac{e^{-\frac{L}{2\sqrt{\nu}}}}{\sinh(L/(2\sqrt{\nu}))} \sinh(\frac{y }{\sqrt{\nu}}) \\
=&  \frac{1}{2\nu}\int_{-\delta}^{\delta} e^{\frac{ -z}{\sqrt{\nu}}}\lan n\ran (z)dz  \left(  e^{\frac{y}{\sqrt{\nu}}}+\frac{e^{-\frac{L}{2\sqrt{\nu}}}\sinh(\frac{y }{\sqrt{\nu}})}{\sinh(\frac{L}{2\sqrt{\nu}})}\right)\\
=&\frac{1}{2\nu}\int_{ -\delta}^{ \delta} e^{\frac{ -z}{\sqrt{\nu}} }\lan n\ran (z)dz \frac{\sinh(\frac{y}{\sqrt{\nu}}+\frac{L}{2\sqrt\nu})}{\sinh(\frac{L}{2\sqrt{\nu}})}\geq 0.
\end{align}
By symmetry, we have that the gradient is negative for $y\in [\delta, L/2 ]$.

\noindent \textbf{Step \# 2: }Compare the two hitting times. If the starting point $y_0$ is in $[-\delta,\delta]$, then both hitting times are zero. Hence it is enough to consider $y_0\in L\Torus\backslash [-\delta,\delta]$. Thanks to the periodicity of the domain, we can focus on $y\in[\delta, L-\delta]$. Now we consider the following two elliptic equations:
\begin{align}
\frac{1}{2}\nu\pa_{yy}\mathcal{T}_{B}=-1, \quad \mathcal{T}_{B}(y=\delta, L -\delta)=0;\\
\frac{1}{2}\nu\pa_{yy}\mathcal{T}_{\text{eff}}+V_{\text{eff}}\pa_y \mathcal{T}_{\text{eff}}=-1, \quad \mathcal{T}_{\text{eff}}(y= \delta, L -\delta)=0.
\end{align} By Dynkin' s formula, $\mathbb{E}^{y_0}T_{1D}^\chi=\mathcal{T}_{\text{eff}}^\chi(y_0),\
  \mathbb{E}^{y_0}T_{1D}=\mathcal{T}_{\text{eff}}(y_0)$.
The expected first hitting time for the Brownian motion is direct $\mathcal{T}_{B}=-\frac{1}{\nu}(y-\delta)(y-(L-\delta))$. Now we have that
\begin{align}
\frac{1}{2}\nu\pa_{yy}(\mathcal{T}_B-\mathcal{T}_{\text{eff}})+V_{\text{eff}}\pa_{y}(\mathcal{T}_B-\mathcal{T}_{\text{eff}})=-V_{\text{eff}}\frac{1}{\nu}\pa_y\left((y-\delta)(y-L+\delta)\right)\leq 0.
\end{align}
By maximum principle, we have $ \mathbb{E}^{y_0}T_{1D}-\mathbb{E}^{y_0}T_{1D}^\chi=\mathcal{T}_{B}(y_0)-\mathcal{T}_{\text{eff}}(y_0)\geq (\mathcal{T}_{B}-\mathcal{T}_{\text{eff}})(y=\delta,L-\delta)=0$ for all $y_0\in[\delta,L-\delta]$.
\end{proof}}
\ifx
(Argument below are for Checking purpose only:) Note that in the searching strip, $Au_m=\min A|u(y)|\geq A^{4/5}u_d/2$.
 Similar to the previous argument, we  consider the events $F_i:=\{\tau^{in}\leq \frac{5L}{A ^{4/5}  u_d/2}\}\cup\{\tau^{in}\geq \frac{5L}{A^{4/5}u_d/2},\, d_\rr(X_{\frac{5L}{A^{4/5}u_d/2}},X_0)\leq L\}$, which are adjustments to definition \eqref{F_i}. Then the probability of $F_i$ can be estimated as follows. First of all,
\begin{align*}
\mathbb{P}\left(\tau^{in}\leq \frac{5L}{A ^{4/5}  u_d/2}\right)
 \leq&\mathbb{P}\left(|Y_{\tau^{in}}|=\bigg|\int_0^{\tau^{in}} V^{(2)} (X_t,Y_t)dt+\sqrt{\nu}\int_0^{\tau^{in}}dB_t\bigg|\geq A^{-1/3},\quad \tau^{in}\leq \frac{5L}{A ^{4/5}  u_d/2}\right)\nb\\
\leq&\mathbb{P}\left(\bigg|\int_0^{\tau^{in}} V^{(2)} (X_t, Y_t)dt\bigg|\geq \frac{1}{2A^{1/3}}\text{ for some }\tau^{in}\leq \frac{5L}{A ^{4/5}  u_d/2}\right)\nb\\
&+\mathbb{P}\left(\sqrt{\nu}|B_{\tau^{in}}|\geq\frac{1}{2A^{1/3}}\text{ for some  }\tau^{in}\leq \frac{5L}{A ^{4/5}  u_d/2}\right)=:P_1+P_2.
%$%\leq 2\mathbb{P}\left(\sqrt{\nu}B_{\frac{2L}{A^{4/5}  u_d/2}}\geq A^{-1/3}\right) \leq 2\frac{\mathbb{E}\left[\left(B_{\frac{2L}{A^{4/5} u_d/2}}\right)^{2}\right]}{\nu^{-1}A^{-2/3}} \leq \frac{8\nu L}{A^{2/15}u_d}.
%\end{align}
%\begin{align}
\end{align*}
The first term $P_1$ can be estimated using the fact that $||V||_\infty\leq C$ as follows:%_{\mathrm{eff}}_{\mathrm{eff}}
\begin{align*}
P_1\leq \mathbb{P}\left(||V||_\infty\frac{5L}{A ^{4/5}  u_d/2}\geq \frac{1}{2A^{1/3}}\right)=0;
\end{align*}
the probability is zero if $A$ is chosen large enough compared to $||V||_\infty,\, L,\, u_d$. %_{\mathrm{eff}}  has probability zero because the time is too short and t
The second term $ P_2$ can be estimated with reflection principle of Brownian motion as follows
\begin{align*}
P_2\leq\mathbb{P}\left(\min_{t}\{t|\sqrt{\nu}B_t=\frac{1}{2}A^{-1/3}\}\leq\frac{5L}{A ^{4/5}  u_d/2}\right)\leq 2\mathbb{P}\left(\sqrt{\nu}B_{\frac{5L}{A ^{4/5}  u_d/2}}\geq \frac{1}{2}A^{-1/3}\right)\leq \frac{80\nu L}{A^{2/15}u_d}.
\end{align*}
As $A\rightarrow \infty$, this approaches $0$.

Then,
\begin{align*}
\mathbb{P}\left(\tau^{in}\geq \frac{5L}{A^{4/5}u_d/2},\, d_\rr(X_{\frac{5L}{A^{4/5}u_d/2}},X_0)\leq L\right)\leq \mathbb{P}\left(\sqrt{\nu}B_{ \frac{5L}{A^{4/5}u_d/2}}\geq  L\right)\leq \frac{10\nu}{A^{4/5} u_dL}.%\tau^{in} MATH
\end{align*}
Hence,
\begin{align*}
\mathbb{P}(F_i)\leq\frac{80\nu L}{A^{2/15}u_d}+\frac{10\nu}{A^{4/5} u_dL}.
\end{align*}Now the same argument as above yields the convergence result
\begin{align*}
0\leq\mathbb{E}^{(x_0,y_0)}[T_{2D}^{A;\chi}]\leq &\mathbb{E}\tau_0+\mathbb{P}\left(F_1^c\right)\frac{5L}{A^{4/5} u_d/2}+\sum_{i=1}^\infty\left(\prod_{j=1}^{i}\mathbb{P}\left(F_j\right)\right)\left(\mathbb{E}(\tau^{out})+\mathbb{P}\left(F_{i+1}^c\right)\frac{5L}{A^{4/5} u_d/2}\right)\\
\leq &C(\nu, L, u_d,\chi,||\varphi||_{C^2},||u'||_\infty,||n||_{W^{2,\infty}})A^{-2/15}.
\end{align*}
As a result, $\lim_{A\rightarrow \infty}\mathbb{E}^{(x_0,y_0)}[T_{2D}^{A;\chi}]=0$.
\fi

\section{Supplementary Materials}
{The datasets supporting the conclusions of this article are included within the article and its additional files.
\begin{itemize}
    \item Additional File 1. Hitting Time Data (CSV 18kb)
    \item Additional File 2. Hitting Angles Data (CSV 188kb)
\end{itemize}}
{\bf Acknowledgement}. \rm
The authors acknowledge partial support of the NSF-DMS grants 1848790, 2006372 and 2006660.
 SH would like to thank Xiangying Huang and Yiyue Zhang for helpful suggestions, and Lihan Wang for pointing out the formula \eqref{formula} to him.
 AK has been partially supported by Simons Fellowship and thanks Andrej Zlatos for stimulating discussions.
We are all grateful to anonymous referees for detailed reports, constructive suggestions, and interesting questions.
%This project grew out of the undergraduate DoMath program that looked at the random search in shear (without chemotaxis)

\end{document}